\documentclass[11pt,a4paper,reqno,oneside]{amsart}

\usepackage[left=0.75in,right=0.75in,bottom=0.75in,top=0.75in]{geometry}

\usepackage{tikz,amssymb,dsfont}
\usepackage{hyperref}
\usepackage{bm}

\DeclareMathOperator{\spec}{spec}
\DeclareMathOperator{\supp}{supp}

\DeclareMathOperator{\R}{\mathbb{R}}
\DeclareMathOperator{\N}{\mathbb{N}}
\DeclareMathOperator{\dS}{dS}
\DeclareMathOperator{\cotan}{cotan}

\newcommand{\pz}{\partial_{\bar{z}}}
\newcommand{\specd}{\spec_\mathrm{disc}}
\newcommand{\spece}{\spec_\mathrm{ess}}

\renewcommand{\Tilde}{\widetilde}

\def \rr {{\mathbb R}}

\def \nn {{\mathbb N}}

\numberwithin{equation}{section}
\numberwithin{figure}{section}
\theoremstyle{definition}

\newtheorem{definition}{Definition}[section]
\newtheorem{remark}[definition]{Remark}
\newtheorem{notation}[definition]{Notation}

\theoremstyle{plain}

\newtheorem{theorem}[definition]{Theorem}
\newtheorem{proposition}[definition]{Proposition}
\newtheorem{lemma}[definition]{Lemma}
\newtheorem{corollary}[definition]{Corollary}

\makeatletter
\title[Strong $\delta$-interactions]{Eigenvalue asymptotics for strong $\delta$-interactions\\
	supported on curves with corners}

\author[B. Benhellal]{Badreddine Benhellal}
\address{Technische Universit\"at Hamburg, Institut f\"ur Mathematik, 
	Am Schwarzenberg-Campus 3, 21073 Hamburg, Germany \& National Higher School of Mathematics, Scientific and Technology Hub of Sidi Abdellah, P.O. Box 75, Algiers, 16093, Algeria}
\email{badreddine.benhellal@tuhh.de and badreddine.benhellal@nhsm.edu.dz}

\author[N. K\"orner]{Noah K\"orner}

\address{Carl von Ossietzky Universit\"at Oldenburg,
	 Fakult\"at V -- Mathematik und Naturwissenschaften, Institut f\"ur Mathematik, 26111 Oldenburg, Germany}
    \email{noah.koerner@uol.de}
    
   \author[K. Pankrashkin]{Konstantin Pankrashkin}
\address{Carl von Ossietzky Universit\"at Oldenburg,
	Fakult\"at V -- Mathematik und Naturwissenschaften, Institut f\"r Mathematik, 26111 Oldenburg, Germany}
\email{konstantin.pankrashkin@uol.de}



\begin{document}

\begin{abstract} Let $\Gamma\subset\mathbb{R}^2$ be a piecewise smooth closed curve with corners. We discuss the asymptotic behavior of the individual eigenvalues of the two-dimensional Schr\"odinger operator $-\Delta-\alpha\delta_\Gamma$ for $\alpha\to\infty$, where $\delta_\Gamma$ is the Dirac $\delta$-distribution supported by $\Gamma$. It is shown that the asymptotics of several first eigenvalues is determined by the corner opening only, while the main term in the asymptotic expansion for the other eigenvalues is the same as for smooth curves. Under an additional assumption on the corners of $\Gamma$ (which is satisfied, in particular, if $\Gamma$ has no acute corners), a more detailed eigenvalue asymptotics is established in terms of a one-dimensional effective operator on the boundary.
\end{abstract}
	
	\maketitle	
	
%	\tableofcontents

\section{Introduction}

\subsection{Motivation} The present work is devoted to the spectral analysis of Schr\"odinger operators with attractive $\delta$-potentials. If $\Sigma\subset\R^n$ is a hypersurface (suitably regular, e.g., Lipschitz, either compact or with an appropriate behavior near boundary or at infinity), the operators of such a type are often formally written as $H_\alpha^\Sigma:=-\Delta-\alpha\delta_\Sigma$, where $-\Delta$ is the usual Laplacian, $\alpha>0$ is a parameter (interpreted as the coupling constant) and $\delta_\Sigma$ is the Dirac $\delta$-distribution supported on $\Sigma$. More rigorously, one defines $H^\Sigma_\alpha$ as the unique self-adjoint operator in $L^2(\R^n)$  generated by the Hermitian sesquilinear form
\[
h^\Sigma_\alpha(u,u):=\int_{\R^n}|\nabla u|^2\mathrm{d}x-\alpha\int_\Sigma |u|^2\dS,
\quad
D(h^\Sigma_\alpha):=H^1(\rr^n),
\]
which is closed and lower semibounded \cite{BEKS} (under suitable geometric assumptions on $\Sigma$). From the physics point of view, such  operators $H^\Sigma_\alpha$ represent an important class of solvable quantum-mechanical models \cite{AGHH}
being the limits of the usual Schr\"odinger operators $-\Delta+V_{\alpha,\Sigma}$ with attractive regular potentials $V_{\alpha,\Sigma}$ strongly localized near $\Sigma$, see e.g. \cite{BEHL} for a rigorous formulation.

It is of natural interest to study the dependence of the spectral and scattering properties of $H^\Sigma_\alpha$ on the interaction support $\Sigma$, which gave
rise to interesting developments, see e.g. the reviews in \cite{E08,galk}. One of the particularly relevant regimes is the case $\alpha\to\infty$ corresponding to strongly attractive $\Sigma$. Elementary considerations show a strong localization of the eigenfunctions of $H^\Sigma_\alpha$ near $\Sigma$, which leads to the standard expectation
that an effective operator on $\Sigma$ may play a key role in the spectral analysis. The first result in this direction was obtained in \cite{EY02} for the case when $n=2$ and $\Sigma$ is a smooth loop: for each $j\in\N$ the $j$-th eigenvalue $E_j$ (if counted in the non-decreasing order with multiplicities taken into account) satisfies
\[
E_j(H^\Gamma_\alpha) = - \frac{\alpha^2}{4} + E_j(T) + \mathcal{O}\Big(\frac{\log \alpha}{\alpha}\Big) \quad \text{ as } \alpha \to \infty,
\]
where $T$ is the effective Schr\"odinger operator in $L^2(\Gamma)$
given by
\[
T:=-\partial^2-\dfrac{k^2}{4},
\]
with $\partial$ being the derivative with respect to the arc-length and $k$ the curvature on $\Gamma$. The analysis of \cite{EY02} used in an essential way both the smoothness and the closedness of $\Sigma$. The later paper \cite{EPan} extended the above result to the case of smooth curves $\Sigma$ with (regular) endpoints: in that case, one has
\[
E_j(H^\Gamma_\alpha) = - \frac{\alpha^2}{4} + E_j(T^D) + \mathcal{O}\Big(\frac{\log \alpha}{\alpha}\Big) \quad \text{ as } \alpha \to \infty,
\]
where $T^D$ is the one-dimensional operator in $L^2(\Gamma)$ given by the same expression as above but with Dirichlet boundary conditions imposed at the endpoints of $\Gamma$. The work \cite{FP20} studied the case of curves with peaks, for which both the eigenvalue asymptotics and the nature of the effective operator turn out to be completely different. Some results are available for the case of smooth $\Sigma$ in higher dimensions, see e.g., \cite{DEKP,EK03}, and several works analyzed very specific non-smooth and non-compact $\Sigma$ related to conical geometries \cite{bel1,LOB,OBP,OBPP}, however, we are not aware of any extension of the asymptotics of individual eigenvalues to curves $\Sigma$ with corners. (However, we mention the very recent work \cite{BCP} showing that the asymptotic of the first eigenvalue is generally different from that in the smooth case.) In fact, the absence of such an extension has been mentioned as one of the principal gaps in the study of $\delta$-potentials for a long time; see, e.g., the review~\cite{E08}. The purpose of the present work is to fill this gap at least partially by considering curves with corners subject to some restrictions on the corner openings.
From the methodological point of view, we are inspired
by the paper \cite{KOBP20} considering Robin Laplacians in a domain with corners, and it was our secondary goal to illustrate the robustness of that machinery.

\subsection{Main results}
We are studying the operators $H^\Gamma_\alpha$ defined as above in two-dimensions ($n=2$)
for injective, closed, piecewise $C^3$-smooth curves $\Gamma$.
More precisely, we assume that one may decompose $\Gamma = \Gamma_1 \cup \dots \cup \Gamma_M $,
where $\Gamma_j $ are $C^3$-smooth regular curves with regular ends such that $\Gamma_{j-1}$ and $\Gamma_j$ meet at the corners $A_j$ at an angle $2 \theta_j$ with $\theta_j \in (0, \frac{\pi}{2}) \cup (\frac{\pi}{2}, \pi)$. This means that each $\Gamma_j$ admits an arc-length parametrization
\[
\gamma_j:[0,l_j]\to\R^2,
\]
where $l_j$ is the length of $\Gamma_j$ and $\gamma_j$ is $C^3$-smooth with $|\gamma_j'(s)|=1$ for all $s$, with 
\[
\gamma_{j-1}(l_{j-1})=\gamma_j(0)=:A_j
\]
(under the convention that $\gamma_0:=\gamma_M$), such that $s\mapsto \gamma_j(s)$ corresponds to the anti-clockwise direction along $\Gamma$ and that for each $j$ the vector $\gamma'_j(0)$ is obtained by the anti-clockwise direction of the vector $-\gamma'_{j-1}(l_{j-1})$ by the angle $2\theta_j$, and that the points $A_j$
are the only intersection points of the curves $\Gamma_j$.  A visualization can be found in Figure \ref{fig: lmao get}. The curvature of $\Gamma_j$ at the point $\gamma_j(s)$ is
\[
k_j(s):=\det\big(\gamma'_j(s),\gamma''_j(s)\big),
\]
and $k_j$ is $C^1$-smooth on $[0,l_j]$ due to the above assumptions. Standard considerations \cite{BEKS} show that $\spece H^\Gamma_\alpha=[0,\infty)$, and we are interested in the discrete eigenvalues.

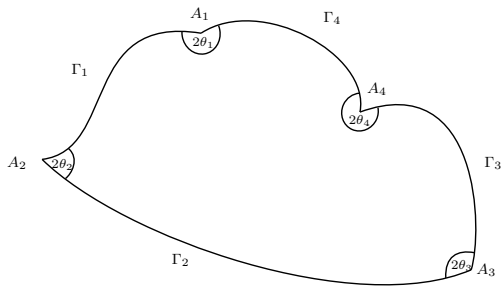
\begin{figure}
  \begin{center}  
\tikzset{every picture/.style={line width=0.75pt}} %set default line width to 0.75pt
\scalebox{0.7}{
\begin{tikzpicture}[x=0.85pt,y=0.85pt,yscale=-1,xscale=1]
%uncomment if require: \path (0,300); %set diagram left start at 0, and has height of 300

%Curve Lines [id:da6697252582185478] 
\draw [line width=0.75]    (90,160) .. controls (137.4,154.2) and (113.8,67.8) .. (190,80) ;
%Curve Lines [id:da9091187292621014] 
\draw [line width=0.75]    (90,160) .. controls (141.4,210.6) and (288.2,261) .. (360,230) ;
%Curve Lines [id:da40181230220956254] 
\draw [line width=0.75]    (190,80) .. controls (227.4,55.4) and (297,93.4) .. (290,130) ;
%Curve Lines [id:da6619535175576159] 
\draw [line width=0.75]    (290,130) .. controls (365.8,102.2) and (367.4,205.8) .. (360,230) ;
%Curve Lines [id:da12450672000519813] 
\draw    (106.67,153.17) .. controls (113.67,159.5) and (108.67,169.17) .. (104.67,172.17) ;
%Curve Lines [id:da07646621086357652] 
\draw    (344,235.17) .. controls (343.33,221.5) and (351.33,217.17) .. (362,219.17) ;
%Shape: Arc [id:dp2085609242308193] 
\draw  [draw opacity=0] (201.07,75.28) .. controls (201.4,76.16) and (201.65,77.08) .. (201.81,78.05) .. controls (203.01,85.29) and (198.68,92.03) .. (192.16,93.11) .. controls (185.64,94.18) and (179.38,89.19) .. (178.19,81.95) .. controls (178.03,80.97) and (177.97,80.01) .. (178,79.06) -- (190,80) -- cycle ; \draw   (201.07,75.28) .. controls (201.4,76.16) and (201.65,77.08) .. (201.81,78.05) .. controls (203.01,85.29) and (198.68,92.03) .. (192.16,93.11) .. controls (185.64,94.18) and (179.38,89.19) .. (178.19,81.95) .. controls (178.03,80.97) and (177.97,80.01) .. (178,79.06) ;  
%Shape: Arc [id:dp09557817661486967] 
\draw  [draw opacity=0] (301.32,127.26) .. controls (301.76,129.64) and (301.5,132.2) .. (300.42,134.62) .. controls (297.71,140.73) and (290.85,143.62) .. (285.09,141.06) .. controls (279.34,138.51) and (276.87,131.49) .. (279.58,125.38) .. controls (281.49,121.07) and (285.46,118.37) .. (289.65,118.04) -- (290,130) -- cycle ; \draw   (301.32,127.26) .. controls (301.76,129.64) and (301.5,132.2) .. (300.42,134.62) .. controls (297.71,140.73) and (290.85,143.62) .. (285.09,141.06) .. controls (279.34,138.51) and (276.87,131.49) .. (279.58,125.38) .. controls (281.49,121.07) and (285.46,118.37) .. (289.65,118.04) ;  

% Text Node
\draw (181.93,62.67) node [anchor=north west][inner sep=0.75pt]  [font=\scriptsize,color={rgb, 255:red, 0; green, 0; blue, 0 }  ,opacity=1 ] [align=left] {$\displaystyle A_{1}$};
% Text Node
\draw (66.67,157) node [anchor=north west][inner sep=0.75pt]  [font=\scriptsize,color={rgb, 255:red, 0; green, 0; blue, 0 }  ,opacity=1 ] [align=left] {$\displaystyle A_{2}$};
% Text Node
\draw (362.07,225.13) node [anchor=north west][inner sep=0.75pt]  [font=\scriptsize,color={rgb, 255:red, 0; green, 0; blue, 0 }  ,opacity=1 ] [align=left] {$\displaystyle A_{3}$};
% Text Node
\draw (293.2,110.27) node [anchor=north west][inner sep=0.75pt]  [font=\scriptsize,color={rgb, 255:red, 0; green, 0; blue, 0 }  ,opacity=1 ] [align=left] {$\displaystyle A_{4}$};
% Text Node
\draw (107.07,98.27) node [anchor=north west][inner sep=0.75pt]  [font=\scriptsize,color={black}  ,opacity=1 ] [align=left] {$\displaystyle \Gamma _{1}$};
% Text Node
\draw (183.33,81) node [anchor=north west][inner sep=0.75pt]  [font=\tiny] [align=left] {$\displaystyle 2\theta_{1}$};
% Text Node
\draw (94.67,158.33) node [anchor=north west][inner sep=0.75pt]  [font=\tiny] [align=left] {$\displaystyle 2\theta_{2}$};
% Text Node
\draw (346,222.33) node [anchor=north west][inner sep=0.75pt]  [font=\tiny] [align=left] {$\displaystyle 2\theta_{3}$};
% Text Node
\draw (281.67,130.67) node [anchor=north west][inner sep=0.75pt]  [font=\tiny] [align=left] {$\displaystyle 2\theta_{4}$};
% Text Node
\draw (367.07,157.27) node [anchor=north west][inner sep=0.75pt]  [font=\scriptsize,color={black}  ,opacity=1 ] [align=left] {$\displaystyle \Gamma_{3}$};
% Text Node
\draw (170.4,218.6) node [anchor=north west][inner sep=0.75pt]  [font=\scriptsize,color={black}  ,opacity=1 ] [align=left] {$\displaystyle \Gamma_{2}$};
% Text Node
\draw (265.73,64.93) node [anchor=north west][inner sep=0.75pt]  [font=\scriptsize,color={black}  ,opacity=1 ] [align=left] {$\displaystyle \Gamma_{4}$};

\end{tikzpicture}
}
  \end{center}
  \caption{A curve with corners}
  \label{fig: lmao get}
\end{figure}

Our first result is that the asymptotics of the first few eigenvalues is only dependent on the angles $\theta_j$. For the rigorous formulation, let us consider the operator $H^\alpha_\theta:=H^{\Gamma_\theta}_\alpha$, where $\Gamma_\theta$ is an infinite corner of angle $2\theta \in (0,2\pi)$ as shown on Figure \ref{fig: angle for introduction}. As discussed in Section~\ref{sec: angle}, for any $\alpha>0$ the essential spectrum of $H^\alpha_\theta$ is $[-\frac{\alpha^2}{4},\infty)$ and its discrete spectrum is finite, so we denote
\[
\kappa(\theta) :=  \text{the number of discrete eigenvalues of } H^\alpha_\theta,
\]
which is independent of $\alpha$ due to the obvious unitary equivalence $H^\alpha_\theta\simeq \alpha^2 H^1_\theta$. This allows to define the following quantities associated with the curve $\Gamma$:
\begin{align*}
	\mathcal{K} &:= \kappa(\theta_1) + \dots + \kappa(\theta_M),\\
	\mathcal{E} &:= \text{the disjoint union of the discrete eigenvalues of } H^1_{\theta_j}, \, j = 1,\dots,M,\\
	\mathcal{E}_j &:= \text{the $j$-th element of } \mathcal{E} \text{ if numbered in non-decreasing order.}
\end{align*}

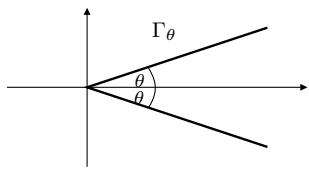
\begin{figure}
	%  \begin{minipage}[c]{6.5cm}
		\begin{tikzpicture}[x=0.75pt,y=0.75pt,yscale=-1,xscale=1]
			%uncomment if require: \path (0,300); %set diagram left start at 0, and has height of 300
			
			%Straight Lines [id:da03068538429055301] 
			\draw [color={black}  ,draw opacity=1 ] [line width=0.90]  (210,110) -- (120,140) -- (210,170) ;
			%Straight Lines [id:da1307866764438127] 
			\draw    (120,180) -- (120,103) ;
			\draw [shift={(120,100)}, rotate = 90] [fill={rgb, 255:red, 0; green, 0; blue, 0 }  ][line width=0.08]  [draw opacity=0] (3.57,-1.72) -- (0,0) -- (3.57,1.72) -- cycle    ;
			%Straight Lines [id:da7663107071187568] 
			\draw    (80,140) -- (227,140) ;
			\draw [shift={(230,140)}, rotate = 180] [fill={rgb, 255:red, 0; green, 0; blue, 0 }  ][line width=0.08]  [draw opacity=0] (3.57,-1.72) -- (0,0) -- (3.57,1.72) -- cycle    ;
			%Curve Lines [id:da790671729394219] 
			\draw    (150.67,130.33) .. controls (155.67,138.5) and (154.67,143.5) .. (150.67,150.33) ;
			
			% Text Node
			\draw (142,140.67) node [anchor=north west][inner sep=0.75pt]  [font=\tiny] [align=left] {$\displaystyle \theta $};
			% Text Node
			\draw (142.33,132) node [anchor=north west][inner sep=0.75pt]  [font=\tiny] [align=left] {$\displaystyle \theta $};
			% Text Node
			\draw (151,106) node [anchor=north west][inner sep=0.75pt]  [font=\scriptsize,color={black}  ,opacity=1 ] [align=left] {$\displaystyle \Gamma_{\theta }$};

		\end{tikzpicture} \caption{The infinite angle $\Gamma_\theta$ } 
		\label{fig: angle for introduction}
	\end{figure}

Our first result on the eigenvalue asymptotics is as follows:

\begin{theorem}\label{theo: corner induced} Let $n\in\N$, then there exists $\alpha_0 >0$ such that for all $\alpha \geq \alpha_0$
	the operator $H^\Gamma_\alpha$ has at least $\mathcal{K}+n$ discrete eigenvalues, and
	for $\alpha\to\infty$ one has:
	\begin{gather*}
		E_j(H^\Gamma_\alpha) = \mathcal{E}_j \alpha^2  + \mathcal{O}(\alpha^{\frac{4}{3}}) \text{ for } j=1,\dots,\mathcal{K},\qquad
		E_{j}(H^\Gamma_\alpha)=-\dfrac{\alpha^2}{4}+o(\alpha^2) \text{ for } j\ge\mathcal{K}+1.
	\end{gather*}	
\end{theorem}

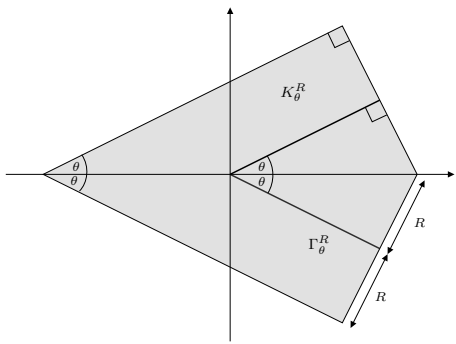
\begin{figure}[b]
\centering
\scalebox{0.7}{
\begin{tikzpicture}[x=1pt,y=1pt,yscale=-1,xscale=1]
%uncomment if require: \path (0,300); %set diagram left start at 0, and has height of 300

%Straight Lines [id:da673219251757111] 
\draw [color={black}  ,draw opacity=1 ]  [line width=0.80] (260,140) -- (340,180) ;
%Shape: Polygon [id:ds7343152412465792] 
\draw  [fill={rgb, 255:red, 155; green, 155; blue, 155 }  ,fill opacity=0.3 ] (320,60) -- (360,140) -- (320,220) -- (160,140) -- cycle ;
%Straight Lines [id:da9641363637549958] 
\draw    (260,230) -- (260,53) ;
\draw [shift={(260,50)}, rotate = 90] [fill={rgb, 255:red, 0; green, 0; blue, 0 }  ][line width=0.08]  [draw opacity=0] (3.57,-1.72) -- (0,0) -- (3.57,1.72) -- cycle    ;
%Straight Lines [id:da9445629497091704] 
\draw    (140,140) -- (377,140) ;
\draw [shift={(380,140)}, rotate = 180] [fill={rgb, 255:red, 0; green, 0; blue, 0 }  ][line width=0.08]  [draw opacity=0] (3.57,-1.72) -- (0,0) -- (3.57,1.72) -- cycle    ;
%Shape: Right Angle [id:dp7325962002512014] 
\draw   (323.77,67.98) -- (315.97,71.67) -- (312.2,63.69) ;
%Shape: Right Angle [id:dp22538521829831037] 
\draw   (343.77,107.98) -- (335.97,111.67) -- (332.2,103.69) ;
%Curve Lines [id:da2181706414990453] 
\draw    (180.6,129.4) .. controls (185,135.8) and (183.8,144.6) .. (179.4,149) ;
%Curve Lines [id:da9936839867064204] 
\draw    (279.8,130.07) .. controls (283.8,136.07) and (283.8,144.07) .. (280.2,150.07) ;
%Straight Lines [id:da7597399178104686] 
\draw    (342.86,185.69) -- (325.74,220.11) ;
\draw [shift={(324.4,222.8)}, rotate = 296.45] [fill={rgb, 255:red, 0; green, 0; blue, 0 }  ][line width=0.08]  [draw opacity=0] (3.57,-1.72) -- (0,0) -- (3.57,1.72) -- cycle    ;
\draw [shift={(344.2,183)}, rotate = 116.45] [fill={rgb, 255:red, 0; green, 0; blue, 0 }  ][line width=0.08]  [draw opacity=0] (3.57,-1.72) -- (0,0) -- (3.57,1.72) -- cycle    ;
%Straight Lines [id:da20899462485635134] 
\draw    (362.66,145.89) -- (345.54,180.31) ;
\draw [shift={(344.2,183)}, rotate = 296.45] [fill={rgb, 255:red, 0; green, 0; blue, 0 }  ][line width=0.08]  [draw opacity=0] (3.57,-1.72) -- (0,0) -- (3.57,1.72) -- cycle    ;
\draw [shift={(364,143.2)}, rotate = 116.45] [fill={rgb, 255:red, 0; green, 0; blue, 0 }  ][line width=0.08]  [draw opacity=0] (3.57,-1.72) -- (0,0) -- (3.57,1.72) -- cycle    ;
%Straight Lines [id:da45607457910559124] 
\draw [line width=0.80]   (260,140) -- (340,100) ;

% Text Node
\draw (336.4,202.6) node [anchor=north west][inner sep=0.75pt]  [font=\tiny] [align=left] {$\displaystyle R$};
% Text Node
\draw (357.3,162.5) node [anchor=north west][inner sep=0.75pt]  [font=\tiny] [align=left] {$\displaystyle R$};
% Text Node
\draw (174.8,132.4) node [anchor=north west][inner sep=0.75pt]  [font=\tiny] [align=left] {$\displaystyle \theta $};
% Text Node
\draw (173.6,140) node [anchor=north west][inner sep=0.75pt]  [font=\tiny] [align=left] {$\displaystyle \theta $};
% Text Node
\draw (274.07,140.67) node [anchor=north west][inner sep=0.75pt]  [font=\tiny] [align=left] {$\displaystyle \theta $};
% Text Node
\draw (274,132.4) node [anchor=north west][inner sep=0.75pt]  [font=\tiny] [align=left] {$\displaystyle \theta $};
% Text Node
\draw (286,90) node [anchor=north west][inner sep=0.75pt]  [font=\scriptsize] [align=left] {$\displaystyle K_{\theta }^{R}$};
% Text Node
\draw (301,172) node [anchor=north west][inner sep=0.75pt]  [font=\scriptsize,color={black}  ,opacity=1 ] [align=left] {$\displaystyle \Gamma_{\theta }^{R}$};

\end{tikzpicture}
}
    \caption{The kite $K^R_\theta$. It has angles $2\theta$ and $\pi-2 \theta$ and the edges non adjacent to the $2 \theta$ angle have length $2R$.} 
\label{fig: introducing Karl}
 
\end{figure}

Our second result concerns an improvement of the result on $E_j(H^\Gamma_\alpha)$ for $j\ge \mathcal{K}+1$
under additional assumptions on the corners. For that, one analyzes the ``finite-volume'' versions of $H^\Gamma_\alpha$ on the kites $K^R_\theta$ depicted in Figure \ref{fig: introducing Karl}. Namely, denote
by $N^R_{\theta, \alpha}$ the self-adjoint operator in $L^2(K^R_\theta)$ generated by the Hermitian
sesquilinear form $n^R_{\theta,\alpha}$ given by
\[
n^R_{\theta,\alpha} (u) = \int_{K^R_\theta} |\nabla u|^2 \mathrm{d}x  - \alpha \int_{\Gamma^R_\theta} |u|^2 \dS, \quad D(n^R_{\theta,\alpha}) = H^1(K^R_\theta),
\quad
\Gamma^R_\theta:=\Gamma_\theta\cap K^R_\theta,
\]
i.e. $N^R_{\theta,\alpha}$ is the Laplacian with Neumann boundary condition on $\partial K^R_\theta$
and the $\delta$-interaction on $\Gamma^R_\theta$, then one easily shows that the first $\kappa(\theta)$ eigenvalues converge for $R\to\infty$ to the corresponding eigenvalues of $H^\alpha_\theta$.
The following condition will play a central role in our analysis: A half-angle $\theta$ is called \emph{non-resonant} if for some $\alpha>0$ (and then for any $\alpha>0$) there exists $C> 0$ such that
$$E_{\kappa(\theta)+1}(N^R_{\theta, \alpha}) \geq - \frac{\alpha^2}{4} + \frac{C}{R^2}$$
holds for $R \to \infty$. 

\begin{theorem}\label{theo: side induced}
Assume that all corners of $\Gamma$ are non-resonant (i.e. all half-angles $\theta_j$ are non-resonant). 
Then for any $n \in \mathbb{N}$ there exists $\alpha_0 >0$ such that for all $\alpha \geq \alpha_0$ the operator $H^\Gamma_\alpha$ has at least $\mathcal{K}+n$ discrete eigenvalues, and
\[
E_{\mathcal{K}+n}(H^\Gamma_\alpha) = - \frac{\alpha^2}{4} + E_n \bigg(\bigoplus_{j=1}^M\Big(D_j - \frac{k_j^2}{4}\Big)\bigg) + \mathcal{O}\Big(\frac{\log \alpha}{\sqrt\alpha}\Big) \quad \text{ as } \alpha \to \infty,
\]
where $D_j$ is the Dirichlet-Laplacian on the interval $(0,l_j)$.
\end{theorem}

Using some elementary observations on the non-resonance condition (see Corollary \ref{corol-nonres} below)
we obtain a more straightforward version:

\begin{corollary}\label{corol3}
	Assume that $\Gamma$ has only right and obtuse corners, i.e., that
	\[
	\theta_j\in \Big[ \frac{\pi}{4}, \frac{\pi}{2}\Big) \cup \Big(\frac{\pi}{2}, \frac{3\pi}{4}\Big] \text{ for all }j=1,\dots,M,
	\]
	then $\mathcal{K}=M$, and for any $n \in \mathbb{N}$ there exists $\alpha_0 >0$ such that for all $\alpha \geq \alpha_0$ the operator $H^\Gamma_\alpha$ has at least $M+n$ discrete eigenvalues, and
	\[
	E_{M+n}(H^\Gamma_\alpha) = - \frac{\alpha^2}{4} + E_n \bigg(\bigoplus_{j=1}^M\Big(D_j - \frac{k_j^2}{4}\Big)\bigg) + \mathcal{O}\Big(\frac{\log \alpha}{\sqrt\alpha}\Big) \quad \text{ as } \alpha \to \infty,
	\]
	where $D_j$ is the Dirichlet-Laplacian in $L^2(0,l_j)$.
\end{corollary}

\begin{remark}
It would be interesting to understand if there are half-angles violating the non-resonance condition (i.e., if there are ``resonant'' half-angles): our analysis does not give any indication in this direction. For a similar problem involving Robin Laplacians, the existence of resonant angles was shown in \cite[Sec.~6.1]{KOBP20}. Another (positive) difference between the present work and \cite{KOBP20} is that our study does not require additional conditions on the curvatures (due to some specific features of the $\delta$-potentials), while the analysis~\cite{KOBP20} for the Robin case required the constancy of the curvatures.
\end{remark}

The above main results have a direct application to the analysis of recently introduced Schr\"odinger operators with oblique transmission conditions~\cite{bhs, BCP}. Namely, let $(\nu_1,\nu_2)$ be the outer unit normal on $\Gamma$ and consider the complex valued function $n:=\nu_1+i\nu_2$ defined on $\Gamma$.
For $\beta\in\R$ denote by $Q^\Gamma_\beta$ the operator in $L^2(\R^2)$ acting as
\[
Q^\Gamma_\beta f:= -\Delta f \text{ in }\mathcal{D}'(\R^2\setminus\Gamma)
\]
on the domain 
\begin{equation}
	\label{def:schrodinger}
	\begin{aligned}
		D(Q^\Gamma_\beta):=\Big\{f\simeq &(f_+,f_-)\in H^{\frac{1}{2}}(\Omega_+)\oplus H^{\frac{1}{2}}(\Omega_-):\\
		&\ \Delta f_\pm\in L^2(\Omega_\pm),\ (\pz f_+,\pz f_-)\in H^1(\R^2),\\
		&\  n(f_+- f_- ) +\beta(\pz f_+ +\pz f_-)=0 \text{ on }\Gamma
		\Big\}.
	\end{aligned}	
\end{equation}
where $\Omega_+$, respectively $\Omega_-$, stands for the interior of $\Gamma$, respectively the exterior of $\Gamma$, $f_\pm$ denotes the restriction of $f$ on $\Omega_\pm$,
and the expression
\[
\pz:=\frac{1}{2}(\partial_1+i\partial_2)
\]
is known as the Wirtinger derivative. As discussed in \cite{bhs,BCP}, the operator $Q^\Gamma_\beta$
is self-adjoint and appears as the non-relativistic limit of Dirac operators with $\delta$-type potentials supported on $\Gamma$. It holds $\spece Q^\Gamma_\beta=[0,\infty)$,
and for $\beta<0$ one has an infinite discrete spectrum accumulating at $-\infty$ only, so we can
enumerate all eigenvalues $\Tilde E_n(Q^\Gamma_\beta)<0$ in the \emph{non-increasing} order with multiplicities
taken into account. The asymptotic behavior of $\Tilde E_n(Q^\Gamma_\beta)$ for $\beta\to 0^-$ turns out to be closely related to the analysis of $H^\Gamma_\alpha$ with $\alpha\to\infty$, which was first observed in \cite{bhs} and then formalized in \cite[Cor.~8]{BCP} as follows:
\begin{proposition}\label{prop-transm}
	If for some $b>0$ and $j\in\N$ one has
	$E_j(H^\Gamma_\alpha)=-b\alpha^2+o(\alpha^2) \text{ for }\alpha\to\infty$,
	then
	\[
	\Tilde E_j(Q^\Gamma_\beta)=-\frac{1}{b\beta^2}+o\Big(\dfrac{1}{\beta^2}\Big) \text{ for }\beta\to 0^-.
	\]
\end{proposition}
A direct applications of Theorem \ref{theo: corner induced} and Corollary~\ref{corol3} leads to the following observation extending the earlier analysis of $\Tilde E_1$ in \cite{BCP}:
	\begin{corollary}
		As $\beta\to0^-$ one has
		\[
		 \Tilde E_j(Q^\Gamma_\beta)=\dfrac{1}{\mathcal{E}_j\beta^2}+o\Big(\dfrac{1}{\beta^2}\Big) \text{ for } j\in\{1,\dots,\mathcal{K}\},\qquad
		 	\Tilde E_{j}(Q^\Gamma_\beta)=-\dfrac{4}{\beta^2}+o\Big(\dfrac{1}{\beta^2}\Big)
		 	\text{ for } j\ge\mathcal{K}+1,
		\]
		while $\mathcal{K}=M$ if all corners of $\Gamma$ are right or obtuse.
	\end{corollary}

\subsection{Structure of the paper}

In Section \ref{sec-prelim} we set some notation and recall the main tools used throughout the text, which includes the min-max principle (with several technical reformulations), distances between subspaces,  the IMS localization formula for $\delta$-interactions, and the spectral analysis of several one-dimensional operators ($\delta$-interactions on bounded intervals). Section \ref{sec: angle} summarizes known facts about the above operator $H^\alpha_\theta$, i.e. of $\delta$-potentials supported on broken lines. In particular, we prove an Agmon-type decay estimate for the eigenfunctions corresponding to discrete eigenvalues. In Section~\ref{sec: Karl}, we consider the truncations of $H^\alpha_\theta$ on the kites $K^R_\theta$ corresponding to Dirichlet/Neumann boundary conditions on $\partial K^R_\theta$, and we establish several estimates for the eigenvalues and the eigenfunctions in terms of $\alpha, R$, and $\theta$.
In particular, we show that the non-resonance condition is satisfied if $\Gamma_\theta$ forms a right or obtuse corner. The analysis is mostly by applying suitable truncations and decompositions, combined
with the min-max principle and the analysis of one-dimensional operators from Section \ref{sec-prelim}.
A curvilinear version of $K^R_\theta$ and associated operators are considered in Section \ref{sec: V delta}. Using a suitable deformation map (i.e., by mapping the curvilinear version onto the straight one), we show that most eigenvalue results of Section \ref{sec: Karl} can be transferred to the curved kites 
in suitable asymptotic regimes. Section \ref{sec: W delta} is devoted to the spectral analysis of $\delta$-interactions in thin tubes constructed around curved open arcs. This is mainly achieved using
the passage to tubular coordinates and asymptotic estimates, reducing the analysis to operators with separated variables. In Section \ref{sec: Endzeit}, we combine all the findings of the previous sections to prove the two main results. We first introduce a special decomposition of $\R^2$ into small neighborhoods of corners $A_j$ and ``edges'' $\Gamma_j$, so that the restriction of $H^\Gamma_\alpha$ on each piece is covered by the analysis of the preceding sections. Using the standard Dirichlet-Neumann bracketing, we then complete the last steps of the proof for Theorem \ref{theo: corner induced}
in Proposition \ref{lem: corner induced EV}. The last subsection \ref{ss73} is devoted to the proof 
of Theorem \ref{theo: side induced}, and it is explicitly based on spectral estimates requiring the non-resonance condition. While the upper bound is again deduced by rather elementary Dirichlet-Neumann bracketing-type arguments, the lower bound represents the most demanding part of the work, and it combines
the min-max principle with suitable distance estimates for spectral subspaces and subspaces spanned by truncated eigenfunctions.

The overall proof structure is a quite straightforward adaptation of the scheme proposed in \cite{KOBP20} for the analysis of Robin Laplacians in curvilinear polygons. However, the implementation of each proof step required significant technical efforts, as one needed to recognize, rigorously define, and then analyze in detail various analogs of the intermediate objects arising in the Robin case, and none of these objects have appeared previously in the literature on $\delta$-potentials, and a large portion of the basic theory had to be thoroughly reworked.

% \subsection{Acknowledgments}
% All authors were partially supported by the Deutsche For\-schungsgemein\-schaft (German Research Foundation, DFG), project 491606144.

\section{Preliminaries}\label{sec-prelim}
\subsection{Notation}

Let $\mathcal{H}$ be an infinite-dimensional Hilbert space. A sesquilinear form $a: D(a) \times D(a) \to \mathbb{C}$ will always be referred to by a lowercase letter. For brevity, we write $a(u):= a(u,u)$ for any $u \in D(a) \subseteq \mathcal{H}$.
The linear operator generated by a closed, symmetric sesquilinear form $a$ is denoted by the corresponding uppercase letter $A$. Explicitly, one has
\[
a(u,v)  = \langle u,Av \rangle_{\mathcal{H}} \quad \text{ for all }\, u \in D(a)\,  \text{ and }\, v \in D(A).
\]
Using the Min-Max principle, we define the $n$-th Rayleigh quotient $\Lambda_n(A)$ by
\[
\Lambda_n(A) = \inf_{ \underset{\dim V = n}{V \subset D(a)}} \sup_{\underset{u \neq 0}{u \in V}} \frac{a(u,u)}{\| u \|^2}.
\]
 Moreover, we set 
 \[
\Sigma(A) := 
\begin{cases}
\inf \spece(A), & \text{if } \spece(A) \neq \emptyset, \\
+\infty, & \text{otherwise},
\end{cases}
\]
where $\spec$, $\spece$, and $\specd$ denote, respectively, the spectrum, the essential spectrum, and the discrete spectrum of the operator $A$. We also denote by $E_n(A)$ the $n$th eigenvalue of $A$, whenever it exists. Finally, if $A$ and $B$ are unitarily equivalent to each other, then we write $A \cong B$.

\subsection{Comparing operators and distance between closed subspaces}

Let us first recall some useful lemmas for comparing operators. The first one is very standard and follows directly from the Min-Max principle.
%normal comparison
\begin{lemma}\label{lem: comparing operators}
    Given two lower semibounded operators $A_1$ and $A_2$ in infinite-dimensional Hilbert spaces 
    $\mathcal{H}_1$ and $\mathcal{H}_1$, respectively. Suppose that there exists a linear map $J: D(a_1) \to D(a_2)$ such that   $ \|Ju\|_{\mathcal{H}_2} = \|u\|_{\mathcal{H}_1} $ and $a_2(Ju) \leq a_1(u)$ hold for all $u \in D(a_1)$.  Then, $ \Lambda_n(A_2) \leq \Lambda_n(A_1) $ for any $n \in \nn$.
\end{lemma}
\begin{proof}Since $\dim J(L) = \dim L$ holds for any finite-dimensional subspace $L \subset D(a_1)$, it follows that
    \begin{align*}
        \Lambda_n(A_2) &= \inf_{\overset{L_2 \subset D(a_2)}{\dim L_2 = n}} \sup_{ \overset{u \in L_2}{u \neq 0}} \frac{a_2(u)}{\| u\|^2_{\mathcal{H}_2}} \leq 
        \inf_{\overset{L_1 \subset D(a_1)}{\dim L_1 = n}} \sup_{ \overset{u \in J(L_1)}{u \neq 0}} \frac{a_2(u)}{\|u\| ^2_{\mathcal{H}_2}} \\
        &= \inf_{\overset{L_1 \subset D(a_1)}{\dim L_1 = n}} \sup_{ \overset{v \in L_1}{v \neq 0}} \frac{a_2(Jv)}{\|Jv\|^2_{\mathcal{H}_2}} 
        \leq \inf_{\overset{L_1 \subset D(a_1)}{\dim L_1 = n}} \sup_{ \overset{v \in L_1}{v \neq 0}} \frac{a_1(v)}{\| v \|^2_{\mathcal{H}_1}} = \Lambda_n(A_1). \qedhere
    \end{align*}
\end{proof}
 We will need a less standard approach to comparing eigenvalues of operators acting on different spaces
 motivated by estimates in~\cite{EP05,Post05}. The following lemma is a slightly adapted version of the approach as presented in~\cite[Prop.~6]{FP20}.
%fancy comparison
\begin{proposition}\label{lem: comparing operators but fancy}
Let $ \mathcal{H}$ and $ \mathcal{H}'$ be infinite-dimensional Hilbert spaces, and let $B$ and $B'$ be self-adjoint and semibounded from below operators on $\mathcal{H}$ and $\mathcal{H}'$, respectively.  Assume there exists a linear map $ J: D(B)\to D(B') $ and $\varepsilon_1,\varepsilon_2 >0$ such that for a fixed $n \in \mathbb{N}$ the following hold:
\begin{align*}
 \varepsilon_1  &< \frac{1}{\Lambda_n(B) +1 + c_0},\\
\|u\|^2_\mathcal{H} - \|Ju\|^2_{\mathcal{H}'} &\leq \varepsilon_1 (b(u) + \| u \|^2_\mathcal{H}(1 + c_0)), \\
 b'(Ju) - b(u) &\leq \varepsilon_2 (b(u) + \| u \|^2_\mathcal{H}(1+c_0)).
 \end{align*}
Then, we have
\[
\Lambda_n(B') \leq \Lambda_n(B) + \frac{(\Lambda_n(B)\varepsilon_1 + \varepsilon_2)(\Lambda_n(B) +1 +c_0)}{1-\varepsilon_1(\Lambda_n(B) + 1+c_0)}.
\]
\end{proposition}

Finally, we recall the notion of the distance between two closed subspaces following \cite{HS84}.
\begin{definition}\label{def: distance between subspaces}
Let $L_1$ and $L_2$ be closed subspaces of a Hilbert space $\mathcal{H}$, and let $P_1$ and $P_2$ be the orthogonal projectors in $\mathcal{H}$ onto $L_1$ and $L_2$ respectively. The {\bf distance} $d(L_1,L_2)$ between $L_1$ and $L_2$ is defined by
\[ d(L_1,L_2) := \sup_{0 \neq x \in L_1} \frac{\| x-P_1x\|}{ \|x\|} \equiv \| P_1 - P_2P_1\| \equiv \| P_1-P_1P_2\|.
\]
\end{definition}
While the distance is not symmetric, it satisfies the triangular inequality: for any closed subspaces $L_1$, $L_2$, and $L_3$, we have
\begin{align}\label{lem: triangular inequality distance subspaces}
 d(L_1,L_3) \leq d(L_1,L_2) + d(L_2,L_3)
\end{align}
The following estimate for the distance between subspaces will be used, see \cite[Prop.~2.5]{HS84}.
%distance between closed subspaces
\begin{proposition}
\label{lem: distance between subspaces}
Let $A$ be a self-adjoint operator in a Hilbert space $\mathcal{H}$. Given $n\in\nn$, let  $\mu_1, \dots,  \mu_n \in \rr$ be contained in a compact interval $I$ and $\psi_1, \dots,\psi_n \in D(A)$ be linearly independent vectors. Set
\begin{align*}
    \varepsilon &:= \max_{j=1, \dots,n} \|(A-\mu_j) \psi_j\| , \quad \eta := \frac{1}{2} \text{dist}(I, (\spec A) \setminus I),\\
    \lambda& := \text{the smallest eigenvalue of the Gram matrix } (\langle \psi_i, \psi_j \rangle)_{i,j=1,\dots,n}.
\end{align*} 
Consider the subspaces 
\begin{align*}
L_1 := \text{span} \{ \psi_1, \dots,\psi_n \} \quad\text{and}\quad L_2 := \text{the spectral subspace associated with A and I}.
\end{align*} 
If $\eta >0$ holds, then one has
 \[
 d(L_1,L_2) \leq \frac{\varepsilon}{\eta} \sqrt{\frac{n}{\lambda}}.
 \]
 \end{proposition}

\subsection{IMS localization formula and scaling for $\delta$-interactions}

While analyzing spectral properties of a Laplacian on a subset of $\rr$ or $\rr^2$ with a $\delta$-interaction supported at a point or a curve, it is convenient to rescale the interaction parameter and consider the unitary equivalent operator with a fixed parameter that naturally arises from this scaling. We summarize these in the following two lemmas. For the purposes of the present work, a Lipschitz hypersurface $\Sigma\subset\R^n$ is called regular if it can be extended to a Lipschitz hypersurface $\Sigma'$ with $\overline{\Sigma}\subset\Sigma'$.

\begin{lemma}\label{lem: scaling anything}
Let $\Omega \subseteq \rr^n$ be a domain with Lipschitz boundary $\partial \Omega$. For a hypersurface $\Gamma_D \subseteq \partial \Omega$  and a regular hypersurface $\Gamma \subset \overline{\Omega}$, we define for $\alpha\geq 0$ the operator $Q^{\Omega,\Gamma}_\alpha$ generated by
\[ q^{\Omega,\Gamma}_\alpha (u) =\int_{\Omega} |\nabla u|^2 \mathrm{d}x - \alpha \int_{\Gamma} |u|^2 \dS, \quad D(q^{\Omega, \Gamma}_\alpha) = \{u \in H^1(\Omega) :  u|_{\Gamma_D} = 0 \}.
\]
Then $Q^{\Omega,\Gamma}_\alpha$  is  unitarily equivalent to $\alpha^2 Q^{\alpha \Omega, \alpha \Gamma}_1$.
\end{lemma}
\begin{proof}Consider the map
     \[
     \Phi: L^2( \alpha\Omega) \rightarrow L^2( \Omega), \quad u(x_1, \dots,x_n) \mapsto  \alpha^{\frac{n}{2}} \cdot u(\alpha x_1, \dots, \alpha x_n).
     \]
    One easily checks that $\Phi$ is unitary, its inverse is given by $ \Phi^{-1}(u(x)) = \alpha^{-\frac{n}{2}} u(\alpha^{-1}x)$, and it holds that
    \[
    \|\Phi u\|^2_{L^2(\Omega)} = \int_{ \Omega} \alpha^n |u(\alpha x)|^2 \mathrm{d}x = \int_{\alpha\Omega}\alpha^{n-n}  |u(y)|^2 \mathrm{d}y = \| u \|^2_{L^2(\alpha \Omega)}.
    \]
    We can then write
    \begin{align*}
        q^{\Omega,\Gamma}_\alpha (\Phi u)& = \alpha^{n+2} \int_{\Omega} |\nabla u(\alpha x) |^2 \mathrm{d}x - \alpha^{n+1} \int_\Gamma |u(\alpha x)|^2 \dS \\
        &= \alpha^{n+2-n} \int_{\alpha \Omega} |u|^2 \mathrm{d}x - \alpha^{n+1-(n-1)} \int_{\alpha \Gamma} |u|^2 \dS = \alpha^2 q^{\alpha\Omega, \alpha\Gamma}_1 (u).
    \end{align*}
   Since $\Phi(H^1(\alpha \Omega)) = H^1(\Omega)$, we conclude that $\Phi(D(q^{\alpha \Omega, \alpha \Gamma}_1)) = D(q_\alpha^{\Omega,\Gamma})$. Therefore,  $Q^{\Omega,\Gamma}_\alpha \cong \alpha^2 Q^{\alpha \Omega, \alpha \Gamma}_1$ and the lemma is proved.
\end{proof}
\begin{remark}\label{Transmission condition} The operator $Q^{\Omega,\Gamma}_\alpha$ of Lemma \ref{lem: scaling anything} is often formally written as  $Q^{\Omega,\Gamma}_\alpha= -\Delta -\alpha\delta_\Gamma$. If one denotes by $\Omega_+$ and $\Omega_-$ the interior and exterior domains with respect to $\Gamma$ in $\Omega$, i.e., $\partial\Omega_+=\Gamma$, let $u_\pm$ be the restriction of  functions $u\in L^2(\Omega)$ in $\Omega_\pm$. Then, for $u=(u_+,u_-)\in D(Q^{\Omega,\Gamma}_\alpha)$ the Dirac distribution $\delta_\Gamma$ gives rise to the transmission condition on $\Gamma$:
\[
\alpha u_\pm= \partial_{\nu}u_+ -\partial_{\nu}u_-  \quad\text{with }\, u_+=u_-,
\]
where $\partial_{\nu}u_+ -\partial_{\nu}u_-$ is the jump of the normal derivative with respect to the normal vector $\nu$ pointing outward of $\Omega_+$.     
\end{remark}

Throughout this paper, we frequently employ the IMS localization formula. To avoid redundancy, we present the details of its application in our specific setting only once.
%IMS partition
\begin{lemma}
\label{lem: IMS partition}
Let $\Omega \subseteq \rr^n$ be a domain with Lipschitz boundary $\partial\Omega$. Let $\Gamma_D \subseteq \partial \Omega$ be a hypersurface and $\Gamma \subset \overline{\Omega}$ a regular hypersusrface.
Consider the sesquilinear form
 \[
 q(u) = \int_\Omega |\nabla u|^2 \mathrm{d}x - \alpha \int_{\Gamma} |u|^2 \dS, \quad D(q) = \{u \in H^1(\Omega) :  u|_{\Gamma_D} = 0 \}.
 \]
Let $0<a<b$ and $R>0$, and let $\chi_0, \chi_1\in C^\infty(\rr_+)$ be such that 
    \begin{align*}  \chi_0(t) = \begin{cases}
        1, & t\in[0,a], \\
        0, & t \in [b,\infty),
    \end{cases} \quad\text{and}\quad \chi_0^2+\chi_1^2 \equiv 1, 
    \end{align*}
    and set $\chi_j^R(x) := \chi_j(|x|/R)$ for $j\in\{0,1\}$. Then for any $u,v\in D(q)$ we have
\[
q(u,v) = q(\chi_0^R u,\chi_0^R v) + q(\chi^R_1 u, \chi^R_1 v) - \int_{\Omega} \overline{u}v( |\nabla \chi_0^R|^2 + |\nabla \chi_1^R|^2)\mathrm{d}x.
\]
\end{lemma}
\begin{proof} For $j \in \{0,1 \}$ a simple computation gives
    \begin{align*} 
    \langle \nabla(\chi_j^R u), \nabla(\chi_j^R v) \rangle =& (\chi_j^R)^2 \langle \nabla u, \nabla v \rangle + \overline{u}v |\nabla \chi_j^R|^2 + \frac{\overline{u}}{2} \langle \nabla((\chi_j^R)^2), \nabla v \rangle +\frac{{v}}{2} \langle \nabla u  , \nabla((\chi_j^R)^2) \rangle .
    \end{align*}
   Observe that $\nabla((\chi_0^R)^2 + (\chi_1^R)^2 ) \equiv 0$ because $(\chi_0^R)^2 + (\chi_1^R)^2 \equiv1 $. Using this, it follows that
    \begin{align*}
        q(\chi_0^R u,\chi_0^R v) + &q(\chi^R_1 u, \chi^R_1 v) = \int_{\Omega}\langle \nabla  u, \nabla v \rangle \mathrm{d}x - \alpha \int_\Gamma ((\chi_0^R)^2 + (\chi_1^R)^2) \overline{u}v \dS\\
        &+  \int_{\Omega} \overline{u}v( |\nabla \chi_0^R|^2 + |\nabla\chi_1^R|^2) \mathrm{d}x 
        + \frac{1}{2} \int_{\Omega}\Big( \overline{u} \langle \nabla((\chi_0^R)^2 + (\chi_0^R)^2 ), \nabla v \rangle \\
        &+ v \langle \nabla u, \nabla((\chi_0^R)^2 + (\chi_0^R)^2 )  \rangle \Big)\mathrm{d}x \\
        =&\, q(u,v)+ \int_{\Omega} \overline{u}v( |\nabla \chi_0^R |^2 + |\nabla\chi_1^R|^2) \mathrm{d}x.\qedhere
    \end{align*}
\end{proof}

\subsection{Laplacian with a point interaction}
The model operator of a Laplacian with a point interaction defined in an interval with various boundary conditions will frequently appear in the analysis of spectral properties of the operators discussed in Sections \ref{sec: angle}–\ref{sec: W delta}. In particular, the asymptotic behavior of their first and second eigenvalues will be needed. In this subsection. We gather the necessary results. To begin, let us fix some definitions and notations that will be used throughout.
\begin{definition}\label{def t X}Given $\alpha\geq 0$, we define the sesquilinear form of the Laplacian on an interval of length $2L$ by
     \[
     t^X_{L,\alpha} (u) = \int_{-L}^L |\nabla u|^2 \mathrm{d}x - \alpha |u(0)|^2 ,
     \]
    where $X \in \{D, N ,ND\}$ corresponds to the Neumann/Dirichlet/Neumann-Dirichlet boundary conditions at $-L$ and $L$ with domains
    \[
    D(t^D_{L,\alpha}) = H^1_0(-L,L), \,\,\, D(t^N_{L,\alpha}) =H^1(-L,L), \,\,\, 
    D(t^{ND}_{L,\alpha}) = \{ u \in H^1(-L,L): u(L) = 0 \}.
    \]
    Additionally, given $\beta\geq0$,  we let $t^\beta_{L,\alpha}$ be the sesquilinear form with mixed boundary conditions at $-L$ and $L$ defined by
     \[
     t^\beta_{L,\alpha} (u) = \int_{-L}^L |\nabla u|^2 \mathrm{d}x - \alpha |u(0)|^2 - \beta (|u(-L)|^2 + |u(L)|^2), \quad  D(t^\beta_{L,\alpha}) = H^1(-L,L).
     \]
    We abbreviate
    \[
    t^X_{L}:=t^X_{L,\alpha},
    \]
    and from now on, the use of the expression $t^X_{L,\alpha}$ implies the assumption $\alpha > 0$.
\end{definition}

We start with the first two eigenvalues of $T^N_{L,\alpha}$ and $T^\beta_{L,\alpha}$.
\begin{proposition}\label{lem: ev of t,N,L,alpha}%\label{lem: ev of T,beta,L}
There exists $c>0$ such that as $L\alpha \to \infty$, $\alpha \to \infty$, and $L \to 0^+ $, the following hold:
\begin{itemize}
\item[(i)] $-\frac{1}{4} \alpha^2 < E_1(T^D_{L,\alpha}) <  -\frac{\alpha^2}{4} + \mathcal{O}(\alpha^2 e^{-\frac{1}{2} L \alpha } )$. 
    \item[(ii)] $ - \frac{\alpha^2}{4} + \mathcal{O}(\alpha^2 e^{-\frac{1}{2}L \alpha}) <   E_1(T^N_{L,\alpha}) < -\frac{\alpha^2}{4}  $.
    \item[(iii)] $ - \frac{\alpha^2}{4} + \mathcal{O}(\alpha^2 e^{-\frac{1}{2}L \alpha}) <   E_1(T^\beta_{L,\alpha}) < -\frac{\alpha^2}{4}  $.
    \item[(iv)] $E_2(T^\beta_{L,\alpha}) > \frac{c}{L^2} $.
    \item[(v)] $E_2(T^N_{L,\alpha}) > \left(\frac{\pi}{2L}\right)^2 $.
    \item[(vi)] $E_1(T^{ND}_{L,\alpha}) > -\frac{\alpha^2}{4} + \mathcal{O} (\alpha^2e^{-\frac{1}{2}L \alpha})$.
\end{itemize}
\end{proposition}
\begin{proof} Assertion (i) is proved in \cite[Proposition 2.4]{EY02}.   The assertions (ii) and (iii) have been proved in \cite[Proposition 2.5]{EY02}. In there, it was also shown that $E_1(T^X_{L,\alpha})$ is the unique negative eigenvalue of $T^X_{L,\alpha}$, $X=N, \beta$. From this and Remark \ref{Transmission condition} it follows that the second eigenvalue can be written as $E_2(T^X_{L,\alpha}):= \lambda^2$ with $\lambda \geq 0$, and its associated eigenfunction $u = (u_+,u_-)$ satisfies the equation $ -u^{\prime\prime}(x) = \lambda^2u(x)$ for $ x \in (-L,L)\setminus\{0\}$, and fulfills the transmission and boundary conditions
\begin{align}\label{eq: T beta L alpha}
 u_-'(0^-) - u_+'(0^+) = \alpha u(0),\quad    u_+(L) = \beta u_+(L) \quad    u_-'(-L) = - \beta u_-(-L),
\end{align}
where the case $\beta = 0$ corresponds to the Neumann boundary condition. Since $u_+(0^+)=u_-(0^-)$ it follows that $u_\pm = a \cos(\lambda x) + b_\pm \sin(\lambda x)$ for some $a,b_\pm \in \mathbb{C}$. Note that $\lambda \neq 0$ because $u$ cannot vanish identically on $(-L,L)$ due to the transmission condition.

Observe that $T^X_{L,\alpha}$ commutes with the parity operator $v(x)\mapsto v(-x)$. Thus, $u_-(x)=u_+(-x)$ for $x\in(-L,0]$, and therefore $u_-(x)=u_+(-x)  = a \cos(\lambda x) - b_+ \sin(\lambda x)$, which implies that $b_-=-b_+:=b$, and thus $u_\pm = a \cos(\lambda x) \pm b \sin(\lambda x)$.
On one hand, the transmission condition at $0$ gives $\alpha a \equiv\alpha u(0)= u_-'(0) - u_+(0)'  = -2b\lambda$.
On the other hand, the Neumann boundary conditions yield
\begin{align*}
u_+'(L) &= (-a\lambda \sin(\lambda L) +  b \lambda \cos(\lambda L)) = \beta (a \cos(\lambda L) + b (\sin(\lambda L))) = \beta u_+(L),
\end{align*}
and the Robin boundary condition gives 
\begin{align*}
u_-'(-L) &= (-a\lambda \sin(-\lambda L) -  b \lambda \cos(-\lambda L)) \\
&= -\beta (a \cos(-\lambda L) - b (\sin(-\lambda L))) = -\beta u_-(-L).
\end{align*}
The above conditions lead to the system 
\[
A(\lambda)\begin{pmatrix}
    a\\b 
\end{pmatrix}:=\begin{pmatrix}
    \alpha & 2\lambda\\ 
    \beta \cos(\lambda L) + \lambda \sin(\lambda L) 
    & \beta \sin(\lambda L) - \lambda \cos(\lambda L)
\end{pmatrix}  \cdot \begin{pmatrix}
    a\\b 
\end{pmatrix} = \begin{pmatrix}
    0\\0
\end{pmatrix},
\]
and we compute
\begin{align*}
 \det(A(\lambda)) &= \sin(\lambda L ) (\alpha \beta - 2 \lambda^2) + \cos(\lambda L ) (- \lambda \alpha - 2 \lambda \beta)\\
&= (\alpha \beta - 2 \lambda^2)(\lambda L + \mathcal{O}((\lambda L)^3)) + (- \lambda \alpha - 2 \lambda \beta) (1 + \mathcal{O}((\lambda L)^2))\\
&= \lambda L (\alpha(\beta - \frac{1}{L}) - 2 \lambda^2 - \beta) + \mathcal{O}((\lambda L)^2).
\end{align*}
Let us first assume that $\beta\neq0$. Then, for $L < \beta^{-1}$ there exists $c_0>0$ such that for all $\lambda L < c_0$ we have $\det(A(\lambda))< 0$. From this, we see that if $\lambda \in [0, \frac{c_0}{L})$ then  $\det(A(\lambda))< 0$, which implies that there are no solutions to the matrix equation above, and this would imply that $u=0$. Therefore, we get 
\[
E_2(T^\beta_{L,\alpha})= \lambda^2 \geq \left(\frac{c_0}{L}\right)^2,
\]
and this proves (iv). On the contrary, if $\beta = 0$ then
\[
\det(A(\lambda)) = -2\lambda^2 \sin(\lambda L) - \lambda \alpha \cos(\lambda L) = -\lambda( 2\lambda \sin(\lambda L) + \alpha \cos(\lambda L) ),
\]
and we easily see that $\det(A(\lambda))<0$ for $\lambda \in [0,\frac{\pi}{2L}]$. Then, the same arguments as before give the desired lower bound for $E_2(T^N_{L,\alpha})$ and complete the proof of (v). 

Finally, using (v) and the Dirichlet-Neumann bracketing, we obtain
\[
 E_1(T^{ND}_{L,\alpha}) \geq E_1(T^{N}_{L,\alpha})   > -\frac{\alpha^2}{4} +\mathcal{O}( \alpha^2e^{-\frac{1}{2}L \alpha}).
    \]
    which yields (vi), and the proposition is proved.
\end{proof}
The following assertion is elementary:
\begin{lemma}\label{lem: ev of t,ND,L}  $E_1(T^{ND}_{L}) = \frac{\pi^2}{16 L^2}$.
\end{lemma}

\section{Schr\"odinger operator with a strong $\delta$-interaction supported on a broken line}\label{sec: angle}
In this section, we study the properties of the Schr\"{o}dinger operator with a strong $\delta$-interaction supported on the boundary of an infinite sector. In particular, we analyze the asymptotics of the eigenvalues (when they exist), which, after localization arguments,  will lead to the asymptotic results stated in Theorem \ref{theo: corner induced}.

To begin with, let us fix the notations used in this section.
\begin{notation} From now on, for $\theta\in(0,\pi)$ we let 
	\begin{align}\label{def Gamma_theta}
		\Gamma_\theta:=\big\{(r\cos\omega,r\sin\omega)\in\rr^2:\, r\ge 0,\  |\omega|=\theta \big\},
	\end{align}
which is the union of two half-lines meeting at the origin with the angle $2\theta$ between them.
Then, we have the decomposition $\rr^2= \Omega_\theta^+\cup \Gamma_\theta \cup \Omega_\theta^-$ with the convention that $ \Omega_\theta^+ $ is the wedge with the angle $2\theta$ and the normal $\nu$ pointing outwards of $\Omega_+$ as can be seen on Figure \ref{fig: H^alpha_theta}. We also use the notation $\rr^2_\pm = \{(x_1,x_2) \in \rr^2 :\, \pm x_2 > 0 \}$ for the upper (respectively the lower) half-plane.
\begin{figure}
    \centering
\begin{tikzpicture}[x=0.75pt,y=0.75pt,yscale=-1,xscale=1]
%uncomment if require: \path (0,300); %set diagram left start at 0, and has height of 300

%Straight Lines [id:da46399712745662836] 
\draw [line width=1.25]   (220,60) -- (100,120) -- (220,180) ;
%Straight Lines [id:da0527690213482247] 
\draw    (100,190) -- (100,53) ;
\draw [shift={(100,50)}, rotate = 90] [fill={rgb, 255:red, 0; green, 0; blue, 0 }  ][line width=0.08]  [draw opacity=0] (3.57,-1.72) -- (0,0) -- (3.57,1.72) -- cycle    ;
%Straight Lines [id:da6176930156938303] 
\draw    (10,120) -- (217,120) ;
\draw [shift={(220,120)}, rotate = 180] [fill={rgb, 255:red, 0; green, 0; blue, 0 }  ][line width=0.08]  [draw opacity=0] (3.57,-1.72) -- (0,0) -- (3.57,1.72) -- cycle    ;
%Curve Lines [id:da33637562769662166] 
\draw    (120,110) .. controls (125,118.17) and (124,123.17) .. (120,130) ;

% Text Node
\draw (114,121.33) node [anchor=north west][inner sep=0.75pt]  [font=\tiny] [align=left] {$\displaystyle \theta $};
% Text Node
\draw (113.67,112.33) node [anchor=north west][inner sep=0.75pt]  [font=\tiny] [align=left] {$\displaystyle \theta $};
% Text Node
\draw (210.33,121.67) node [anchor=north west][inner sep=0.75pt]  [font=\tiny] [align=left] {$\displaystyle x_{1}$};
% Text Node
\draw (88,51) node [anchor=north west][inner sep=0.75pt]  [font=\tiny] [align=left] {$\displaystyle x_{2}$};
% Text Node
\draw (178.67,59.67) node [anchor=north west][inner sep=0.75pt]  [font=\scriptsize,color={black}  ,opacity=1 ] [align=left] {$\displaystyle \Gamma_{\theta }$};
% Text Node
\draw (176.67,100.33) node [anchor=north west][inner sep=0.75pt]  [font=\scriptsize] [align=left] {$\displaystyle \Omega_{\theta }^{+}$};
% Text Node
\draw (44.67,82.33) node [anchor=north west][inner sep=0.75pt]  [font=\scriptsize] [align=left] {$\displaystyle \Omega_{\theta }^{-}$};

\end{tikzpicture}
    \caption{Broken line $\Gamma_\theta$ with angle $2\theta$ splitting $\rr^2$ into $\Omega_\theta^+$ and $\Omega_\theta^-$.}
    \label{fig: H^alpha_theta}
\end{figure}
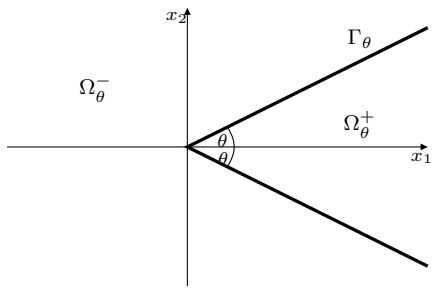

    Throughout this section, for $\alpha > 0$ we define  the sesquilinear form
    \begin{align}\label{h_alpha}
    h^\alpha_\theta (u,v) = \langle \nabla u, \nabla v \rangle_{L^2(\rr^2)} - \alpha \int_{\Gamma_\theta} \overline{u} v \dS, \quad D(h^\alpha_\theta) = H^1(\rr^2).
    \end{align}
    It is well known that $h^\alpha_\theta$ is closed and semi-bounded from below (see \cite{EN03}), and therefore generates a self-adjoint operator $H^\alpha_\theta$.
    In the following, we set
    \begin{align}\label{eigenvaluenumber}
    \begin{split}
        \kappa(\theta) &:= \text{the number of discrete eigenvalues of } H^\alpha_\theta,
    \end{split}
    \end{align}
    and for the special case $\alpha=1$ we use the special notation
    \begin{align}\label{eigenvaluenumber1}
    \begin{split}
        \mathcal{E}_n(\theta) &= E_n(H^1_\theta), \quad n \in \{1,\dots,\kappa(\theta) \},
    \end{split}
    \end{align}
    where we recall that $E_n(H^1_\theta)$ denotes the $n$-th eigenvalue of $H^1_\theta$.
    \end{notation}

  As mentioned in Remark \ref{Transmission condition}, one easily checks with the help of representation theorems that $H^\alpha_\theta$ is the operator defined on the domain 
  \begin{align*} 
		\begin{split}
        D(H^\alpha_\theta) = \big\{u = (u_+,u_-) \in H^1(\Omega_\theta^+)&\oplus H^1(\Omega_\theta^-): \Delta u_\pm \in L^2(\Omega_\theta^\pm),\\
   & u_+=u_- \text{ on }\Gamma_\theta \,\text{ and }\, \alpha u = \partial_{\nu}u_+ -\partial_{\nu}u_-\text{ on } \Gamma_\theta \big\},
    \end{split}
    \end{align*}
    and acts  in the sense of distributions as $H^\alpha_\theta u= (-\Delta u_+, -\Delta u_- )$.

We next introduce the Laplacian on $\Omega_\theta^+$ with an $\alpha-$Robin boundary condition on $\partial \Omega_\theta^+ = \Gamma_\theta$.
\begin{definition}\label{def_s_theta} Given $\alpha > 0$ and $\theta \in (0,\pi)$, we denote by $s^\alpha_\theta$ the sesquilinear form defind by
    \[
    s^\alpha_\theta(u) = \int_{\Omega_\theta^+} |\nabla u |^2 \mathrm{d}x - \int_{\Gamma_\theta} | u|^2 \dS, \quad D(s^{\alpha}_\theta) = H^1(\Omega_\theta^+)
    \]
   and let $S^\alpha_\theta$ be the operator generated by $s^\alpha_\theta$.
\end{definition}
The following proposition gathers the main properties of the operator $S^\alpha_\theta$ that were proved in \cite{KP18}. %We mention that assertion (vii) has only been shown for $\alpha =1$, and the case for any $\alpha >0$ is a direct result of Lemma \ref{lem: scaling anything}.
\begin{proposition}\label{lem: proposition 20} For any $\theta \in (0,\pi)$ and $\alpha > 0 $ the following hold true:
    \begin{itemize}
        \item[(i)] $S^\alpha_\theta$ is well-defined and semi-bounded from below by $-\frac{\alpha^2}{\sin^2\theta}$. 
        \item[(ii)] $S^\alpha_\theta \cong \alpha^2 S^1_\theta $.
        \item[(iii)] $\spece(S^\alpha_\theta) = [-\alpha^2,\infty)$.
        \item[(iv)] $ \specd (S^\alpha_\theta) \neq \emptyset$ if and only if $\theta < \frac{\pi}{2}$.
        \item[(v)] $S^\alpha_\theta$ has finitely many discrete eigenvalues for any $\theta$.
        \item[(vi)] $S^\alpha_\theta$ has exactly one discrete eigenvalue for $\theta \in [\frac{\pi}{6}, \frac{\pi}{2})$.
%        \item[(vii)] $S^\alpha_\theta \cong S^{N,\alpha}_\theta \oplus S^{D,\alpha}_\theta $ where $S^{D/N,\alpha}_\theta $ are defined by the sesquilinear forms
%        \begin{align*} 
%        s^{D/N,\alpha}_\theta (v) &= \int_{\Omega_\theta^+ \cap \rr^2_+} | \nabla v |^2 \mathrm{d} x - \alpha \int_{\Gamma_\theta \cap \rr^2_+} |v|^2 \dS ,\\
%        D(s^{D,\alpha}_\theta ) &= H^1(\Omega_\theta^+ \cap \rr^2_+), \quad D(s^{D,\alpha}_\theta) = \{ v \in H^1 (\Omega_\theta^+ \cap \rr^2_+) \mid v(\,\cdot\,, 0) = 0 \},
%        \end{align*}
%        where $\rr^2_+ = \{(x_1,x_2) \in \rr^2 \mid x_2 > 0 \}$, and it holds that $S^{D,\alpha}_\theta \geq -\alpha^2$.  
    \end{itemize}
\end{proposition}

\begin{remark}
The lower bound in (i) is not optimal, but sufficient for our purposes. In fact, it is known that the function
\[
(0,\dfrac{\pi}{2})\ni\theta\mapsto E_1 (S^1_\theta)\in(-1,-\frac{1}{4})
\]
is strictly increasing, continuous and surjective, and its asymptotic behavior for $\theta$ close to $0$
and $\frac{\pi}{2}$ can described very precisely, see e.g.~\cite{dr,EK15}.
\end{remark}

There are some relations between $H^\alpha_\theta$ and $S^\alpha_\theta$. Indeed, let $u \in H^1(\rr^2)$ then
    \begin{align}\label{H and S}
    \begin{split}
        h^\alpha_\theta (u) &:= 
    {\int_{\Omega_\theta^+} |\nabla u |^2 \mathrm{d}x - \frac{\alpha}{2} \int_{\Gamma_\theta} | u |^2 \dS} +
    {\int_{\Omega_\theta^-} |\nabla u |^2 \mathrm{d}x - \frac{\alpha}{2} \int_{\Gamma_\theta} |u|^2 \dS} 
    =s^\alpha_\theta(u_+)  +  s^\alpha_{\pi-\theta}(u_-).
    \end{split}
    \end{align}
    
Hence, several spectral properties of $S^\alpha_\theta$ can be transferred to $H^\alpha_\theta$.

\begin{lemma} \label{lem: H1theta in [-1/4)} For any $\theta \in (0, \frac{\pi}{2})$ there holds $\spece H ^1_\theta = [-\frac{1}{4}, \infty)$. Moreover, $H^1_\theta$ has at least one and at most finitely many discrete eigenvalues. 
\end{lemma}

\begin{proof}
The equality for the essential spectrum and the non-emptiness of the discrete spectrum
are shown in \cite[Props.~5.4 and~5.6]{EN03}.	In order to show the finiteness of the discrete spectrum,
we introduce $s^{\alpha}_{\pi -\theta}$ as in Definition \ref{def_s_theta} and consider the map
    \[
    J: H^1(\rr^2) \rightarrow H^1(\Omega_\theta^+) \times H^1(\Omega_\theta^-), \, u \mapsto  (u_+,u_-).
    \]
This allows us to write
     \begin{align*}
     h^1_\theta(u) &= \int_{\Omega_\theta^+} | \nabla u_+ |^2 \mathrm{d}x + \int_{\Omega_\theta^-} | \nabla u_-|^2 \mathrm{d}x - \frac{1}{2} \int_{\Gamma_\theta}  | u_+ |^2 \dS  
    - \frac{1}{2} \int_{\Gamma_\theta}  | u_- |^2 \dS  = (s^{\frac{1}{2}}_\theta \oplus s^{\frac{1}{2}}_{\pi -\theta}) (Ju),
    \end{align*}
   and thus $H^1_\theta \geq  S^{\frac{1}{2}}_\theta \oplus  S^{\frac{1}{2}}_{\pi-\theta}$.
   By Proposition~\ref{lem: proposition 20}, both operators on the right-hand side have a finite discrete spectrum in $(-\infty,-\frac{1}{4})$, which gives the sought conclusion by the min-max principle.
\end{proof}
Next, we summarize further properties of $H^\alpha_\theta$.

\begin{proposition}\label{lem: big prop} Let $\theta \in (0,\pi)$ and $\alpha > 0 $, and let  $\kappa(\theta)$ be as in \eqref{eigenvaluenumber}. Then, the following hold:
    \begin{itemize}
        \item[(i)] $H^\alpha_\theta$ is semibounded from below by $ -\frac{\alpha^2}{4\sin^2\theta} $.
        \item[(ii)] $H^\alpha_\theta$ is unitarily equivalent to $H^\alpha_{\pi -\theta}$.
        \item[(iii)] $H^\alpha_\theta$ is unitarily equivalent to $\alpha^2 H^1_\theta$.
        \item[(iv)] $\spece(H^\alpha_\theta) = [-\frac{\alpha^2}{4}, \infty)$.
        \item[(v)] $\kappa(\theta)$ is independent of $\alpha$ and it holds that $\kappa(\theta) < \infty$.
\end{itemize}
\end{proposition}
\begin{proof}  We first prove (i). Let $u \in H^1(\rr^2)$, then Proposition \ref{lem: big prop} together with \eqref{H and S} yield  
    \begin{align*}
        h^\alpha_\theta (u)  &\geq - \frac{\alpha^2}{4 \sin^2 \theta} \| u|_{\Omega_\theta^+}\|^2_{L^2(\Omega_\theta^+)}
     - \frac{\alpha^2}{4 \sin^2 (\pi-\theta)} \| u|_{\Omega_\theta^-} \|^2_{L^2(\Omega_\theta^-)} 
    = -\frac{\alpha^2}{4 \sin^2 \theta} \| u \|^2_{L^2(\rr^2)},
    \end{align*}
 which gives the desired lower bound.

    That $H^\alpha_\theta$ and $H^\alpha_{\pi -\theta}$ are unitarily equivalent to each other easily follows by doing a rotation of angle $\pi$. Assertion (iii) is a direct consequence of Lemma \ref{lem: scaling anything}, noting that $\alpha \rr = \rr $ and $\alpha \Gamma_\theta = \Gamma_\theta$ hold for any $\alpha>0$. 
    
    Concerning (iv), for $\theta \in (0,\frac{\pi}{2}) \cup (\frac{\pi}{2}, \pi)$ the result follows from (ii) and Lemma \ref{lem: H1theta in [-1/4)}. In the case $\theta = \frac{\pi}{2}$,  by a separation of variables, it follows that 
    $H^1_\frac{\pi}{2} \cong T_1 \otimes I + I \otimes T$,
    with $T_1$ and $T$ as in Definition \ref{def t X}. Since $ \inf \spec T_1 = -\frac{1}{4}$ and $\spec T = [0, \infty)$ we deduce that $\spece(H^\alpha_{\frac{\pi}{2}}) = [-\frac{\alpha^2}{4}, \infty)$.
    
     Finally, assertion (v) follows directly from assertion (iii) and Lemma \ref{lem: H1theta in [-1/4)}.
\end{proof}
%One can easily see that $H^\alpha_\theta $ is unitarily equivalent to $H^\alpha_{\pi - \theta}$ by rotation by $\pi$. Therefore we will consider $\pi \in (0,\frac{\pi}{2})$ from now on.

%\begin{corollary}
%    $\spec_{ess}(H^\alpha_\theta) = [ -\frac{\alpha^2}{4}, \infty)$
%\end{corollary}
%\begin{proof}
%    Follows from \ref{}, \ref{H1theta geq -1/4}, \ref{H1theta in [-1/4)}
%\end{proof}

%\begin{definition}
%    We define 
%    $$\kappa(\theta) := \text{the number of discrete eigenvalues of } H^\alpha_\theta $$
%    Because of $ H^\alpha_\theta \cong \alpha^2 H^1_\theta $ $\kappa(\theta)$ is independent of $\alpha$. Also in \ref{H1theta geq -1/4} we showed that $H^1_\theta \geq S^{\frac{1}{2}}_{\theta} + S^{\frac{1}{2}}_{\pi-\theta}$ and because $S^\alpha_\theta$ has finitely many eigenvalues $\kappa(\theta) < \infty$ holds.
%\end{definition}

The following assertions are proved in \cite[Prop.~5.12 and Thm.~5.8]{EN03}:
\begin{lemma} \label{lem: kappa non increasing}\label{lem: Lambda_n strictly increasing}
Each individual eigenvalue of $H^1_\theta$ is strictly increasing with respect to $\theta\in(0,\frac{\pi}{2})$.
Hence, the counting function  $(0, \frac{\pi}{2})\ni \theta \mapsto\kappa(\theta)$ is non-increasing. It also holds $ \kappa(\theta) \to \infty $ as $\theta$ tends to $0$.
\end{lemma}

Another consequence of Lemma \ref{lem: kappa non increasing} is that, for certain angles, $H^\alpha_\theta$ has exactly one eigenvalue.
\begin{lemma}\label{lem: kappa theta = 1} For any $\theta \in [\frac{\pi}{6}, \frac{\pi}{2} ) \cup (\frac{\pi}{2}, \frac{5 \pi }{6}]$ there holds $\kappa(\theta) = 1$.
\end{lemma}
\begin{proof}
Thanks to Proposition \ref{lem: big prop}(ii) and Lemma \ref{lem: kappa non increasing}, it suffices to prove that $\kappa(\frac{\pi}{6})<2$. We proceed by contradiction. Suppose, to the contrary, $\kappa(\frac{\pi}{6}) \geq 2$, which is equivalent to $\Lambda_2(H^1_{\frac{\pi}{6}}) < -\frac{1}{4}$. From  \eqref{H and S} we know that $ H^1_{\theta} \geq S^{\frac{1}{2}}_\theta \oplus S^\frac{1}{2}_{\pi-\theta}$, and since $S^\alpha_{\pi-\theta}$ has no eigenvalues below $-\frac{1}{4}$  for $\theta = \frac{\pi}{6}$, using Proposition \ref{lem: proposition 20}(iii)-(vi), it follows that 
\[
\Lambda_2(H^1_\frac{\pi}{6})\geq \Lambda_2(S^\frac{1}{2}_\frac{\pi}{6})\geq -\frac{1}{4},
\]
which contradicts our assumption.
\end{proof}

The existence of an Agmon-type decay estimate guarantees that the eigenfunctions of the operator $H^\alpha_\theta$ exhibit a form of spatial concentration near the origin. This property plays a crucial role in the analysis presented in Section \ref{sec: Karl}. We begin by establishing this estimate in the special case $\alpha = 1$.

\begin{lemma}\label{lem: agmon estimate H1Theta} Given $\theta \in (0, \frac{\pi}{2})$, let $\mathcal{E}$ be a discrete eigenvalue of $H^1_\theta$ and $\psi$ be an associated eigenfunction. Then, for any $\varepsilon \in (0,1)$ we have
    \[
    \int_{\rr^2} (|\nabla \psi|^2 + |\psi|^2) e^{2(1-\varepsilon)\sqrt{-\frac{1}{4} - \mathcal{E} } |x|} \mathrm{d} x < \infty.
    \]
\end{lemma}
\begin{proof}
    Let $\varepsilon \in (0,1)$ and $L>0$, and set $f_{L,\varepsilon}(x) = (1-\varepsilon) \sqrt{-\frac{1}{4} - \mathcal{E}} \min(|x|,L)$. Observe that 
    \begin{align}\label{compuet grad}
    \begin{split}
    | \nabla(e^{f_{L,\varepsilon}})|^2 &=|e^{f_{L,\varepsilon}} \nabla \psi |^2 + | e^{f_{L,\varepsilon}} \psi \nabla f_{L,\varepsilon} |^2 + \Re (2e^{2 f_{L,\varepsilon}} \psi \langle \nabla f_{L,\varepsilon}, \nabla \psi \rangle) \\
    &= |\nabla f_{L,\varepsilon}|^2 e^{2 f_{L,\varepsilon}} | \psi |^2+  \Re(\langle \nabla (e^{2f_{L,\varepsilon}} \psi ),\nabla \psi \rangle)
    \end{split}.
    \end{align}
    Hence
    \begin{align*}
    \int_{\rr^2} |\nabla(e^{ f_{L,\varepsilon} } \psi)|^2 \mathrm{d}x \geq& \int_{\rr^2} |e^{ f_{L,\varepsilon}} \nabla \psi|^2
+ | \nabla f_{L,\varepsilon} |^2  |e^{ f_{L,\varepsilon}} \psi|^2
- 2  | e^{ f_{L,\varepsilon}} \nabla \psi|  | \nabla f_{L,\varepsilon} e^{ f_{L,\varepsilon}} \psi| \text{ d}x \\
&\geq \int_{\rr^2} \big( | e^{ f_{L,\varepsilon}} \nabla \psi |^2 + | \nabla f_{L,\varepsilon} |^2 | e^{ f_{L,\varepsilon}} \psi |^2 \\
&- 2 ( \frac{1}{4} | e^{ f_{L,\varepsilon}} \nabla \psi |^2  + | \nabla f_{L,\varepsilon} |^2 | e^{f_{L,\varepsilon}} \psi|^2)\big) \mathrm{d} x\\
&\geq \int \frac{1}{2} | e^{ f_{L,\varepsilon}} \nabla \psi|^2 - (1-\varepsilon)^2 (-\frac{1}{4} - \mathcal{E}) | e^{ f_{L,\varepsilon}} \psi |^2 \mathrm{d}x,
\end{align*}
and thus
 \[
 \| \psi e^{f_{L,\varepsilon}}\|^2_{H^1(\rr^2)} \geq \int_{\rr^2} (1- (1-\varepsilon)^2(-\frac{1}{4} - \mathcal{E}) ) | \psi |^2 e^{2 f_{L,\varepsilon}}  \mathrm{d} x 
+ \frac{1}{2}  \int_{\rr^2} | \nabla \psi |^2 e^{2f_{L,\varepsilon}} \mathrm{d} x.
\]
Now, we claim that there exists a constant $K_\varepsilon>0$, depending only on $\varepsilon$, such that
    \begin{equation}  \label{eq: Agmon estimate h1theta norm bla bla leq K_epsilon} 
    \| \psi e^{f_{L,\varepsilon}} \|^2_{H^1(\rr^2)} \leq K_{\varepsilon}.
    \end{equation}
    Assuming that this bound holds, we obtain\[
\int_{\rr^2}(| \nabla \psi |^2 +| \psi |^2 ) e^{2f_{L,\varepsilon}} \mathrm{d}x \leq 2K_\varepsilon \Big(1 + (1-\varepsilon)^2(-\frac{1}{4}-\mathcal{E})\Big).
\]
Since the right-hand side is independent of $L$, the desired inequality follows by letting $L\longrightarrow\infty$.
    
 The proof of the bound \eqref{eq: Agmon estimate h1theta norm bla bla leq K_epsilon} follows similar arguments as the ones in \cite[Theorem 5.1]{KP18}. We first perform an IMS localization to obtain a suitable lower bound for $h^\alpha_\theta$. Let $R>0$ and let $\chi_0, \chi_1 \in C^{\infty}(\rr_+)$ be such that
    \[\chi_0^2 + \chi_1^2 =1, \quad 
    \chi_0(t) = 1 \,\text{ if } t\leq 1, \quad \chi_0(t) = 0 \,\text{ if } t\geq 2,
    \]
    and set $\chi_{j,R}(|x|/R)$ for $j = 0,1$. Then, Lemma \ref{lem: IMS partition} yields  
   \[
   h^\alpha(u) = h^\alpha_\theta(u\chi_{0,R}) + h^\alpha_\theta(u\chi_{1,R}) - \sum_{j= 1,2,} \| u \nabla \chi_{j,R}\|^2_{L^2(\rr^2)}.
   \]
    Thus, there exists $C_1>0$ such that
    \[
    h^\alpha(u) \geq h^\alpha_\theta(u\chi_{0,R}) + h^\alpha_\theta(u\chi_{1,R}) - \frac{C_1}{R^2} \| u \|^2_{L^2(\rr^2)}.
    \]
    Using this, together with the identity $h^1_\theta(u) = \delta \| \nabla u \|^2_{L^2(\rr^2)} + (1-\delta )h^{\frac{1}{1-\delta}}_\theta (u)$,  where $\delta \in (0,1)$ will later be chosen sufficiently small, we get
    \begin{equation}
        \label{eq: Agmon geq j = 0 j = 1}
        h^1_\theta(u) \geq \delta \| \nabla u  \|^2_{L^2(\rr^2)} + (1-\delta) \left(h^{\frac{1}{1-\delta}}_\theta (\chi_{0,R}u) + h^{\frac{1}{1-\delta}}_\theta (\chi_{1,R}u) - \frac{C_1}{R} \| u \|^2_{L^2(\rr^2)} \right).
    \end{equation}
    Next, for $j = 0,1$, we estimate from below $ h^{\frac{1}{1-\delta}}_\theta (\chi_{j,R}u)$. Notice that for $j = 0$, we have
    \begin{equation}\label{eq: Agmon j = 0}
        h_\theta^{\frac{1}{1-\delta}} (\chi_{0,R}u) \geq E_1(H^{\frac{1}{1-\delta}}_\theta) \| u \chi_{0,R} \|^2_{L^2(\rr^2)} \geq - \frac{1}{4(1-\delta)^2 \sin^2 \theta} \| u \chi_{0,R}\|^2_{L^2(\rr^2)}.
    \end{equation}
     Now, in the case $j = 1$, we partition the region $\{|x|> R \}$, which contains the support of $\chi_{1,R}$ (see Figure \ref{fig: agmon estimate H1theta}) as follows:
    
    \begin{align*}
        D_+ &= \{ (x_1,x_2) \in \rr_+ \times \rr_+:\, x_1- R \leq \frac{x_2}{\tan \theta} \leq x_1+R \} \cap \{ | x |>R \};\\
    D_- &= \{ (x_1,x_2) \in \rr_+ \times (-\rr_+):\, x_1- R \leq \frac{-x_2}{\tan \theta} \leq x_1+R \} \cap \{ | x |>R \};\\
     D_{in} &= (R,0 )+ \Omega_\theta^+,\quad\text{and } D_{out} = \rr^2 \setminus (D^+ \cup D^- \cup D_{in} \cup \{ | x | \leq R \}).
     \end{align*}
     \begin{figure}
        \centering
        \scalebox{0.9}{\begin{tikzpicture}[x=0.85pt,y=0.85pt,yscale=-1,xscale=1]
%uncomment if require: \path (0,300); %set diagram left start at 0, and has height of 300

%Straight Lines [id:da28588103064167314] 
\draw [line width=1.25]    (170,110) -- (70,160) -- (170,210) ;
%Straight Lines [id:da517307523901633] 
\draw    (100,160) -- (180,120) ;
%Straight Lines [id:da42627547107970054] 
\draw    (88.14,135.86) -- (160,100) ;
%Straight Lines [id:da8518539626433472] 
\draw    (100,160) -- (180,200) ;
%Straight Lines [id:da31913108889647235] 
\draw    (88.14,183.86) -- (160,220) ;
%Shape: Polygon [id:ds7872814003308647] 
\draw  [draw opacity=0][fill={rgb, 255:red, 155; green, 155; blue, 155 }  ,fill opacity=0.3 ] (160,100) -- (180,120) -- (100,160) -- (88.14,135.86) -- cycle ;
%Shape: Polygon [id:ds9205176341580783] 
\draw  [draw opacity=0][fill={rgb, 255:red, 155; green, 155; blue, 155 }  ,fill opacity=0.3 ] (100,160) -- (180,200) -- (160,220) -- (88.14,183.86) -- cycle ;
%Shape: Circle [id:dp7142442207756687] 
\draw  [fill={rgb, 255:red, 255; green, 255; blue, 255 }  ,fill opacity=1 ] (40,160) .. controls (40,143.43) and (53.43,130) .. (70,130) .. controls (86.57,130) and (100,143.43) .. (100,160) .. controls (100,176.57) and (86.57,190) .. (70,190) .. controls (53.43,190) and (40,176.57) .. (40,160) -- cycle ;
%Straight Lines [id:da3380590865912586] 
\draw    (70,230) -- (70,93) ;
\draw [shift={(70,90)}, rotate = 90] [fill={rgb, 255:red, 0; green, 0; blue, 0 }  ][line width=0.08]  [draw opacity=0] (3.57,-1.72) -- (0,0) -- (3.57,1.72) -- cycle    ;
%Straight Lines [id:da9303865382982106] 
\draw    (20,160) -- (217,160) ;
\draw [shift={(220,160)}, rotate = 180] [fill={rgb, 255:red, 0; green, 0; blue, 0 }  ][line width=0.08]  [draw opacity=0] (3.57,-1.72) -- (0,0) -- (3.57,1.72) -- cycle    ;

% Text Node
\draw (171,159) node [anchor=north west][inner sep=0.75pt]  [font=\scriptsize] [align=left] {$\displaystyle D_{in}$};
% Text Node
\draw (31,112) node [anchor=north west][inner sep=0.75pt]  [font=\scriptsize] [align=left] {$\displaystyle D_{out}$};
% Text Node
\draw (140.43,106.86) node [anchor=north west][inner sep=0.75pt]  [font=\scriptsize] [align=left] {$\displaystyle D_{+}$};
% Text Node
\draw (135.29,197.43) node [anchor=north west][inner sep=0.75pt]  [font=\scriptsize] [align=left] {$\displaystyle D_{-}$};
% Text Node
\draw (99.2,162.6) node [anchor=north west][inner sep=0.75pt]  [font=\tiny] [align=left] {$\displaystyle R$};
% Text Node
\draw (213.6,161) node [anchor=north west][inner sep=0.75pt]  [font=\tiny] [align=left] {$\displaystyle x_{1}$};
% Text Node
\draw (55.6,89.6) node [anchor=north west][inner sep=0.75pt]  [font=\tiny] [align=left] {$\displaystyle x_{2}$};
\end{tikzpicture}}
        \caption{Partition of the region $\{ | x | >R \}$.}
        \label{fig: agmon estimate H1theta}
    \end{figure}
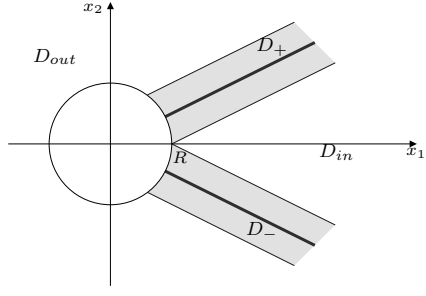
    
We introduce the sesquilinear forms
    \begin{align*}
     q^\pm (v) &= \int_{D_\pm} |\nabla u|^2 \mathrm{d}x - \frac{1}{1-\delta} \int_{D_\pm \cap \Gamma_\theta} |u|^2 \dS, \,\, D(q^\pm) = \{ v \in H^1(D_\pm): v(x) = 0 \text{ if }  x | = R \},\\
      q^{in/out}  (v) &= \int_{D_{in/out}} |\nabla v|^2 \mathrm{d} x, \quad D(q^{in/out}) = \{ v \in H^1(D_{in/out}):  v(x) = 0\text{ if } | x | = R \}.
      \end{align*}
    %And we set 
    %$$ J_1:H^1(\rr^2) \rightarrow  $$
    %(...)
    %\\and we get $H^{\frac{1}{1-\delta}}_\theta \geq Q^+ \oplus Q^- \oplus Q^{in} \oplus Q^{out}$:
   Note that $Q^+$ and $Q^-$ are unitarily equivalent. Using this and taking the restriction of $u\chi_{1,R}$ onto  $D_+$, $ D_-$, $ D_{in}$, and $D_{out}$, we arrive at
    \begin{align*}
        h^{\frac{1}{1-\delta}}_\theta (u\chi_{1,R}) = q^+(u\chi_{1,R}) + q^-(u\chi_{1,R}) + q^{in}(u\chi_{1,R}) + q^{out}(u\chi_{1,R}) \\ 
        \geq \Lambda_1(Q^+) \left(\| u \chi_{1,R} \|^2_{L^2(D^+)} + \| u \chi_{1,R} \|^2_{L^2(D^+)}\right),
    \end{align*}
 where we used the non-negativity of $Q^{in/out}$.  To estimate $\Lambda_1(Q^+)$, set
    \[
    U_\theta = \{ (x_1,x_2) \in \rr^2 \mid x_1 -R \leq \frac{x_2}{\tan \theta} \leq x_1 + R \}.
    \]
    Observe that for $v\in D(q^+)$, if we denote by $\tilde{v}$ its zero extension to $U_\theta$, then $q^+(v) = q(Jv) $, where $q$ is defined by
    \[
    q(v) = \int_{U_\theta} |\nabla v|^2 \mathrm{d}x - \frac{1}{1-\delta} \int_{\rr} |v(\frac{x_2}{\theta}, x_2)|^2 \frac{\mathrm{d} x_2}{\sin \theta}, \quad D(q) = H^1(U_\theta).
    \]
    Applying a clockwise rotation of angle $ \theta$, we easily see that $Q\cong I \otimes(- \Delta) + T^N_{R,\frac{1}{1-\delta}} \otimes I$, with $T^N_{R,\frac{1}{1-\delta}}$ as in Definition \ref{def t X} and $-\Delta$ as the free Laplacian in $\R$. From this and Proposition \ref{lem: ev of t,N,L,alpha} it follows that
    \[
    \Lambda_1(Q^+) \geq \Lambda_1(Q) \geq -\frac{1}{4(1-\delta)^2} - \frac{C_2}{(1-\delta)^2} e^{-\frac{1}{2} \frac{R \sin \theta}{1-\delta}}
    \]
    holds some $C_2 > 0$, and thus for $j = 1$, we obtain the estimate 
    \[
     h^{\frac{1}{1-\delta}}(u \chi_{1,R}) \geq \left(-\frac{1}{4(1-\delta)^2} - \frac{C_2}{(1-\delta)^2} e^{-\frac{1}{2} \frac{R \sin \theta}{1-\delta}} \right) \| u \chi_{1,R} \|,
    \]
    which, combined with \eqref{eq: Agmon j = 0} and \eqref{eq: Agmon geq j = 0 j = 1}, yields
    \begin{align} \label{eq: Agmon H1theta geq}
    \begin{split}
       h^1_\theta (u) \geq& \delta \| \nabla u  \|^2_{L^2(\rr^2)} - \frac{1}{4(1-\delta) \sin^2 \theta} \| u \chi_{0,R} \|^2\\
    &-\left(\frac{1}{4(1-\delta)^2} + \frac{C_2}{(1-\delta)^2} e^{-\frac{1}{2} \frac{R \sin \theta}{1-\delta}}\right) \| u \chi_{1,R} \|^2 - \frac{C_1(1-\delta)}{R^2} \| u \|^2_{L^2(\rr^2)}.
     \end{split}
    \end{align}
    The next step is to apply this estimate to $\psi e^{f_{L,\varepsilon}}$. For this, note that
    \begin{equation}\label{eq: Agmon H1theta nabla fLepsilon leq}
    \| \nabla f_{L,\varepsilon}\|^2 \leq (1-\varepsilon)^2(-\frac{1}{4} - \mathcal{E}).
    \end{equation}
    Using \eqref{compuet grad} and integrating by parts, we rewrite $h^1_\theta(\psi e^{f_{L,\varepsilon}}) $ as
    \begin{align*}
        h^1_\theta(\psi  e^{f_{L,\varepsilon}}) =&  \int_{\rr^2} |\nabla (\psi e^{f_{L,\varepsilon} })|^2 \mathrm{d}x - \int_{\Gamma_\theta} |\psi e^{f_{L,\varepsilon} } |^2 \dS \\ 
        =& \int_{\rr^2} |\nabla f_{L,\varepsilon}|^2 e^{2 f_{L,\varepsilon}} | \psi |^2 \mathrm{d}x  + \int_{\Omega_\theta^+} \Re(\langle \nabla (e^{2f_{L,\varepsilon}} \psi), \nabla\psi  \rangle) \mathrm{d}x  +\int_{\Omega_\theta^-} \Re (\langle \nabla (e^{2f_{L,\varepsilon}} \psi), \nabla\psi  \rangle) \mathrm{d}x \\
        &- \int_{\Gamma_\theta} |\psi e^{f_{L,\varepsilon} }|^2 \dS \\
        =& \int_{\rr^2}| \nabla f_{L,\varepsilon} |^2 e^{2 f_{L,\varepsilon}} | \psi |^2 + e^{2f_{L,\varepsilon}} \overline{\psi}(- \Delta \psi) \mathrm{d}x
        + \int_{\Gamma_\theta} \Re (e^{2 f_{L,\varepsilon}} \overline{\psi} (\partial_\nu \psi - \partial_\nu\psi - \psi ))\dS\\
        =& \int_{\rr^2} | \psi |^2 e^{2 f_{L,\varepsilon}} (\mathcal{E} + | \nabla f_{L,\varepsilon}|) \text{ d}x.
    \end{align*}
    Combining this with (\ref{eq: Agmon H1theta geq}) and (\ref{eq: Agmon H1theta nabla fLepsilon leq}), we obtain    \begin{align*}
         (\mathcal{E}&- \mathcal{E} -\frac{1}{4} + (\varepsilon^2 - 2 \varepsilon)(-\frac{1}{4} - \mathcal{E})) \| \psi e^{f_{L,\varepsilon}} \|^2_{L^2(\rr^2)} \geq \int_{\rr^2} | \psi |^2 e^{2f_{L,\varepsilon}}(\mathcal{E} + |\nabla f_{L,\varepsilon} |) \mathrm{d} x \\
             =&  h^1_\theta(\psi e^{f_{L,\varepsilon}}) \geq \delta \| \nabla \psi e^{f_{L,\varepsilon}} \|^2_{L^2(\rr^2)} - \frac{1}{4(1-\delta) \sin^2 \theta} \|\psi e^{f_{L,\varepsilon}}  \chi_{0,R}\|_{L^2(\rr^2)}^2 \\
         &-(\frac{1}{4(1-\delta)^2} + \frac{C_2}{(1-\delta)^2} e^{-\frac{1}{2} \frac{R \sin \theta}{1-\delta}}) \| \psi e^{f_{L,\varepsilon}}  \chi_{1,R}\|_{L^2(\rr^2)}^2 - \frac{C_1(1-\delta)}{R^2} \| \psi e^{f_{L,\varepsilon}} \|^2_{L^2(\rr^2)},
    \end{align*}
    and since $\chi_{0,R}^2 + \chi_{1,R}^2 = 1$, it follows that
    \[
    A \| \psi e^{f_{L,\varepsilon}} \chi_{0,R} \|^2_{L^2(\rr^2)} \geq \delta \| \nabla (\psi e^{f_{L,\varepsilon}} )\|^2_{L^2(\rr^2)} + B \| \psi e^{f_{L,\varepsilon}} \chi_{1,R} \|^2_{L^2(\rr^2)},
    \]
where A and B are defined by
\begin{align*}
A &= {(\varepsilon^2 - 2 \varepsilon)} 
{(-\frac{1}{4} - \mathcal{E})} + {\frac{\cos^2\theta + \delta \sin^2\theta}{4(1-\delta)\sin^2\theta}} 
+ \frac{C_1(1-\delta)}{R^2},\\
B &= (2\varepsilon - \varepsilon^2) (-\frac{1}{4} - \mathcal{E}) - \frac{1}{4} \frac{\delta}{1-\delta} - \frac{C_2}{1-\delta} e^{-\frac{1}{2} \frac{R \sin \theta}{1-\delta}} - \frac{C_1(1-\delta)}{R^2} .
\end{align*}
By Proposition \ref{lem: big prop}(i), we have $-\frac{1}{4}- \mathcal{E} \leq -\frac{1}{4} - E_1(H^1_\theta) \leq -\frac{1}{4}\frac{\cos^2\theta}{\sin^2\theta}$. Moreover,  since $(\varepsilon^2-2\varepsilon) < 0$ for all $\varepsilon \in(0,1)$, it follows that we can choose $\delta_\varepsilon$ sufficiently small such that 
\begin{align*}
    (2\varepsilon - \varepsilon^2) (-\dfrac{1}{4} - \mathcal{E}) - \frac{1}{4} \frac{\delta_\varepsilon}{1-{\delta}_\varepsilon}  > 0\, \text{ and }\,
    \ (\varepsilon^2 - 2\varepsilon) (-\frac{1}{4} - \mathcal{E})  + {\frac{\cos^2\theta + \delta_\varepsilon \sin^2\theta}{4(1-\delta_\varepsilon)\sin^2\theta}} > 0 .
\end{align*}
 Since the residual terms involved in $A$ and $B$ vanish as $R\to \infty$, we can find some $R_\varepsilon>0 $ such that  for all $R\geq R_\varepsilon$ we have $A> 0$ and $B > 0$. Consequently, there exist $A_\varepsilon>0$ and $B_\varepsilon >0$ such that $ A_\varepsilon \geq A $ and $B_\varepsilon \leq B$  for all $R\geq R_\varepsilon$. This implies the inequality
\[
A_\varepsilon \| \psi e^{f_{L,\varepsilon}} \chi_{0,R} \|^2_{L^2(\rr^2)} \geq \delta_\varepsilon \| \nabla (\psi e^{f_{L,\varepsilon}} )\|^2_{L^2(\rr^2)} + B_\varepsilon \|\psi e^{f_{L,\varepsilon}} \chi_{1,R} \|^2_{L^2(\rr^2)}, \quad \forall R \geq R_\varepsilon.
\]
Thus, 
\[
C_\varepsilon \| \psi e^{f_{L,\varepsilon}} \chi_{0,R} \|^2_{L^2(\rr^2)} \geq \|\psi e^{f_{L,\varepsilon}} \|^2_{H^1(\rr^2)}
\quad
\text{for }
C_\varepsilon := \frac{(A_\varepsilon + B_\varepsilon)}{ B_\varepsilon}.
\]
Combining this with the straightforward bound
\[
\| \psi e^{f_{L,\varepsilon}} \chi_{0,R} \|^2_{L^2(\rr^2)} 
\leq e^{4(1-\varepsilon)\sqrt{-\frac{1}{4} - \mathcal{E} }R} \| \psi \chi_{0,R} \|^2_{L^2(\rr^2)}
\leq e^{4(1-\varepsilon)\sqrt{-\frac{1}{4} - \mathcal{E} }R} \| \psi \|^2_{L^2(\rr^2)},
\]
we then obtain the claimed inequality \eqref{eq: Agmon estimate h1theta norm bla bla leq K_epsilon}. This concludes the proof. 
\end{proof}
After establishing the Agmon-type estimate for $\alpha = 1$, extending it to any $\alpha >0$ by means of scaling becomes significantly simpler. 
\begin{corollary} \label{lem: agmon estimate HAlphaTheta}
    Let $\theta \in (0, \frac{\pi}{2})$ and let $\psi_\alpha$ be an eigenfunction of $H^\alpha_\theta$. Then, there exist $b,B >0$ such that
    \[
    \int_{\rr^2} e^{b \alpha | x |} \left(\frac{1}{\alpha^2} |\nabla \psi_\alpha| + | \psi_\alpha | \right) \mathrm{d}x \leq B \| \psi_\alpha \|^2_{L^2(\rr^2)}.
    \]
\end{corollary}
\begin{proof} From Lemma \ref{lem: scaling anything} it follows that $ \psi_\alpha $ is an eigenfunction of $H^\alpha_\theta$ if and only if $ \frac{1}{\alpha} \psi_\alpha(\frac{x_1}{\alpha} , \frac{x_2}{\alpha}) = \psi_1 $ is an eigenfunction of $H^1_\theta$. Let $E$ be the eigenvalue associated with  $\psi_1$, and set 
\[
b = 2(1-\varepsilon)\sqrt{-\tfrac{1}{4}- E}.
\]
 By Lemma \ref{lem: agmon estimate H1Theta} there exists $B>0$ such that
     \[
    \int_{\rr^2} (|\nabla \psi_1| + | \psi_1 |^2)e^{b | x |} \mathrm{d}x \leq B\| \psi_1 \|^2_{L^2(\rr^2)}.
    \]
    Applying the change of variables $(y_1,y_2) = (\alpha x_1, \alpha x_2)$, we get
    \[
    \int_{\rr^2} e^{b \alpha | x |}(\frac{1}{\alpha^2} |\nabla\psi_\alpha|^2 + | \psi_\alpha |^2) \mathrm{d}x = \int_{\rr^2}e^{b |y|}( | \nabla \psi_1 |^2 + | \psi_1 |) \mathrm{d}y \leq B \| \psi_1\|^2_{L^2(\rr^2)} = B \| \psi_\alpha \|^2_{L^2(\rr^2)}.
    \]
This gives the desired inequality.
\end{proof}

\section{Neighborhoods of straight corners}\label{sec: Karl}
While in the case of Robin Laplacians in curvilinear polygons (cf. \cite{KOBP20}), the behavior of eigenvalues induced by corners is analyzed via truncated curved sectors, the case of $\delta$-interactions is more subtle. In particular, one must carefully select an appropriate neighborhood near each corner to capture the spectral effects accurately. Before addressing this point in detail, we first focus on the analysis within a suitable neighborhood of $\delta$-interactions whose support locally coincides with the curve $\Gamma_\theta$ defined as in \eqref{def Gamma_theta}. In this context, we provide the definition of the non-resonance condition.  To this end, let us first fix the notation used throughout this section and precisely define the geometric setting as well as the operators of interest.
%As was already shown in \cite{BCP25} for the simple case of $\delta$-interactions supported on a piecewise smooth curve with a single corner, the asymptotics of eigenvalues induced by corners 
\begin{definition} \label{def: Karl}Let $\theta \in (0,\frac{\pi}{2})$ and let $\Gamma_\theta$ be as in \eqref{def Gamma_theta}. For $R>0$, we denote by $K^R_\theta$ the interior of a kite with angle $2\theta$ at the vertex $A=(-R/\sin\theta, 0)$ such that non-adjacent edges to $A$ are of length $2R$ (see Figure \ref{fig: Karl minus Karl}(a)). We further denote by $\partial_* K^R_\theta$ the part of $\partial K^R_\theta$ non-adjacent to $A$, and let 
    \[
    \Gamma_\theta^R := \Gamma_\theta \cap \left\{| x | < \frac{R}{\tan \theta}\right\}
    \]
    be the support of the $\delta$-interaction when restricted to the domain $K^R_\theta$. In addition, for $0<r<R$, we define the complement of two kites with the same angle as (see Figure \ref{fig: Karl minus Karl}(b)):
     \[
     \quad K^{R,r}_\theta := K^R_\theta \setminus \overline{K^r_\theta}, \quad  \Gamma^{R,r}_\theta := \Gamma_\theta \cap \left\{ \frac{r}{\tan \theta} < | x | < \frac{R}{\tan \theta} \right\}.
     \]
\end{definition}
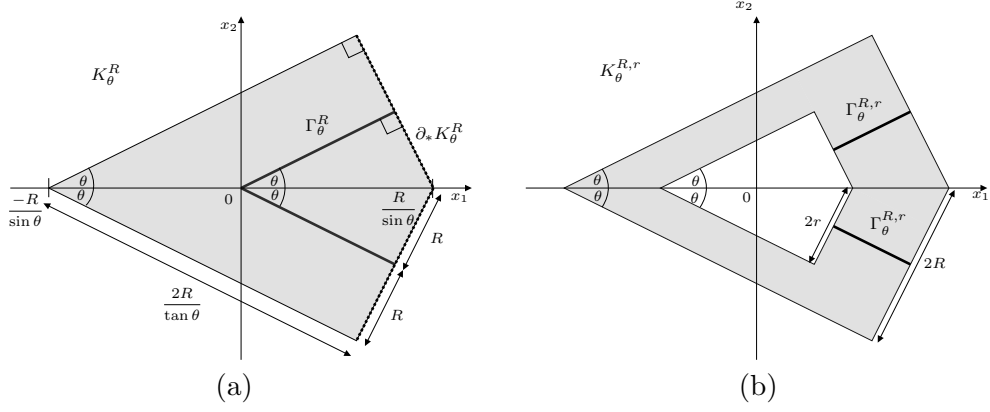
\begin{figure}
\centering
\begin{tabular}{ccc}
  \scalebox{0.8}{
  \begin{tikzpicture}[x=0.9pt,y=0.9pt,yscale=-1,xscale=1]
%uncomment if require: \path (0,300); %set diagram left start at 0, and has height of 300

%Straight Lines [id:da13684878187424554] 
\draw [line width=1.25]   (490,150) -- (570,110) ;
%Straight Lines [id:da5988517570536442] 
\draw [line width=1.25]   (490,150) -- (570,190) ;
%Shape: Polygon [id:ds7387159274155469] 
\draw  [fill={rgb, 255:red, 155; green, 155; blue, 155 }  ,fill opacity=0.3 ] (550,70) -- (590,150) -- (550,230) -- (390,150) -- cycle ;
%Straight Lines [id:da042721572879728886] 
\draw    (490,240) -- (490,63) ;
\draw [shift={(490,60)}, rotate = 90] [fill={rgb, 255:red, 0; green, 0; blue, 0 }  ][line width=0.08]  [draw opacity=0] (3.57,-1.72) -- (0,0) -- (3.57,1.72) -- cycle    ;
%Straight Lines [id:da974814482683094] 
\draw    (370,150) -- (607,150) ;
\draw [shift={(610,150)}, rotate = 180] [fill={rgb, 255:red, 0; green, 0; blue, 0 }  ][line width=0.08]  [draw opacity=0] (3.57,-1.72) -- (0,0) -- (3.57,1.72) -- cycle    ;
%Straight Lines [id:da4622809516771925] 
\draw    (389.88,159.34) -- (544.52,236.66) ;
\draw [shift={(547.2,238)}, rotate = 206.57] [fill={rgb, 255:red, 0; green, 0; blue, 0 }  ][line width=0.08]  [draw opacity=0] (3.57,-1.72) -- (0,0) -- (3.57,1.72) -- cycle    ;
\draw [shift={(387.2,158)}, rotate = 26.57] [fill={rgb, 255:red, 0; green, 0; blue, 0 }  ][line width=0.08]  [draw opacity=0] (3.57,-1.72) -- (0,0) -- (3.57,1.72) -- cycle    ;
%Shape: Right Angle [id:dp8543629147037692] 
\draw   (553.77,77.98) -- (545.97,81.67) -- (542.2,73.69) ;
%Shape: Right Angle [id:dp07307094105098622] 
\draw   (573.77,117.98) -- (565.97,121.67) -- (562.2,113.69) ;
%Straight Lines [id:da4683863683526899] 
 \draw [densely dotted] [line width=1.25]   (550,70) -- (590,150) ;
%Straight Lines [id:da7108946752405022] 
\draw [densely dotted] [line width=1.25]   (590,150) -- (550,230) ;
%Curve Lines [id:da0719142997971346] 
\draw    (410.6,139.4) .. controls (415,145.8) and (413.8,154.6) .. (409.4,159) ;
%Curve Lines [id:da005022712495085235] 
\draw    (509.8,139.4) .. controls (513.8,145.4) and (513.8,153.4) .. (510.2,159.4) ;
%Straight Lines [id:da1691363335309154] 
\draw    (572.86,195.69) -- (555.74,230.11) ;
\draw [shift={(554.4,232.8)}, rotate = 296.45] [fill={rgb, 255:red, 0; green, 0; blue, 0 }  ][line width=0.08]  [draw opacity=0] (3.57,-1.72) -- (0,0) -- (3.57,1.72) -- cycle    ;
\draw [shift={(574.2,193)}, rotate = 116.45] [fill={rgb, 255:red, 0; green, 0; blue, 0 }  ][line width=0.08]  [draw opacity=0] (3.57,-1.72) -- (0,0) -- (3.57,1.72) -- cycle    ;
%Straight Lines [id:da6988103094670276] 
\draw    (592.66,155.89) -- (575.54,190.31) ;
\draw [shift={(574.2,193)}, rotate = 296.45] [fill={rgb, 255:red, 0; green, 0; blue, 0 }  ][line width=0.08]  [draw opacity=0] (3.57,-1.72) -- (0,0) -- (3.57,1.72) -- cycle    ;
\draw [shift={(594,153.2)}, rotate = 116.45] [fill={rgb, 255:red, 0; green, 0; blue, 0 }  ][line width=0.08]  [draw opacity=0] (3.57,-1.72) -- (0,0) -- (3.57,1.72) -- cycle    ;
%Straight Lines [id:da9368480887871207] 
\draw    (390,144.8) -- (390,154.8) ;
%Straight Lines [id:da5119927948259866] 
\draw    (589.6,145.2) -- (589.6,155.2) ;

% Text Node
\draw (522.4,110.6) node [anchor=north west][inner sep=0.75pt]  [font=\scriptsize,color={black}  ,opacity=1 ] [align=left] {$\displaystyle \Gamma_{\theta }^{R}$};
% Text Node
\draw (579.6,114.8) node [anchor=north west][inner sep=0.75pt]  [font=\scriptsize,color={black}  ,opacity=1 ] [align=left] {$\displaystyle \partial_{*} K_{\theta }^{R}$};
% Text Node
\draw (447.6,201) node [anchor=north west][inner sep=0.75pt]  [font=\tiny] [align=left] {$\displaystyle \frac{2R}{\tan \theta }$};
% Text Node
\draw (566.4,212.6) node [anchor=north west][inner sep=0.75pt]  [font=\tiny] [align=left] {$\displaystyle R$};
% Text Node
\draw (587.3,172.5) node [anchor=north west][inner sep=0.75pt]  [font=\tiny] [align=left] {$\displaystyle R$};
% Text Node
\draw (480.4,152.4) node [anchor=north west][inner sep=0.75pt]  [font=\tiny] [align=left] {$\displaystyle 0$};
% Text Node
\draw (367,152) node [anchor=north west][inner sep=0.75pt]  [font=\tiny] [align=left] {$\displaystyle \frac{-R}{\sin \theta }$};
% Text Node
\draw (561,151.2) node [anchor=north west][inner sep=0.75pt]  [font=\tiny] [align=left] {$\displaystyle \frac{R}{\sin \theta }$};
% Text Node
\draw (404.8,142.4) node [anchor=north west][inner sep=0.75pt]  [font=\tiny] [align=left] {$\displaystyle \theta $};
% Text Node
\draw (403.6,150) node [anchor=north west][inner sep=0.75pt]  [font=\tiny] [align=left] {$\displaystyle \theta $};
% Text Node
\draw (504.4,150) node [anchor=north west][inner sep=0.75pt]  [font=\tiny] [align=left] {$\displaystyle \theta $};
% Text Node
\draw (504,142.4) node [anchor=north west][inner sep=0.75pt]  [font=\tiny] [align=left] {$\displaystyle \theta $};
% Text Node
\draw (598,152) node [anchor=north west][inner sep=0.75pt]  [font=\tiny] [align=left] {$\displaystyle x_{1}$};
% Text Node
\draw (478,62) node [anchor=north west][inner sep=0.75pt]  [font=\tiny] [align=left] {$\displaystyle x_{2}$};
% Text Node
\draw (411,84) node [anchor=north west][inner sep=0.75pt]  [font=\scriptsize] [align=left] {$\displaystyle K_{\theta }^{R}$};
\end{tikzpicture}}
 &\qquad
& \scalebox{0.8}{\begin{tikzpicture}[x=0.9pt,y=0.9pt,yscale=-1,xscale=1]
%uncomment if require: \path (0,300); %set diagram left start at 0, and has height of 300

%Shape: Polygon [id:ds9292823006865445] 
\draw   (210,70) -- (250,150) -- (210,230) -- (50,150) -- cycle ;
%Shape: Polygon [id:ds555007325893604] 
\draw   (200,150) -- (180,190) -- (100,150) -- (180,110) -- cycle ;
%Shape: Polygon [id:ds1806222902663457] 
\draw  [draw opacity=0][fill={rgb, 255:red, 155; green, 155; blue, 155 }  ,fill opacity=0.3 ] (210,70) -- (240,130) -- (250,150) -- (240,170) -- (200,150) -- (170,90) -- cycle ;
%Shape: Polygon [id:ds09505520836970116] 
\draw  [draw opacity=0][fill={rgb, 255:red, 155; green, 155; blue, 155 }  ,fill opacity=0.3 ] (170,90) -- (180,110) -- (100,150) -- (180,190) -- (200,150) -- (240,170) -- (210,230) -- (188.23,219.11) -- (50,150) -- cycle ;
%Straight Lines [id:da9527972744831095] 
\draw    (195.99,152.02) -- (178.67,186.65) ;
\draw [shift={(177.33,189.33)}, rotate = 296.57] [fill={rgb, 255:red, 0; green, 0; blue, 0 }  ][line width=0.08]  [draw opacity=0] (3.57,-1.72) -- (0,0) -- (3.57,1.72) -- cycle    ;
\draw [shift={(197.33,149.33)}, rotate = 116.57] [fill={rgb, 255:red, 0; green, 0; blue, 0 }  ][line width=0.08]  [draw opacity=0] (3.57,-1.72) -- (0,0) -- (3.57,1.72) -- cycle    ;
%Straight Lines [id:da588071368987579] 
\draw    (251.66,153.68) -- (214.34,228.32) ;
\draw [shift={(213,231)}, rotate = 296.57] [fill={rgb, 255:red, 0; green, 0; blue, 0 }  ][line width=0.08]  [draw opacity=0] (3.57,-1.72) -- (0,0) -- (3.57,1.72) -- cycle    ;
\draw [shift={(253,151)}, rotate = 116.57] [fill={rgb, 255:red, 0; green, 0; blue, 0 }  ][line width=0.08]  [draw opacity=0] (3.57,-1.72) -- (0,0) -- (3.57,1.72) -- cycle    ;
%Straight Lines [id:da06607533706424962] 
\draw    (150,240) -- (150,63) ;
\draw [shift={(150,60)}, rotate = 90] [fill={rgb, 255:red, 0; green, 0; blue, 0 }  ][line width=0.08]  [draw opacity=0] (3.57,-1.72) -- (0,0) -- (3.57,1.72) -- cycle    ;
%Straight Lines [id:da064606250032063] 
\draw    (30,150) -- (267,150) ;
\draw [shift={(270,150)}, rotate = 180] [fill={rgb, 255:red, 0; green, 0; blue, 0 }  ][line width=0.08]  [draw opacity=0] (3.57,-1.72) -- (0,0) -- (3.57,1.72) -- cycle    ;
%Curve Lines [id:da8463122572147715] 
\draw    (69.67,140) .. controls (73.67,145.83) and (74.33,152.83) .. (69.67,160) ;
%Curve Lines [id:da46369634332554266] 
\draw    (120.67,140) .. controls (124.67,145.83) and (125.33,152.83) .. (120.67,160) ;
%Straight Lines [id:da8018198214336886] 
\draw [line width=1.25]   (190,130) -- (230,110) ;
%Straight Lines [id:da7101644091675984] 
\draw [line width=1.25]   (190,170) -- (204.69,177.34) -- (230,190) ;

% Text Node
\draw (67.29,81.43) node [anchor=north west][inner sep=0.75pt]  [font=\scriptsize] [align=left] {$\displaystyle K_{\theta }^{R,r}$};
% Text Node
\draw (173.9,163.81) node [anchor=north west][inner sep=0.75pt]  [font=\tiny] [align=left] {$\displaystyle 2r$};
% Text Node
\draw (141.67,151.33) node [anchor=north west][inner sep=0.75pt]  [font=\tiny] [align=left] {$\displaystyle 0$};
% Text Node
\draw (236,185.33) node [anchor=north west][inner sep=0.75pt]  [font=\tiny] [align=left] {$\displaystyle 2R$};
% Text Node
\draw (115.14,142.29) node [anchor=north west][inner sep=0.75pt]  [font=\tiny] [align=left] {$\displaystyle \theta $};
% Text Node
\draw (114.57,151.43) node [anchor=north west][inner sep=0.75pt]  [font=\tiny] [align=left] {$\displaystyle \theta $};
% Text Node
\draw (64,142) node [anchor=north west][inner sep=0.75pt]  [font=\tiny] [align=left] {$\displaystyle \theta $};
% Text Node
\draw (64.29,150.57) node [anchor=north west][inner sep=0.75pt]  [font=\tiny] [align=left] {$\displaystyle \theta $};
% Text Node
\draw (195.43,101.86) node [anchor=north west][inner sep=0.75pt]  [font=\scriptsize,color={black}  ,opacity=1 ] [align=left] {$\displaystyle \Gamma_{\mathtt{\theta }}^{R,r}$};
% Text Node
\draw (208,161.29) node [anchor=north west][inner sep=0.75pt]  [font=\scriptsize,color={black}  ,opacity=1 ] [align=left] {$\displaystyle \Gamma_{\mathtt{\theta }}^{R,r}$};
% Text Node
\draw (261,151) node [anchor=north west][inner sep=0.75pt]  [font=\tiny] [align=left] {$\displaystyle x_{1}$};
% Text Node
\draw (138,51) node [anchor=north west][inner sep=0.75pt]  [font=\tiny] [align=left] {$\displaystyle x_{2}$};
\end{tikzpicture}}\\
(a) && (b)
    \end{tabular}
\caption{(a) Kite $K^R_\theta$ and its boundary part $\partial_* K^R_\theta$ in dashed lines . (b) The complement of two kites $K_\theta^{R,r}$. }\label{fig: Karl minus Karl}
\end{figure}
Next, we introduce the sesquilinear forms 
    \begin{align}\label{ope straight}
    \begin{split}
         d^R_{\theta,\alpha} (u) &= \int_{K^R_\theta} |\nabla u|^2\mathrm{d}x - \alpha \int_{\Gamma_\theta^{R} } |u|^2 \dS, \quad D(d^R_\theta) = H^1_0(K^R_\theta),\\
      n^R_{\theta,\alpha} (u) &= \int_{K^R_\theta} |\nabla u|^2\mathrm{d}x - \alpha \int_{\Gamma_\theta^{R} } |u|^2 \dS, \quad D(n^R_\theta) = H^1(K^R_\theta),\\
      p^{R,r}_{\theta,\alpha}(u) &=  \int_{K^{R,r}_\theta} |\nabla u|^ 2 \mathrm{d}x - \alpha \int_{\Gamma^{R,r}_\theta} |u|^2 \dS, \quad D(p^{R,r}_{\theta,\alpha}) = H^1(K^{R,r}_\theta).
      \end{split}
     \end{align}
In the following, we will focus on the asymptotic properties of the operators $D^R_{\theta,\alpha}$, $ N^R_{\theta,\alpha} $, and $P^{R,r}_{\theta,\alpha}$.

We begin by establishing a lower bound for the operator $P^{R,r}_{\theta,\alpha}$.
\begin{lemma} \label{lem: P R,r,theta}
    There exists $C>0$ such that, for $\alpha r $ sufficiently large, one has 
    \[
    P^{R,r}_{\theta, \alpha} \geq  \alpha^2(-\frac{1}{4} - C e^{-\frac{1}{2} r \alpha }).
    \]
\end{lemma}
\begin{proof}
\tikzset{every picture/.style={line width=0.75pt}} %set default line width to 0.75pt        
 We start by dividing $K^{R,r}_\theta$ as shown in Figure \ref{fig: Karl minus Karl divided}. That is, we set
    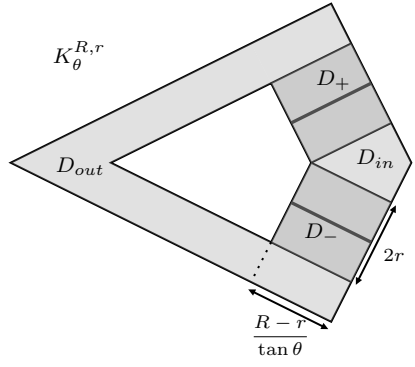
\begin{figure}
    \centering
    \begin{tikzpicture}[x=0.75pt,y=0.75pt,yscale=-1,xscale=1]
%uncomment if require: \path (0,244); %set diagram left start at 0, and has height of 244

%Straight Lines [id:da7427049453618583] 
\draw [line width=1.25]   (170,110) -- (210,90) ;
%Straight Lines [id:da5003422063509453] 
\draw [line width=1.25]   (170,150) -- (210,170) ;
%Shape: Polygon [id:ds05909656057376267] 
\draw   (190,50) -- (230,130) -- (190,210) -- (30,130) -- cycle ;
%Shape: Polygon [id:ds600917428689328] 
\draw   (180,130) -- (160,170) -- (80,130) -- (160,90) -- cycle ;
%Straight Lines [id:da22313640640937504] 
\draw    (160,170) -- (200,190) ;
%Straight Lines [id:da06837089621935843] 
\draw    (180,130) -- (220,150) ;
%Straight Lines [id:da534848818783287] 
\draw    (160,90) -- (200,70) ;
%Straight Lines [id:da1508418440055107] 
\draw    (180,130) -- (220,110) ;
%Straight Lines [id:da8952642159629233] 
\draw    (151.25,195.06) -- (185.89,212.37) ;
\draw [shift={(188.57,213.71)}, rotate = 206.57] [fill={rgb, 255:red, 0; green, 0; blue, 0 }  ][line width=0.08]  [draw opacity=0] (3.57,-1.72) -- (0,0) -- (3.57,1.72) -- cycle    ;
\draw [shift={(148.57,193.71)}, rotate = 26.57] [fill={rgb, 255:red, 0; green, 0; blue, 0 }  ][line width=0.08]  [draw opacity=0] (3.57,-1.72) -- (0,0) -- (3.57,1.72) -- cycle    ;
%Straight Lines [id:da5223495344039888] 
\draw    (203.63,188.75) -- (220.94,154.11) ;
\draw [shift={(222.29,151.43)}, rotate = 116.57] [fill={rgb, 255:red, 0; green, 0; blue, 0 }  ][line width=0.08]  [draw opacity=0] (3.57,-1.72) -- (0,0) -- (3.57,1.72) -- cycle    ;
\draw [shift={(202.29,191.43)}, rotate = 296.57] [fill={rgb, 255:red, 0; green, 0; blue, 0 }  ][line width=0.08]  [draw opacity=0] (3.57,-1.72) -- (0,0) -- (3.57,1.72) -- cycle    ;
%Shape: Polygon [id:ds9438664612257839] 
\draw  [draw opacity=0][fill={rgb, 255:red, 155; green, 155; blue, 155 }  ,fill opacity=0.3 ] (190,50) -- (230,130) -- (220,150) -- (180,130) -- (150,70) -- cycle ;
%Shape: Polygon [id:ds11017577525517763] 
\draw  [draw opacity=0][fill={rgb, 255:red, 155; green, 155; blue, 155 }  ,fill opacity=0.3 ] (150,70) -- (160,90) -- (80,130) -- (160,170) -- (180,130) -- (220,150) -- (190,210) -- (30,130) -- cycle ;
%Shape: Polygon [id:ds933197743523264] 
\draw  [draw opacity=0][fill={rgb, 255:red, 155; green, 155; blue, 155 }  ,fill opacity=0.3 ] (180,130) -- (220,150) -- (200,190) -- (160,170) -- cycle ;
%Shape: Polygon [id:ds7655148643448793] 
\draw  [draw opacity=0][fill={rgb, 255:red, 155; green, 155; blue, 155 }  ,fill opacity=0.3 ] (200,70) -- (220,110) -- (180,130) -- (160,90) -- cycle ;
%Straight Lines [id:da8712369344466439] 
\draw  [dash pattern={on 0.84pt off 2.51pt}]  (160,170) -- (150,190) ;

% Text Node
\draw (181,82) node [anchor=north west][inner sep=0.75pt]  [font=\scriptsize] [align=left] {$\displaystyle D_{+}$};
% Text Node
\draw (175,159) node [anchor=north west][inner sep=0.75pt]  [font=\scriptsize] [align=left] {$\displaystyle D_{-}$};
% Text Node
\draw (201,122) node [anchor=north west][inner sep=0.75pt]  [font=\scriptsize] [align=left] {$\displaystyle D_{in}$};
% Text Node
\draw (51,126) node [anchor=north west][inner sep=0.75pt]  [font=\scriptsize] [align=left] {$\displaystyle D_{out}$};
% Text Node
\draw (49.6,67.4) node [anchor=north west][inner sep=0.75pt]  [font=\scriptsize] [align=left] {$\displaystyle K_{\theta }^{R,r}$};
% Text Node
\draw (214.57,172.14) node [anchor=north west][inner sep=0.75pt]  [font=\tiny] [align=left] {$\displaystyle 2r$};
% Text Node
\draw (148.24,205.16) node [anchor=north west][inner sep=0.75pt]  [font=\tiny] [align=left] {$\displaystyle \frac{R-r}{\tan \theta }$};

\end{tikzpicture}

    \caption{$K^{R,r}_\theta$ divided into $D_+$, $D_-$, $D_{in}$, $D_{out}$ }
    \label{fig: Karl minus Karl divided}
\end{figure}
\begin{align*}
D_{in} := \left(\frac{R +r}{2\sin \theta},0\right) + K^{\frac{R-r}{2}}_\theta, \quad D_{out} := K^{R,r}_\theta \setminus( D_+ \cup D_- \cup D_{in} ),
\end{align*}
where for $M_\pm=\begin{pmatrix}
        \cos \theta & \mp \sin \theta \\ \pm\sin \theta & \cos \theta
    \end{pmatrix}$, $D_\pm$ are given by
    \[
    D_\pm := \left( \frac{r}{\sin \theta} ,0 \right) +  \left\{ M_\pm \cdot  \begin{pmatrix}
        x_1\\x_2
    \end{pmatrix} : \, (x_1,x_2) \in \left(0, \frac{R-r}{\tan \theta} \right) \times\left(\min\{ 0, \pm 2r \}, \max\{0, \pm2r \} \right) \right\},
    \]
and consider the sesquilinear forms
\begin{align*} 
q_{in/out} (v) &= \int_{D_{in/out}} |\nabla v|^2 \mathrm{d}x, \quad D(q_{in/out}) = H^1(D_{in/out}) ,\\
 q^\alpha_\pm(v) &= \int_{D_\pm} |\nabla v|^2 \mathrm{d}x - \alpha \int_{\frac{r}{\tan \theta} }^{\frac{R}{\tan\theta}} | v(x_2,\pm\tan\theta x_2)|^2 \frac{\mathrm{d}x_2}{\cos\theta}, \quad D(q_\pm^\alpha ) = H^1(D_\pm).
 \end{align*}
By restricting $u$ to $D_X$, $X \in \{ +, -,in,out \}$, we get
 $p^{R,r}_{\theta,\alpha}( u) = q_+^\alpha( u) + q_-^\alpha( u) + q_{in}( u) +q_{out}( u)$, and
\begin{equation} \label{eq: proof lower bound P^R,r_theta}
P^{R,r}_{\theta,1} \geq Q_+ \oplus Q_- \oplus Q_{in} \oplus Q_{out} .
\end{equation}
The operators $Q_+$ and $Q_-$ are unitarily equivalent to $ I \otimes T^N_{\frac{R-r}{2\tan \theta}} \oplus T^N_{r,\alpha} \otimes I $. As the Neumann Laplacian is non-negative, Proposition \ref{lem: ev of t,N,L,alpha} ensures that there exists a $C>0$ such that
\[
E_1(Q^\alpha_\pm) = E_1(I \otimes T^N_{\frac{R-r}{2\tan \theta}} \oplus T^N_{r,\alpha} \otimes I ) > \alpha^2(-\frac{1}{4} - Ce^{-\frac{1}{2}r \alpha } )
\]
holds for $\alpha r > 8$. This together with (\ref{eq: proof lower bound P^R,r_theta}) and
the non-negativity of $Q_{in/out}$ implies the desired lower bound.
\end{proof}
The following corollary is a direct consequence of Lemma \ref{lem: scaling anything}. 
\begin{corollary}\label{lem: scaling kite}
 For $X \in \{ D,N \}$ and any $\alpha>0$, one has $X^R_{\theta,\alpha} \cong \alpha^2 X^{\alpha R}_{\theta,1}$.
\end{corollary}
The next lemma addresses the eigenvalue asymptotics of $D^R_{\theta,\alpha}$. Recall that $\kappa(\theta)$ is defined by \eqref{eigenvaluenumber}.
\begin{lemma} \label{lem: eigenvalues of dirichlet on karl} 
    There exists $c>0$ such that,  as $\alpha R \to \infty$, the following hold:
    \begin{itemize}
        \item[(i)] $ E_n(D^R_{\theta,\alpha})= \alpha^2(\mathcal{E}_n (\theta) + \mathcal{O} (e^{-c\alpha R}))$ for any $ n \in\{ 1, \dots,\kappa(\theta) \}$.
        \item[(ii)] $E_{\kappa(\theta) + 1} (D^R_{\theta,\alpha} ) \geq - \frac{\alpha^2}{4} $.
    \end{itemize}
\end{lemma}
\begin{proof}
    In view of Corollary \ref{lem: scaling kite}, it suffices to consider the case $\alpha = 1$ and $R \to \infty$. Using the monotonicity of $D^R_{\theta,1}$, it follows that $\Lambda_n(D^R_{\theta,1}) \geq \Lambda_n(H^1_{\theta})$ for any $n \in \nn $, and consequently, the following statements hold:
    \begin{enumerate}
        \item[(a)] $E_n(D^R_{\theta,1}) \geq \Lambda_n(H^1_{\theta}) = \mathcal{E}_n(\theta) $ for $n \in \{1,\dots, \kappa(\theta) \}$,
        \item[(b)] $E_{\kappa(\theta) +1 }(D^R_{\theta,1}) \geq \Lambda_{\kappa(\theta)+1} (H^1_\theta) = \inf \spec_{ess}(H^1_\theta) = -\frac{1}{4} $.
    \end{enumerate}
    Note that (a) gives a lower bound to $E_n(D^R_{\theta,1})$ for $n \in \{ 1,\dots,\kappa(\theta)\}$, while (b) establishes the second statement of the lemma. Hence, it remains to prove that there exist $c,C > 0$ such that, for any $n\in \{1,\dots,\kappa(\theta) \}$, there holds $ E_n(D^R_{\theta,1}) \leq \mathcal{E}_n(\theta) + C e^{-cR}$.  To this end, we apply an IMS partition and use the Agmon-type estimate from Proposition \ref{lem: big prop}. Let $\chi_0,\chi_1 \in C^\infty(\rr_+)$ be such that 
    \[
    \chi_0^2+\chi_1^2 \equiv 1,\quad \chi_0(t) =1 \text{ if } t\in[0,\frac{1}{2}], \quad \chi_0(t) =0 \text{ if } t \in [1,\infty).
    \]
    For a fixed $n \in \{1,\dots, \kappa (\theta) \}$, let $\psi_1,\dots,\psi_n$ be the orthonormal eigenfunctions corresponding to the first $n$ eigenvalues $\mathcal{E}_1, \dots,\mathcal{E}_n$ of $H^1_\theta$. We then define
    \[
    \chi_j^R(x) = \chi_j\left(\frac{| x |}{R}\right) \text{ for } j = 0,1, \quad \psi_i^R = \chi_0^R\psi_i, \,\,\,\text{and}\,\,\, \tilde{\psi}_i^R = \psi_i^R\downharpoonright_{K^R_\theta} \quad \text{for }i = 1, \dots, n.
    \]
    Obviously, for each $i\in\{ 1, \dots, n\}$, we have $\tilde{\psi}_i^R\in H^1_0(K^R_\theta) = D (d^R_{\theta,1})$, since the support of ${\psi}_i^R$ is contained in $\{ | x | \leq R\}$. We next show that $L_R := span (\tilde{\psi}_1^R,\dots , \tilde{\psi}^R_n)$ is an n-dimensional subspace of $H^1_0(K^R_\theta)$. By Proposition \ref{lem: big prop} there are $b,B > 0 $ such that
    \[
    \int_{\rr^2} e^{b| x |} ( |\nabla \psi_j|^2 +  |\psi_j |^2) \mathrm{d} x \leq B \|\psi_j\|^2_{L^2(\rr^2)} \quad \text{for all } j\in \{1,\dots,n \}.
    \]
   Set
     \[
     C^R_{j,k} := \int_{\rr^2} (\chi_1^R)^2 \psi_j \psi_k \mathrm{d} x,
     \]
    then $|C^R_{j,k}| \leq \frac{1}{2}(C^R_{j,j} + C^R_{k,k})$. This, together with the above Agmon estimate, yields 
    \[
    C^R_{k,k} = \int_{\rr^2}(\chi_1^R) \psi_j^2 \mathrm{d}x \leq \int_{| x | > \frac{r}{2}} \psi_j^2 \mathrm{d}x \leq e^{-\frac{bR}{2}} \int_{| x | > \frac{r}{2}} e^{b | x |} \psi_j^2 \mathrm{d}x \leq Be^{-\frac{bR}{2}}.
    \]
    From this, it follows that $C^R_{j,k} = \mathcal{O}(e^{-cr})$ for $c = \frac{b}{2}$, and that 
    \[
    \langle \tilde{\psi}_j^R,  \tilde{\psi}_k^R \rangle_{L^2(K^R_\theta)} = \langle \psi_j,\psi_k \rangle_{L^2(\rr^2)} - C^r_{j,k} = \delta_{j,k} + \mathcal{O}(e^{-cr}).
    \]
    Assume that $L_R$ is not an n-dimensional subspace of $H^1_0(K^R_\theta)$; that is, there exist $ \lambda_1,\dots ,\lambda_n \in \mathbb{C}$ with $\lambda_1 \neq 0$ such that $\lambda_1 \tilde{\psi}^R_1 + \dots \lambda_n \tilde{\psi}^R_n = 0$. Then, for $R$ sufficiently large, one has
    \begin{align*}
    0 &= \langle  \lambda_1 \tilde{\psi}^R_1 + \dots \lambda_n \tilde{\psi}^R_n , \lambda_1 \tilde{\psi}^R_1 + \dots \lambda_n \tilde{\psi}^R_n  \rangle 
    = \sum_{j,k = 1}^n \lambda_j \lambda_k \langle \tilde{\psi}^R_j, \tilde{\psi}^R_k \rangle \\
   & = (\sum_{j = 1}^n | \lambda_j|^2) (1+ \mathcal{O}(e^{-cR})) + (\sum_{j, k = 1, j \neq k }^n \overline{\lambda_j} \lambda_k ) \mathcal{O}(e^{-cR}) \neq 0 
    \end{align*}
    which leads to a contradiction and proves that $L_R$ has to be an n-dimensional subspace of $H^1_0(K^R_\theta)$. 
    
    Next, we claim that $ d^R_{\theta,1}(\tilde{\psi}^R_j, \tilde{\psi}^R_k) = h^1_{\theta}({\psi}^R_j, {\psi}^R_k) + \mathcal{O}(e^{-cR})$, so that we can use the Min-Max principle. Indeed, applying the IMS formula gives 
      \begin{align*}
      d^R_{\theta, 1} (\psi_j^R) &= h^1_\theta(\psi_j, \psi_k) - 
     \int_{\rr^2} \langle {\nabla(\chi_1^R \psi_j ), \nabla(\chi_1^R \psi_k)} \rangle  \mathrm{d}x
     + \int_{\Gamma_\theta} (\chi_1^R)^2 \psi_j\psi_k
     + \int_{\rr^2} (| \nabla \chi_0^R |^2 + | \nabla \chi_1^R|^2 ) \psi_j \psi_k \mathrm{d}x\\
     &=: h^1_\theta(\psi_j, \psi_k) - A_{j,k} + B_{j,k}+ D_{j,k}.
     \end{align*}
     Let us show that $A_{j,k},B_{j,k}, D_{j,k} = \mathcal{O}(e^{-cR})$.  Observe that
     \[
     |D_{j,k}| 
     \leq \int_{\rr^2} | |\nabla \chi_0^R|^2 + | \nabla \chi_1^R|^2 |  |\psi_j \psi_k| \mathrm{d}x
     \leq \frac{c_2}{R^2} \int_{| x |> \frac{R}{2}} | \psi_j |^2 + | \psi_k |^2 \mathrm{d}x \leq \frac{2c_2B}{R^2}e^{-cR} = \mathcal{O}(e^{-cR}).
     \]
    For $A_{j,k}$, we have $ |A_{j,k}| \leq \frac{1}{2}(A_{j,j} + A_{k,k}) $ and
     \begin{align*} 
     |A_{j,j}| 
     = \| \nabla\chi_1^R \psi_j + \chi_1^R \nabla \psi_j\|^2_{L^2(\rr^2)}
     \leq 2\| \nabla\chi_1^R \psi_j\|^2_{L^2(\rr^2)} + 2\| \chi_1^R \nabla \psi_j \|^2_{L^2(\rr^2)}
     \\ \leq 2(\frac{c_1}{R^2} + 1) \int_{| x | > \frac{R}{2}} | \psi_j|^2 + |\nabla \psi_j|^2 \mathrm{d}x \leq 2(\frac{c_1}{R^2} + 1)B e^{-cR} = \mathcal{O}(e^{-cR}).
     \end{align*}
     Similarly, we have $ | B_{j,k} | \leq \frac{1}{2}(B_{j,j} + B_{k,k}) $, and
    $$ B^R_{j,j} \leq \int_{\rr^2} | \nabla(\chi_1^R\psi_j) |^2 \mathrm{d}x + \frac{1}{4\sin^2\theta} \int_{\rr^2} |\chi_1^R \psi_j |^2 \mathrm{d}x = \mathcal{O}(e^{-cR}), $$
    and the claimed identity for $ d^R_{\theta,1}(\tilde{\psi}^R_j, \tilde{\psi}^R_k)$ follows. Using this, for  $\psi = \lambda_1 \tilde{\psi}_1^R + \dots \lambda_n \tilde{\psi}_n^R \in L_R$ with some $\lambda_j \in \mathbb{C} $,  we get
    \begin{align*}
    d^R_{\theta,1} (\psi) &= \sum_{j,k = 1}^n (H^1_\theta(\psi_j,\psi_k) + \mathcal{O}(e^{-cR})) \lambda_j\lambda_k = \sum_{j,k = 1}^n (\mathcal{E}_j(\theta) \delta_{j,k} + \mathcal{O}(e^{-cR}) \lambda_j\lambda_k )\leq (\mathcal{E}_n(\theta) + \mathcal{O}(e^{-cR})) |\lambda|^2 ,
    \end{align*}
    and since $\| \psi \| = | \lambda |^2(1+ \mathcal{O}(e^{-cR})) $, the Min-Max principle implies that
    \[
    E_n(D^R_{\theta,1}) \leq \sup_{0 \neq\psi \in L_R} \frac{d^R_{\theta,1}(\psi)}{\| \psi\|^2_{L^2(\rr^2)}} \leq \mathcal{E}_n(\theta) + \mathcal{O}(e^{-cR}),
    \]
    which finishes the proof
\end{proof}
We next state the analogous result for $N^R_{\theta,\alpha}$.
\begin{lemma} \label{lem: eigenvalues of neumann on karl}
    As $R \to 0$, $\alpha \to \infty$, and  $\alpha R \to \infty$, the following hold:
    \begin{itemize}
        \item[(i)] $E_n(N^R_{\theta,\alpha}) = \alpha^2(\mathcal{E}_n(\theta) + \mathcal{O}(\frac{1}{(\alpha R)^2}) )$ for $n \in \{1,\dots,\kappa(\theta) \}$.
        \item[(i)] $ E_{\kappa(\theta) +1} (N^R_{\theta,\alpha}) \geq -\frac{\alpha^2}{4} + o(\alpha^2) $.
    \end{itemize}
\end{lemma}
\begin{proof}
By the Dirichlet-Neumann monotonicity, one has the upper bound
 \[
 E_n(N^R_{\theta,\alpha}) \leq E_n(D^R_{\theta,\alpha}) \leq \alpha^2(\mathcal{E}_n(\theta) + Ce^{-cR}).
 \]
 For the lower bounds in (i) and (ii), it is again sufficient to consider the case $\alpha = 1$. Let $ \chi_0$, $\chi_1$, and $\chi_j^R$ be as in the proof of Lemma \ref{lem: eigenvalues of dirichlet on karl}. By Lemma \ref{lem: IMS partition}, we have 
\[
n^R_{\theta,\alpha} (u) \geq n^R_{\theta,\alpha}(\chi_0^R u) + n^R_{\theta,\alpha}(\chi_1^R u) 
- \frac{A}{R^2} \| u \|^2_{K^R_\theta} 
= d^R_{\theta,\alpha} (\chi_0^R u ) + p^{R,\frac{R}{2}}_{\theta,\alpha} (\chi_1^R u) -  \frac{A}{R^2} \| u \|^2_{K^R_\theta}
\]
with $A =\| \chi_0' \|^2_\infty + \| \chi_1'\|^2_\infty<\infty$. Consider the map
\[ 
J: D(n^R_{\theta,\alpha}) \rightarrow D(d^R_{\theta,\alpha}) \oplus D(p^{R,\frac{R}{2}}_\theta), \quad u \mapsto (\chi_0^R u, \chi_1^Ru). \]
Since
\begin{gather*}
\| Ju \|^2_{L^2(K^R_\theta)} = \| \chi_0^R u \|^2_{L^2(K^R_\theta)} + \| \chi_1^R u\|^2_{L^2\left(K^{R,\frac{R}{2}}_\theta\right)} = \| u \|^2_{L^2(K^R_\theta)},\\
\big(n^R_{\theta,\alpha} + \frac{A}{R^2}\big) (u) \geq (d^R_{\theta,\alpha} \oplus p^{R,\frac{R}{2}}_{\theta}) (J u),
\end{gather*}
we deduce that $ E_n(N^R_{\theta,\alpha}) \geq E_n(D^R_\theta\oplus P^{R,\frac{R}{2}}_{\theta,\alpha}) - \frac{A}{R^2}$ for any $n\in \nn$. Hence, for any fixed $n \in \{ 1, \dots,\kappa(\theta) \}$, the lower bound for $P^{R,\frac{R}{2}}_{\theta,\alpha}$ from Lemma \ref{lem: P R,r,theta} gives
\[
E_n(D^R_{\theta,\alpha} \oplus P^{R,\frac{R}{2}}_{\theta,\alpha}) = E_n(D^R_{\theta,\alpha}) = \alpha^2(\mathcal{E}_n(\theta) + \mathcal{O}(e^{-c\alpha R})),
\]
and therefore
\[
E_n(N^R_{\theta,\alpha}) \geq \alpha^2(\mathcal{E}_n(\theta) + \mathcal{O}(e^{-c\alpha R}) - \frac{A}{(\alpha R)^2} ) = \alpha^2(\mathcal{E}_n(\theta) + \mathcal{O}(\frac{1}{(\alpha R)^2})),
\]
which completes the proof of (i). Finally, for $n = \kappa(\theta) + 1$, By Lemma \ref{lem: P R,r,theta} and Lemma \ref{lem: eigenvalues of dirichlet on karl}(ii) we get
\[
E_{\kappa(\theta) + 1} (N^R_{\theta,\alpha}) \geq \min \left\{ E_{\kappa(\theta) + 1}(D^R_{\theta,\alpha}), E_1(P^{R,\frac{R}{2}}_{\theta,\alpha})) \right\} - \frac{A}{R^2} = -\frac{\alpha^2}{4}+ o(\alpha^2).
\]
This shows (ii) and achieves the proof.
\end{proof}
\subsection{The non-resonance condition}
While Lemma \ref{lem: eigenvalues of neumann on karl} shows that the first $\kappa(\theta)$ eigenvalues of $N^R_{\theta,\alpha}$ are close to those of $H^\alpha_\theta$, we shall introduce a non-resonance condition that imposes a restriction on the asymptotic behavior of the subsequent eigenvalue.
\begin{definition}\label{def non_resonant}
    We say that a half-angle $\theta \in (0, \frac{\pi}{2}) $ is non-resonant if there exists $C>0$ such that 
    $$ E_{\kappa(\theta)+1}(N^R_{\theta,\alpha}) \geq -\frac{\alpha^2}{4} + \frac{C}{R^2}, $$
    for $\alpha > 0$ fixed and $R \to \infty$. We also say that $\theta \in (\frac{\pi}{2},\pi)$ is non-resonant if $\pi-\theta$ is non-resonant.
\end{definition}
Note that, by Lemma \ref{lem: scaling kite}, the non-resonance condition is equivalent to requiring that
\[
    E_{\kappa(\theta)+1} (N^R_{\theta,\alpha}) = \alpha^2 E_{\kappa(\theta) + 1} (N^{\alpha R}_{\theta,1}) \geq \alpha^2\left(-\frac{1}{4} + \frac{C}{(\alpha R)^2}\right) = -\frac{\alpha^2}{4} + \frac{C}{R^2}
\]
    as $\alpha R \to \infty$.
\begin{lemma}
\label{lem: existence of non resonant angles}
All half-angles $\theta \in [\frac{\pi}{4},\frac{\pi}{2})$ are non-resonant with $\kappa(\theta)=1$.
\end{lemma}
\begin{proof}Without loss of generality we may assume that $\alpha = 1$ . Because of Proposition \ref{lem: big prop}, we know that $\kappa(\theta)=1$ holds for all $\theta \in [\frac{\pi}{4}, \frac{\pi}{2})$. Therefore, it suffices to show that there exist constants $C>0$ and sufficiently large $R$ such that
\[
E_2(N^R_{\theta,1}) \geq -\frac{1}{4} + \frac{C}{R^2}. 
\]
We first consider the case $\theta = \frac{\pi}{4}$. Note that in this case, one has $K^R_\theta = (-R,R)^2$, and $\Gamma^R_\theta$ coincides with the positive parts of the $x_1$- and $x_2$-axes. This geometric setting allows us to estimate, for any  $u \in H^1(K^R_\theta)$
\[
n^R_{\frac{\pi}{4},1}(u) \geq \int_{(-R,R)^2} |\nabla u|^2 \mathrm{d}x - \int_{-R}^R (|u(x_1,0)|^2 + |u(0,x_1)|^2) \mathrm{d} x_1 =: q^R(u),
\]
where $D(q^R) = D(n^R_{ \frac{\pi}{4},1 } ) = H^1((-R,R)^2)$. Clearly, $Q^R$ is unitarily equivalent to $ I \otimes T^N_{R,1} + T^N_{R,1} \otimes I $, and for some $C_0 > 0 $ one has
\begin{align*}
E_2(N^R_{\frac{\pi}{4},1}) \geq E_2(Q^R) &= E_1(T^N_{R,1}) + E_2(T^N_{R,1})> -\frac{1}{4} + \mathcal{O}(e^{-\frac{R}{2}}) + \frac{\pi^2}{4R^2}\geq -\frac{1}{4} + \frac{C_0}{R^2},
\end{align*}
 which shows the statement for $\theta = \frac{\pi}{4}$.

\begin{figure}
         \centering
         \begin{tikzpicture}[x=0.75pt,y=0.75pt,yscale=-1,xscale=1]
%uncomment if require: \path (0,300); %set diagram left start at 0, and has height of 300

%Straight Lines [id:da884270810318179] 
\draw    (170,180) -- (170,63) ;
\draw [shift={(170,60)}, rotate = 90] [fill={rgb, 255:red, 0; green, 0; blue, 0 }  ][line width=0.08]  [draw opacity=0] (3.57,-1.72) -- (0,0) -- (3.57,1.72) -- cycle    ;
%Straight Lines [id:da02491481181076094] 
\draw    (50,160) -- (287,160) ;
\draw [shift={(290,160)}, rotate = 180] [fill={rgb, 255:red, 0; green, 0; blue, 0 }  ][line width=0.08]  [draw opacity=0] (3.57,-1.72) -- (0,0) -- (3.57,1.72) -- cycle    ;
%Curve Lines [id:da8331822634973286] 
\draw    (190,150) .. controls (194.33,152.5) and (194.33,156) .. (193,159.83) ;
%Straight Lines [id:da4096900193194116] 
\draw    (270.33,155) -- (270.33,165) ;
%Straight Lines [id:da2322761369049574] 
\draw    (70,154.67) -- (70,164.67) ;
%Shape: Polygon [id:ds7840363704381192] 
\draw  [fill={rgb, 255:red, 155; green, 155; blue, 155 }  ,fill opacity=0.3 ] (230,80) -- (270,160) -- (70,160) -- cycle ;
%Straight Lines [id:da0029485881413812143] 
\draw  [line width=1.25]   (170,160) -- (250,120) ;
%Straight Lines [id:da47190209342821166] 
\draw [densely dotted] [line width=1.25]     (70,160) -- (270,160) ;

% Text Node
\draw (184.33,152) node [anchor=north west][inner sep=0.75pt]  [font=\tiny] [align=left] {$\displaystyle \theta $};
% Text Node
\draw (259.33,168.67) node [anchor=north west][inner sep=0.75pt]  [font=\tiny] [align=left] {$\displaystyle \frac{R}{\sin \ \theta }$};
% Text Node
\draw (59,169.33) node [anchor=north west][inner sep=0.75pt]  [font=\tiny] [align=left] {$\displaystyle \frac{-R}{\sin \ \theta }$};
% Text Node
\draw (283,162.33) node [anchor=north west][inner sep=0.75pt]  [font=\tiny] [align=left] {$\displaystyle x_{1}$};
% Text Node
\draw (156.33,61.33) node [anchor=north west][inner sep=0.75pt]  [font=\tiny] [align=left] {$\displaystyle x_{2}$};
% Text Node
\draw (105.57,87.05) node [anchor=north west][inner sep=0.75pt]  [font=\scriptsize] [align=left] {$\displaystyle K_{\theta }^{R,+}$};
\end{tikzpicture}

         \caption{The triangle $K^{R,+}_\theta$.}
         \label{fig: KRtheta geq 0}
\end{figure}
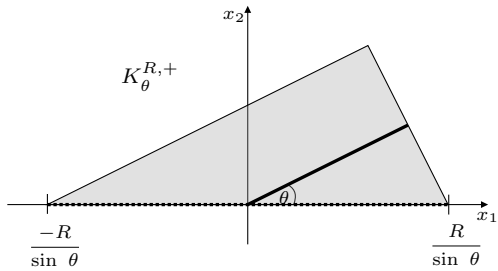
We now turn to the case $\theta \in (\frac{\pi}{4}, \frac{\pi}{2})$. We use the axial symmetry of $K^R_\theta$ by introducing the unitary transform
\begin{align*}
\Phi: L^2(K^R_\theta) &\rightarrow L^2(K^{R,+}_\theta) \oplus   L^2(K^{R,+}_\theta)\\
u  &\mapsto \begin{pmatrix} g\\ h
	\end{pmatrix} := \frac{1}{\sqrt{2}}\begin{pmatrix} u(x_1,x_2)+u(x_1,-x_2) \\ u(x_1,x_2)-u(x_1,-x_2)\end{pmatrix},
\end{align*}
where
\[
K^{R,+}_\theta = K^R_\theta \cap \{(x_1,x_2) : x_2 >0 \},
\]
see Figure \ref{fig: KRtheta geq 0}. Hence, it holds that
\[
n^R_\theta(u) = n^{R,N}_\theta(g) + n_\theta^{R,D}(h) = (n^{R,N}_\theta \oplus n^{R,D}_\theta )(\Phi(u))
\]
where for $\bullet= D,N$, $n_\theta^{R,\bullet}$ is defined by
\[
n_\theta^{R,N}(g) = \int_{K^{R,+}_\theta} |\nabla g|^2 \mathrm{d}x - \int_{0}^{R \cos \theta } |g(\frac{x_2}{\tan\theta}, x_2)|^2 \frac{\mathrm{d}x_2}{\sin\theta},
\]
with $D(n_\theta^{R,N} ) = H^1(K^{R,+}_\theta)$ and $D(n_\theta^{R,N}) = \{ u \in H^1(K^{R,+}_\theta) \mid u( \,\cdot \,,0) = 0 \}$. This proves that $N^R_\theta$ is unitarily equivalent to the direct sum $N^{R,N}_\theta \oplus N^{R,D}_\theta$, which allows us to analyze the spectra of $N^{R,N}_\theta $ and $N^{R,D}_\theta$ separately. 

We begin by establishing a lower bound for the first eigenvalue in the Dirichlet case. By applying Dirichlet bracketing, one obtains the inequality $N^{R,D}_\theta \geq \widetilde{N}^{R,D}_\theta $, where $\widetilde{N}^{R,D}_\theta$ denotes the Laplacian on $\Pi^R_\theta$. The domain $\Pi^R_\theta$ is defined as the rectangle  $(-R,R) \times (\frac{-R}{\tan \theta}, \frac{R}{\tan \theta})$ rotated counterclockwise by angle $\theta$, Dirichlet boundary conditions are imposed on the portion of the boundary below the $x_1$-axis, Neumann boundary conditions on the remaining part of the boundary, and the $\delta$-interaction is supported on the $x_1$-axis; see Figure \ref{fig: rotation with karl}(a). 
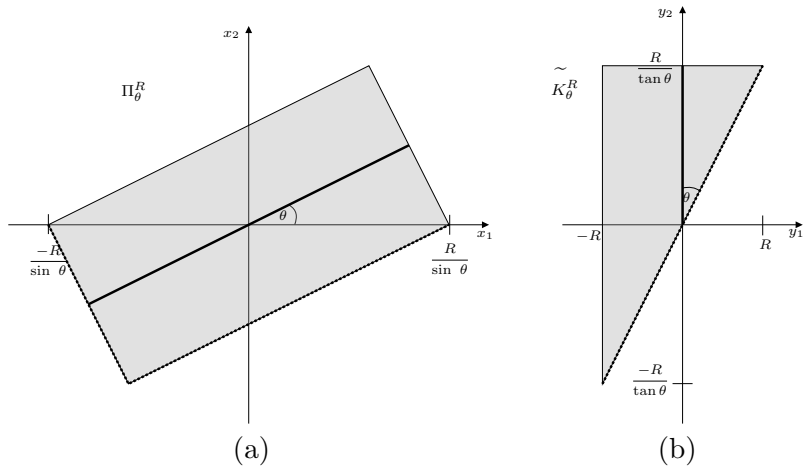
\begin{figure}[b]
\centering
\begin{tabular}{ccc}
         \scalebox{0.75}{\begin{tikzpicture}[x=1pt,y=1pt,yscale=-1,xscale=1]
%uncomment if require: \path (0,273); %set diagram left start at 0, and has height of 273

%Straight Lines [id:da574699052448987] 
\draw    (170,260) -- (170,63) ;
\draw [shift={(170,60)}, rotate = 90] [fill={rgb, 255:red, 0; green, 0; blue, 0 }  ][line width=0.08]  [draw opacity=0] (3.57,-1.72) -- (0,0) -- (3.57,1.72) -- cycle    ;
%Straight Lines [id:da756020215550792] 
\draw    (50,160) -- (287,160) ;
\draw [shift={(290,160)}, rotate = 180] [fill={rgb, 255:red, 0; green, 0; blue, 0 }  ][line width=0.08]  [draw opacity=0] (3.57,-1.72) -- (0,0) -- (3.57,1.72) -- cycle    ;
%Curve Lines [id:da6171016034462887] 
\draw    (190,150) .. controls (194.33,152.5) and (194.33,156) .. (193,159.83) ;
%Straight Lines [id:da5770179079885905] 
\draw    (270.33,155) -- (270.33,165) ;
%Straight Lines [id:da6530569316374979] 
\draw    (70,154.67) -- (70,164.67) ;
%Shape: Polygon [id:ds10368929160306428] 
\draw  [fill={rgb, 255:red, 155; green, 155; blue, 155 }  ,fill opacity=0.3 ] (230,80) -- (270,160) -- (110,240) -- (70,160) -- cycle ;
%Straight Lines [id:da09811792647243633] 
\draw  [line width=1.25]  (90,200) -- (250,120) ;
%Straight Lines [id:da28734092064938144] 
\draw [densely dotted] [line width=1.25]   (70,160) -- (110,240) -- (270,160) ;

% Text Node
\draw (184.33,152) node [anchor=north west][inner sep=0.75pt]  [font=\tiny] [align=left] {$\displaystyle \theta $};
% Text Node
\draw (259.33,168.67) node [anchor=north west][inner sep=0.75pt]  [font=\tiny] [align=left] {$\displaystyle \frac{R}{\sin \ \theta }$};
% Text Node
\draw (59,169.33) node [anchor=north west][inner sep=0.75pt]  [font=\tiny] [align=left] {$\displaystyle \frac{-R}{\sin \ \theta }$};
% Text Node
\draw (283,162.33) node [anchor=north west][inner sep=0.75pt]  [font=\tiny] [align=left] {$\displaystyle x_{1}$};
% Text Node
\draw (156.33,61.33) node [anchor=north west][inner sep=0.75pt]  [font=\tiny] [align=left] {$\displaystyle x_{2}$};
% Text Node
\draw (105.57,87.05) node [anchor=north west][inner sep=0.75pt]  [font=\scriptsize] [align=left] {$\displaystyle \Pi_{\theta }^{R}$};
\end{tikzpicture}}
&\qquad
&
\scalebox{0.75}{\begin{tikzpicture}[x=1pt,y=1pt,yscale=-1,xscale=1]
%uncomment if require: \path (0,273); %set diagram left start at 0, and has height of 273

%Straight Lines [id:da1291958777145592] 
\draw    (127.57,241.67) -- (127.57,34.67) ;
\draw [shift={(127.57,31.67)}, rotate = 90] [fill={rgb, 255:red, 0; green, 0; blue, 0 }  ][line width=0.08]  [draw opacity=0] (3.57,-1.72) -- (0,0) -- (3.57,1.72) -- cycle    ;
%Straight Lines [id:da01759082894789854] 
\draw    (67.57,141.67) -- (184.57,141.67) ;
\draw [shift={(187.57,141.67)}, rotate = 180] [fill={rgb, 255:red, 0; green, 0; blue, 0 }  ][line width=0.08]  [draw opacity=0] (3.57,-1.72) -- (0,0) -- (3.57,1.72) -- cycle    ;
%Shape: Polygon [id:ds28902379054063476] 
\draw  [fill={rgb, 255:red, 155; green, 155; blue, 155 }  ,fill opacity=0.3 ] (87.57,61.67) -- (87.57,221.67) -- (167.57,61.67) -- cycle ;
%Straight Lines [id:da6300318359032451] 
\draw  [line width=1.25]   (127.57,141.67) -- (127.57,61.67) ;
%Straight Lines [id:da5602607924379309] 
\draw    (167.57,137) -- (167.57,147) ;
%Straight Lines [id:da9668384032367252] 
\draw    (122.57,221.57) -- (123.9,221.57) -- (132.57,221.57) ;
%Curve Lines [id:da23763884465807428] 
\draw    (127.32,123.29) .. controls (130.82,122.29) and (134.07,123.04) .. (136.07,124.54) ;
%Straight Lines [id:da3883561247683618] 
\draw [densely dotted] [line width=1.25]    (87.57,221.67) -- (167.57,61.67) ;

% Text Node
\draw (73.57,143.67) node [anchor=north west][inner sep=0.75pt]  [font=\tiny] [align=left] {$\displaystyle -R$};
% Text Node
\draw (164.9,148) node [anchor=north west][inner sep=0.75pt]  [font=\tiny] [align=left] {$\displaystyle R$};
% Text Node
\draw (179.57,142.67) node [anchor=north west][inner sep=0.75pt]  [font=\tiny] [align=left] {$\displaystyle y_{1}$};
% Text Node
\draw (114.24,32) node [anchor=north west][inner sep=0.75pt]  [font=\tiny] [align=left] {$\displaystyle y_{2}$};
% Text Node
\draw (100.76,211.19) node [anchor=north west][inner sep=0.75pt]  [font=\tiny] [align=left] {$\displaystyle \frac{-R}{\tan \theta }$};
% Text Node
\draw (103.24,54) node [anchor=north west][inner sep=0.75pt]  [font=\tiny] [align=left] {$\displaystyle \frac{R}{\tan \theta }$};
% Text Node
\draw (61,67.81) node [anchor=north west][inner sep=0.75pt]  [font=\scriptsize] [align=left] {$\displaystyle K_{\theta }^{R}$};
% Text Node
\draw (62.14,61.67) node [anchor=north west][inner sep=0.75pt]   [align=left] {$\displaystyle \widetilde{\ \ }$};
% Text Node
\draw (126.57,124.42) node [anchor=north west][inner sep=0.75pt]  [font=\tiny] [align=left] {$\displaystyle \theta $};
\end{tikzpicture}}\\
(a) && (b)
\end{tabular}
\caption{(a) Continuation of $K^{R,+}_\theta$ via Dirichlet bracketing. (b) $K^{R,+}_\theta$ rotated counterclockwise by angle $ \frac{\pi}{2} -\theta $.}
\label{fig: rotation with karl}
\end{figure}
After a rotation and separation of variables, it is easy to see that $\widetilde{N}^{R,D}$ is unitarily equivalent to $ I \otimes T^{ND}_{R,1}+T^{ND}_{\frac{R}{\tan \theta}} \otimes I$, and by using Proposition \ref{lem: ev of t,N,L,alpha}(vi) and Lemma \ref{lem: ev of t,ND,L} one gets for $R$ sufficiently large that
\[
E_1(N^{R,D}) \geq E_1(T^{ND}_{R,1}) + E_1(T^{ND}_{\frac{2R}{\tan \theta}}) > - \frac{1} {4} - C e^{-\frac{R}{2}} +  \frac{\pi^2 \tan^2 \theta}{16 R^2} \geq -\frac{1}{4} + \frac{C_D}{R^2},
\]
for some $C_D>0$.  Thus, it remains to establish a lower bound for the Neumann case. To this end, we first apply a counterclockwise rotation by angle $\pi -\theta$  (see Figure \ref{fig: rotation with karl}(b)) to get a unitarily equivalent operator associated with 
\begin{align*}
q^{R}_\theta(g) &= \int_{\tilde{K}_\theta^R}  | \nabla g|^2 \mathrm{d}x - \int_{0}^{\frac{R}{\tan \theta}} |g(0,x_2)|^2 \mathrm{d} x_2
, \quad D(q^R_\theta) = H^1(\tilde{K}^R_\theta),
\end{align*}
where
\begin{align*}
 \tilde{K}^R_\theta &= \Big\{ (y_1,y_2) \in (-R,R) \times (-\frac{R}{\tan \theta}, \frac{R}{\tan \theta}) \mid - y_2 < \frac{y_1}{\tan \theta} \Big\}.
 \end{align*}
We then apply another unitary transform consisting of a scaling by $\tan \theta$:
\[
\Psi:L^2(\tilde{K}^{R \tan \theta}_{\frac{\pi}{4}}) \rightarrow L^2(\tilde{K}^R_{\theta}), \quad g \mapsto \sqrt{\tan \theta} \cdot g(x_1,x_2 \tan\theta),
\]
which leads to 
\[
\tilde{q}^R_\theta (v) = \int_{\tilde{K}^R_{\frac{\pi}{4}}} |\partial_{x_1}v(x_1,t)|^2 + \tan^2\theta |\partial_{t} v(x_1,t)|^2  \mathrm{d}x_1 \text{ d}t - \int_{0}^R |\tilde{g}(0, t)|^2  \mathrm{d} t, \quad D(\tilde{q}_\theta^R) = H^1(\tilde{K}^R_{\frac{\pi}{4}}).
\]
From this, it follows that 
\[
\tilde{q}^R_\theta (v) = \tilde{q}_{\frac{\pi}{4}}^R(v) +(\tan^2\theta - 1) \int_{\tilde{K}^R_{\frac{\pi}{4}}} |\partial_t v(x_1,t) |^2 \mathrm{d} x_1 \text{ d}t \geq \tilde{q}_{\frac{\pi}{4}}^R(v),
\]
and by the Min-Max principle, we obtain
\[
E_2(N^{N,R}_\theta) = E_2(  \tilde{Q}^R_\theta ) \geq E_2(\tilde{Q}^R_{\frac{\pi}{4}}) = E_2(N^{N,R}_{\frac{\pi}{4}}) \geq  -\frac{1}{4} + \frac{ C_0}{R^2},
\]
which concludes the proof. 
\end{proof}

By noting that the change $\theta\mapsto \pi-\theta$ corresponds to a unitary transform, we summarize our observations on non-resonance angles as follows:
\begin{corollary}\label{corol-nonres}
All half-angles $\theta \in [\frac{\pi}{4},\frac{\pi}{2})\cup (\frac{\pi}{2},\frac{3\pi}{4}]$ 
are non-resonant with $\kappa(\theta)=1$.
\end{corollary}

For the purpose of certain estimates later on, we introduce the following operator.
\begin{definition}For $K^R_\theta$ and $\partial_* K^R_\theta$ as in Definition \ref{def: Karl}, we set 
    \[
    r^R_{\theta,\alpha}(u) = \int_{K^R_\theta} |u|^2 \mathrm{d}x - \alpha \int_{\Gamma_\theta^R} |u|^2 \dS - \alpha\int_{\partial_* K^R_\theta} |u|^2 \dS, \quad D(r^R_{\theta,\alpha}) = H^1(K^R_\theta).
    \]
\end{definition}
\begin{lemma}
\label{lem: Kite lower bound robin bc}
    There exists $c>0$ such that $ R^R_{\theta,\alpha} \geq - c\alpha^2$.
\end{lemma}
\begin{proof}
Set $ K_{in} := \frac{R}{2\sin \theta} + K^\frac{R}{2}_\theta $ and $K_{out} = K^R_\theta \setminus K_{in}$ as shown in Figure \ref{fig: Kin Kout}. 
\begin{figure}
    \centering
    \scalebox{0.75}{\begin{tikzpicture}[x=0.7pt,y=0.7pt,yscale=-1,xscale=1]
%uncomment if require: \path (0,211); %set diagram left start at 0, and has height of 211

%Straight Lines [id:da5770766757853512] 
\draw    (140,210) -- (140,13) ;
\draw [shift={(140,10)}, rotate = 90] [fill={rgb, 255:red, 0; green, 0; blue, 0 }  ][line width=0.08]  [draw opacity=0] (3.57,-1.72) -- (0,0) -- (3.57,1.72) -- cycle    ;
%Shape: Polygon [id:ds8246372521341071] 
\draw  [fill={rgb, 255:red, 0; green, 0; blue, 0 }  ,fill opacity=0.1 ] (200,30) -- (240,110) -- (200,190) -- (40,110) -- cycle ;
%Straight Lines [id:da2837099881900341] 
\draw    (20,110) -- (257,110) ;
\draw [shift={(260,110)}, rotate = 180] [fill={rgb, 255:red, 0; green, 0; blue, 0 }  ][line width=0.08]  [draw opacity=0] (3.57,-1.72) -- (0,0) -- (3.57,1.72) -- cycle    ;
%Straight Lines [id:da862097471520508] 
\draw [line width=1.25]   (140,110) -- (220,70) ;
%Straight Lines [id:da9472147717856606] 
\draw [line width=1.25]   (140,110) -- (220,150) ;
%Shape: Polygon [id:ds6301003922761397] 
\draw   (200,30) -- (220,70) -- (140,110) -- (220,150) -- (200,190) -- (40,110) -- cycle ;

% Text Node
\draw (251,111) node [anchor=north west][inner sep=0.75pt]  [font=\tiny] [align=left] {$\displaystyle x_{1}$};
% Text Node
\draw (128,12) node [anchor=north west][inner sep=0.75pt]  [font=\tiny] [align=left] {$\displaystyle x_{2}$};
% Text Node
\draw (211,92) node [anchor=north west][inner sep=0.75pt]  [font=\scriptsize] [align=left] {$\displaystyle K_{in}$};
% Text Node
\draw (167,56) node [anchor=north west][inner sep=0.75pt]  [font=\scriptsize] [align=left] {$\displaystyle K_{out}$};
% Text Node
\draw (78,44.33) node [anchor=north west][inner sep=0.75pt]  [font=\scriptsize] [align=left] {$\displaystyle K_{\theta }^{R}$};

\end{tikzpicture}}
    \caption{$K^R_\theta$ divided into $K_{out}$ and $K_{in}$.}
    \label{fig: Kin Kout}
\end{figure}
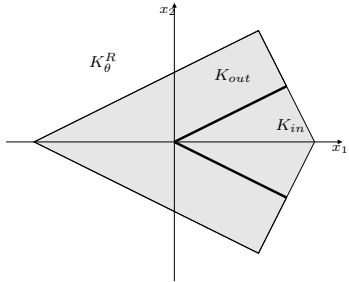

 Consider the map
    \[
    J: H^1(K^R_\theta) \rightarrow H^1(K_{in}) \times H^1(K_{out}), \quad u \mapsto (u\upharpoonright
_{K_{in}}, u\upharpoonright
_{K_{out}}) = (u_{in}, u_{out}).
    \]
    Since $\partial_* K^R_\theta \subset \partial K^R_\theta$, for any $u \in H^1(K^R_\theta)$ one has 
    \begin{align*}
         r^R_{\theta,\alpha} &\geq \int_{K^R_\theta} |u|^2 \mathrm{d}x - 2\alpha \int_{\Gamma_\theta^R} |u|^2 \dS - \alpha \int_{\partial K^R_\theta} |u|^2 \dS \\
         &= \int_{K_{in}} |u|^2 \mathrm{d}x - \alpha \int_{\partial K_{in}} |u|^2 \dS +\int_{K_{out}} |u|^2 \mathrm{d}x - \alpha \int_{\partial K_{out}} |u|^2 \dS \\
         &=: r^R_{in,\alpha}(u_{in}) + r^R_{out,\alpha}(u_{out}) = (r^R_{in,\alpha} \oplus r^R_{out,\alpha})(Ju),
    \end{align*}
    where $R_{in/out}$ is the Laplacian on $K_{in/out}$ with $\alpha$-Robin boundary condition. Consequently, it follows that
  $R^R_{\theta,\alpha} \geq R^R_{in,\alpha} \oplus R^R_{out,\alpha}$.
    Using the same arguments as in Lemma \ref{lem: scaling anything}, one can show that
    $$ R^R_{in,\alpha} \cong \frac{1}{R^2} R^1_{in, \alpha R}\quad R^R_{out,\alpha} \cong \frac{1}{R^2} R^1_{out, \alpha R}. $$
    Thus,  \cite[Lemma 2.7]{KOBP20} ensures that there are constants $c_{in} > 0 $ and $c_{out} > 0 $ such that, as $\alpha R \to \infty$, there holds
    \[
    R^1_{in,\alpha R} \geq -c_{in}(\alpha R)^2 \quad\text{and}\quad R^1_{out,\alpha R} \geq -c_{in}(\alpha R)^2,
    \]
    which entails that $R^{R}_{in,\alpha} \geq  - c_{in}\alpha^2$ and $ R^{R}_{out,\alpha} \geq  - c_{in}\alpha^2$. Therefore, there is $c>0$ such that $R^R_{\theta,\alpha} \geq -c\alpha^2$, and the lemma is proved. 
\end{proof}

\section{Neighborhoods of curved corners}
\label{sec: V delta}
To analyze the eigenvalue asymptotics near the corner, we first construct a neighborhood around the corner by splitting the curve into two segments meeting there, each smoothly continued beyond the intersection. This continuation allows us to build a neighborhood whose spectral properties remain asymptotically independent of the continuation choice; see Figure \ref{fig: V delta 2}(a). We then find a bi-Lipschitz map that straightens this neighborhood into a kite, whose spectral behavior was analyzed previously. The construction adapts techniques from \cite{KOBP20} to this setting.
\begin{figure}
    \centering
    \begin{tabular}{ccc}
        \begin{tikzpicture}[x=0.75pt,y=0.75pt,yscale=-1,xscale=1]
%uncomment if require: \path (0,300); %set diagram left start at 0, and has height of 300

%Straight Lines [id:da08908705509727288] 
\draw [fill={rgb, 255:red, 155; green, 155; blue, 155 }  ,fill opacity=0.3 ]   (215.89,55.89) -- (240.33,123.5) -- (198.67,187.17) ;
%Shape: Polygon [id:ds3921036897665544] 
\draw  [draw opacity=0][fill={rgb, 255:red, 155; green, 155; blue, 155 }  ,fill opacity=0.3 ] (215.89,55.89) -- (198.67,187.17) -- (82.67,114.5) -- (121.06,97.22) -- cycle ;
%Curve Lines [id:da5563651791627514] 
\draw [fill={rgb, 255:red, 255; green, 255; blue, 255 }  ,fill opacity=1 ]   (215.89,55.89) .. controls (181.22,67.22) and (154.14,82.76) .. (138.48,90.76) .. controls (122.81,98.76) and (105.67,99.83) .. (82.67,114.5) ;
%Curve Lines [id:da3885012346964293] 
\draw [fill={rgb, 255:red, 255; green, 255; blue, 255 }  ,fill opacity=1 ]   (82.67,114.5) .. controls (102.33,118.5) and (126,132.5) .. (136.67,142.83) .. controls (147.33,153.17) and (164.67,167.5) .. (198.67,187.17) ;
%Curve Lines [id:da4006091342474386] 
\draw    (80,60) .. controls (117.67,48.17) and (120.33,90.5) .. (160,120) .. controls (199.67,149.5) and (226,144.17) .. (250,190) ;
%Curve Lines [id:da6037176650926172] 
\draw    (70,170) .. controls (98.67,165.5) and (129,144.17) .. (160,120) .. controls (191,95.83) and (242.33,92.83) .. (270,70) ;
%Straight Lines [id:da7699416200014766] 
\draw    (70,120) -- (277,120) ;
\draw [shift={(280,120)}, rotate = 180] [fill={rgb, 255:red, 0; green, 0; blue, 0 }  ][line width=0.08]  [draw opacity=0] (3.57,-1.72) -- (0,0) -- (3.57,1.72) -- cycle    ;
%Straight Lines [id:da6030891567529979] 
\draw    (160,210) -- (160,43) ;
\draw [shift={(160,40)}, rotate = 90] [fill={rgb, 255:red, 0; green, 0; blue, 0 }  ][line width=0.08]  [draw opacity=0] (3.57,-1.72) -- (0,0) -- (3.57,1.72) -- cycle    ;
%Curve Lines [id:da5857231488711578] 
\draw [line width=1.5]    (160,120) .. controls (194.33,97.17) and (216.1,97.1) .. (228.43,90.43) ;
%Curve Lines [id:da5831584343637567] 
\draw [line width=1.5]    (160,120) .. controls (193,142.9) and (201.2,140.7) .. (220.2,154.5) ;
%Curve Lines [id:da7865908868406338] 
\draw [fill={rgb, 255:red, 155; green, 155; blue, 155 }  ,fill opacity=0.3 ][line width=0.75]    (82.67,114.5) .. controls (92.56,107.67) and (111.22,100.33) .. (122.56,97) ;
%Curve Lines [id:da09159016884128734] 
\draw [fill={rgb, 255:red, 155; green, 155; blue, 155 }  ,fill opacity=0.3 ]   (144.11,87.67) .. controls (175,71.67) and (199,61.22) .. (215.89,55.89) ;

% Text Node
\draw (269.25,122.75) node [anchor=north west][inner sep=0.75pt]  [font=\tiny] [align=left] {$\displaystyle x_{1}$};
% Text Node
\draw (146.5,40) node [anchor=north west][inner sep=0.75pt]  [font=\tiny] [align=left] {$\displaystyle x_{2}$};
% Text Node
\draw (190,76.25) node [anchor=north west][inner sep=0.75pt]  [font=\scriptsize] [align=left] {$\displaystyle V_{t}$};
% Text Node
\draw (170.67,135) node [anchor=north west][inner sep=0.75pt]  [font=\scriptsize,color={black}  ,opacity=1 ] [align=left] {$\displaystyle \Gamma_{t}$};
\end{tikzpicture}
&\qquad
&
 \begin{tikzpicture}[x=0.75pt,y=0.75pt,yscale=-1,xscale=1]
%uncomment if require: \path (0,300); %set diagram left start at 0, and has height of 300

%Straight Lines [id:da4781470369372741] 
\draw    (170,220) -- (170,53) ;
\draw [shift={(170,50)}, rotate = 90] [fill={rgb, 255:red, 0; green, 0; blue, 0 }  ][line width=0.08]  [draw opacity=0] (3.57,-1.72) -- (0,0) -- (3.57,1.72) -- cycle    ;
%Straight Lines [id:da4417382327215835] 
\draw    (80,140) -- (287,140) ;
\draw [shift={(290,140)}, rotate = 180] [fill={rgb, 255:red, 0; green, 0; blue, 0 }  ][line width=0.08]  [draw opacity=0] (3.57,-1.72) -- (0,0) -- (3.57,1.72) -- cycle    ;
%Curve Lines [id:da02859796050012009] 
\draw [line width=0.08]   (80,190) .. controls (108.67,185.5) and (139,164.17) .. (170,140) .. controls (201,115.83) and (252.33,112.83) .. (280,90) ;
%Curve Lines [id:da7475642854254311] 
\draw [line width=0.08]   (90,80) .. controls (127.67,68.17) and (130.33,110.5) .. (170,140) .. controls (209.67,169.5) and (236,164.17) .. (260,210) ;
%Straight Lines [id:da9800463930531634] 
\draw    (207.93,120.75) -- (200.9,101.74) ;
\draw [shift={(199.86,98.93)}, rotate = 69.7] [fill={rgb, 255:red, 0; green, 0; blue, 0 }  ][line width=0.08]  [draw opacity=0] (3.57,-1.72) -- (0,0) -- (3.57,1.72) -- cycle    ;
%Straight Lines [id:da09954102953100685] 
\draw    (207.93,120.75) -- (187.5,129.02) ;
\draw [shift={(184.71,130.14)}, rotate = 337.97] [fill={rgb, 255:red, 0; green, 0; blue, 0 }  ][line width=0.08]  [draw opacity=0] (3.57,-1.72) -- (0,0) -- (3.57,1.72) -- cycle    ;
%Straight Lines [id:da09148548428432524] 
\draw    (207.55,161.86) -- (197.55,180.64) ;
\draw [shift={(196.14,183.29)}, rotate = 298.03] [fill={rgb, 255:red, 0; green, 0; blue, 0 }  ][line width=0.08]  [draw opacity=0] (3.57,-1.72) -- (0,0) -- (3.57,1.72) -- cycle    ;
%Straight Lines [id:da809600189117105] 
\draw    (207.55,161.86) -- (187.97,152.06) ;
\draw [shift={(185.29,150.71)}, rotate = 26.59] [fill={rgb, 255:red, 0; green, 0; blue, 0 }  ][line width=0.08]  [draw opacity=0] (3.57,-1.72) -- (0,0) -- (3.57,1.72) -- cycle    ;
%Curve Lines [id:da8437240348529169] 
\draw [densely dotted] [line width=0.85]   (84.14,128.14) .. controls (120.43,112.43) and (123.86,146.71) .. (170,140) .. controls (216.14,133.29) and (259.86,146.14) .. (272.14,151) ;
%Straight Lines [id:da12110035982231904] 
\draw    (258.14,146.43) -- (233.29,176.71) ;
%Straight Lines [id:da12499904722438593] 
\draw    (243.29,109) -- (258.14,146.43) ;
%Straight Lines [id:da7228687254019524] 
\draw    (259.67,151.03) -- (238.62,176.68) ;
\draw [shift={(236.71,179)}, rotate = 309.38] [fill={rgb, 255:red, 0; green, 0; blue, 0 }  ][line width=0.08]  [draw opacity=0] (3.57,-1.72) -- (0,0) -- (3.57,1.72) -- cycle    ;
\draw [shift={(261.57,148.71)}, rotate = 129.38] [fill={rgb, 255:red, 0; green, 0; blue, 0 }  ][line width=0.08]  [draw opacity=0] (3.57,-1.72) -- (0,0) -- (3.57,1.72) -- cycle    ;
%Straight Lines [id:da8494897130958592] 
\draw    (247.54,111.5) -- (260.18,143.35) ;
\draw [shift={(261.29,146.14)}, rotate = 248.35] [fill={rgb, 255:red, 0; green, 0; blue, 0 }  ][line width=0.08]  [draw opacity=0] (3.57,-1.72) -- (0,0) -- (3.57,1.72) -- cycle    ;
\draw [shift={(246.43,108.71)}, rotate = 68.35] [fill={rgb, 255:red, 0; green, 0; blue, 0 }  ][line width=0.08]  [draw opacity=0] (3.57,-1.72) -- (0,0) -- (3.57,1.72) -- cycle    ;

% Text Node
\draw (278,141) node [anchor=north west][inner sep=0.75pt]  [font=\tiny] [align=left] {$\displaystyle x_{1}$};
% Text Node
\draw (158,51) node [anchor=north west][inner sep=0.75pt]  [font=\tiny] [align=left] {$\displaystyle x_{2}$};
% Text Node
\draw (261,82) node [anchor=north west][inner sep=0.75pt]  [font=\scriptsize,color={black}  ,opacity=1 ] [align=left] {$\displaystyle \Gamma_{+}$};
% Text Node
\draw (261,192) node [anchor=north west][inner sep=0.75pt]  [font=\scriptsize,color={black}  ,opacity=1 ] [align=left] {$\displaystyle \Gamma_{-}$};
% Text Node
\draw (179.33,117.38) node [anchor=north west][inner sep=0.75pt]  [font=\tiny] [align=left] {$\displaystyle T_{+}$};
% Text Node
\draw (178.81,153.19) node [anchor=north west][inner sep=0.75pt]  [font=\tiny] [align=left] {$\displaystyle T_{-}$};
% Text Node
\draw (187.48,97.38) node [anchor=north west][inner sep=0.75pt]  [font=\tiny] [align=left] {$\displaystyle n_{+}$};
% Text Node
\draw (185.57,171.24) node [anchor=north west][inner sep=0.75pt]  [font=\tiny] [align=left] {$\displaystyle n_{-}$};
% Text Node
\draw (251.14,164) node [anchor=north west][inner sep=0.75pt]  [font=\tiny] [align=left] {$\displaystyle t$};
% Text Node
\draw (255.71,122.86) node [anchor=north west][inner sep=0.75pt]  [font=\tiny] [align=left] {$\displaystyle t$};
% Text Node
\draw (223.14,91.71) node [anchor=north west][inner sep=0.75pt]  [font=\scriptsize] [align=left] {$\displaystyle A_{+}( t)$};
% Text Node
\draw (211.43,178) node [anchor=north west][inner sep=0.75pt]  [font=\scriptsize] [align=left] {$\displaystyle A_{-}( t)$};
% Text Node
\draw (230.86,142.57) node [anchor=north west][inner sep=0.75pt]  [font=\scriptsize,color={black}  ,opacity=1 ] [align=left] {$\displaystyle Y( t)$};
\end{tikzpicture}\\
(a) && (b)
\end{tabular}
        \caption{(a) A neighborhood of a corner with a smooth continuation of the $\delta$-interaction support. (b) $\Gamma_+$ and $\Gamma_-$ intersecting at angle $2\theta$ and the graph of $Y(t)$ in dashed dots.}
        \label{fig: V delta 2}
\end{figure}
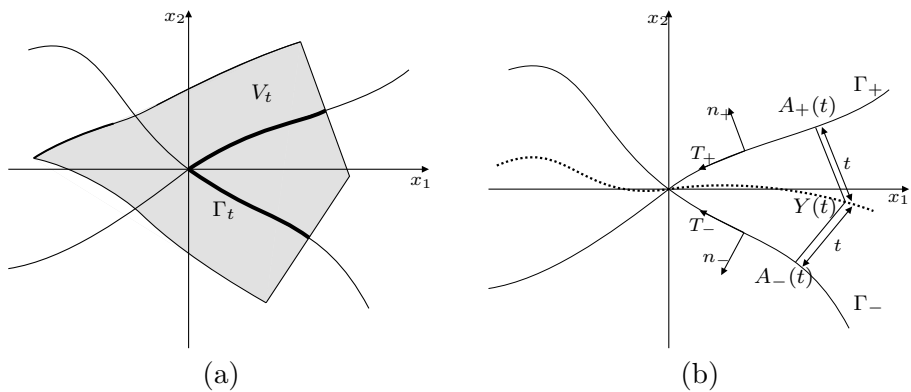
\subsection{Geometric setting and change of variables}
As discussed before, we consider the setting of two curves intersecting at a given angle. More precisely, we introduce the following notations for the remainder of this section.
\begin{notation}\label{notation param}
Let $\Gamma_+$ and $\Gamma_-$ be two injective $C^3$-smooth curves intersecting exactly at the origin with an angle $2\theta \in (0,\pi)$. For parameters $\pm s_\pm >0 $, consider the arc-length parametrizations $\gamma_\pm:[s_-,s_+] \rightarrow \rr^2$ of $\Gamma_\pm$. Without loss of generality, we may assume that 
\[
\Gamma_+ \cap \Gamma_- = (0,0) = \gamma_{\pm}(0), \quad \gamma'_\pm(0) = (\cos \theta, \pm \sin\theta),
\]
Furthermore, we define the tangent vectors, normal vectors, and curvatures of $\Gamma_\pm$ by
\begin{align*}T_\pm(s) = \begin{pmatrix}
    T_1^\pm \\ T_2^\pm
\end{pmatrix} = \gamma'_\pm(s), \quad n_\pm(s) = \begin{pmatrix}
    n_1^\pm \\ n_2^\pm
\end{pmatrix} =\begin{pmatrix}
    0 & 1\\ \mp 1 & 0
\end{pmatrix} \cdot T_\pm(s), \quad n'_\pm(s)=k_\pm(s) \gamma'_\pm(s). 
\end{align*}
\end{notation}
For the construction of the neighborhood, the same arguments in the proof of \cite[Lemma A.1]{KOBP20} imply the following:  There exists $t_1 > 0$ and a $C^2$-smooth function $Y:(-t_1,t_1) \rightarrow \rr^2$ such that, for any $t \in (-t_1,t_1)$, the point $Y(t)$ lies at distance exactly $t$ from both curves $\Gamma_+$ and $\Gamma_-$. More precisely, there exist functions $A_\pm: (-t_1,t_1) \to \Gamma_\pm$ such that $ A_\pm(t)- Y(t) = t $ holds for all $t$, and each $A_\pm $ can be written as $A_\pm(t) = \gamma_\pm(\lambda_\pm(t))$, where
\begin{align}\label{def lambda func}
\lambda_\pm \in C^2, \quad \lambda_\pm(0) = 0, \quad \lambda'_\pm(0) = \cotan \theta, \quad Y(0) = \begin{pmatrix}
    0\\0
\end{pmatrix}, \quad Y'(0) = \frac{1}{\sin \theta} \begin{pmatrix}
    1\\0
\end{pmatrix}.
\end{align}
We refer to Figure \ref{fig: V delta 2}(b) for a visualization of this construction.

\begin{figure}
    \centering
    \begin{tikzpicture}[x=0.75pt,y=0.75pt,yscale=-1,xscale=1]
%uncomment if require: \path (0,300); %set diagram left start at 0, and has height of 300

%Straight Lines [id:da08908705509727288] 
\draw [densely dotted][line width=1.20][color={black}  ,draw opacity=1 ][fill={rgb, 255:red, 155; green, 155; blue, 155 }  ,fill opacity=0.3 ]   (215.89,55.89) -- (240.33,123.5) -- (198.67,187.17) ;
%Shape: Polygon [id:ds3921036897665544] 
\draw  [draw opacity=0][fill={rgb, 255:red, 155; green, 155; blue, 155 }  ,fill opacity=0.3 ] (215.89,55.89) -- (198.67,187.17) -- (82.67,114.5) -- (121.06,97.22) -- cycle ;
%Curve Lines [id:da5563651791627514] 
\draw [fill={rgb, 255:red, 255; green, 255; blue, 255 }  ,fill opacity=5 ]   (215.89,55.89) .. controls (181.22,67.22) and (154.14,82.76) .. (138.48,90.76) .. controls (122.81,98.76) and (105.67,99.83) .. (82.67,114.5) ;
%Curve Lines [id:da3885012346964293] 
\draw [fill={rgb, 255:red, 255; green, 255; blue, 255 }  ,fill opacity=1 ]   (82.67,114.5) .. controls (102.33,118.5) and (126,132.5) .. (136.67,142.83) .. controls (147.33,153.17) and (164.67,167.5) .. (198.67,187.17) ;
%Curve Lines [id:da4006091342474386] 
\draw    (80,60) .. controls (117.67,48.17) and (120.33,90.5) .. (160,120) .. controls (199.67,149.5) and (226,144.17) .. (250,190) ;
%Curve Lines [id:da6037176650926172] 
\draw    (70,170) .. controls (98.67,165.5) and (129,144.17) .. (160,120) .. controls (191,95.83) and (242.33,92.83) .. (270,70) ;
%Straight Lines [id:da7699416200014766] 
\draw    (70,120) -- (277,120) ;
\draw [shift={(280,120)}, rotate = 180] [fill={rgb, 255:red, 0; green, 0; blue, 0 }  ][line width=0.08]  [draw opacity=0] (3.57,-1.72) -- (0,0) -- (3.57,1.72) -- cycle    ;
%Straight Lines [id:da6030891567529979] 
\draw    (160,210) -- (160,43) ;
\draw [shift={(160,40)}, rotate = 90] [fill={rgb, 255:red, 0; green, 0; blue, 0 }  ][line width=0.08]  [draw opacity=0] (3.57,-1.72) -- (0,0) -- (3.57,1.72) -- cycle    ;
%Curve Lines [id:da7865908868406338] 
\draw [fill={rgb, 255:red, 155; green, 155; blue, 155 }  ,fill opacity=0.3 ][line width=0.75]    (82.67,114.5) .. controls (92.56,107.67) and (111.22,100.33) .. (122.56,97) ;
%Curve Lines [id:da09159016884128734] 
\draw [fill={rgb, 255:red, 155; green, 155; blue, 155 }  ,fill opacity=0.3 ]   (144.11,87.67) .. controls (175,71.67) and (199,61.22) .. (215.89,55.89) ;
%Straight Lines [id:da40515208843090555] 
\draw    (241.7,128.02) -- (224.83,153.98) ;
\draw [shift={(223.2,156.5)}, rotate = 303] [fill={rgb, 255:red, 0; green, 0; blue, 0 }  ][line width=0.08]  [draw opacity=0] (3.57,-1.72) -- (0,0) -- (3.57,1.72) -- cycle    ;
\draw [shift={(243.33,125.5)}, rotate = 123] [fill={rgb, 255:red, 0; green, 0; blue, 0 }  ][line width=0.08]  [draw opacity=0] (3.57,-1.72) -- (0,0) -- (3.57,1.72) -- cycle    ;
%Straight Lines [id:da43610555903561266] 
\draw    (219.9,160.35) -- (203.03,186.32) ;
\draw [shift={(201.4,188.83)}, rotate = 303] [fill={rgb, 255:red, 0; green, 0; blue, 0 }  ][line width=0.08]  [draw opacity=0] (3.57,-1.72) -- (0,0) -- (3.57,1.72) -- cycle    ;
\draw [shift={(221.53,157.83)}, rotate = 123] [fill={rgb, 255:red, 0; green, 0; blue, 0 }  ][line width=0.08]  [draw opacity=0] (3.57,-1.72) -- (0,0) -- (3.57,1.72) -- cycle    ;
%Shape: Right Angle [id:dp8411608255255555] 
\draw   (222.83,92.94) -- (220.29,87.28) -- (225.89,84.77) ;
%Curve Lines [id:da6538067937846276] 
\draw [line width=1.5]    (160,120) .. controls (194.33,97.17) and (215.33,97.17) .. (228.43,90.43) ;
%Shape: Right Angle [id:dp06828589649977701] 
\draw   (214.98,151.28) -- (218.23,146) -- (223.46,149.22) ;
%Curve Lines [id:da7553649139641571] 
\draw [line width=1.5]    (160,120) .. controls (193,142.9) and (201.2,140.7) .. (220.2,154.5) ;

% Text Node
\draw (269.25,122.75) node [anchor=north west][inner sep=0.75pt]  [font=\tiny] [align=left] {$\displaystyle x_{1}$};
% Text Node
\draw (146.5,40) node [anchor=north west][inner sep=0.75pt]  [font=\tiny] [align=left] {$\displaystyle x_{2}$};
% Text Node
\draw (190,76.25) node [anchor=north west][inner sep=0.75pt]  [font=\scriptsize] [align=left] {$\displaystyle V_{t}$};
% Text Node
\draw (170.67,135) node [anchor=north west][inner sep=0.75pt]  [font=\scriptsize,color={black}  ,opacity=1 ] [align=left] {$\displaystyle \Gamma_{t}$};
% Text Node
\draw (226.67,59) node [anchor=north west][inner sep=0.75pt]  [font=\scriptsize,color={black}  ,opacity=1 ] [align=left] {$\displaystyle \partial_{*} V_{t}$};
% Text Node
\draw (236.67,141) node [anchor=north west][inner sep=0.75pt]  [font=\tiny] [align=left] {$\displaystyle t$};
% Text Node
\draw (214.67,171.67) node [anchor=north west][inner sep=0.75pt]  [font=\tiny] [align=left] {$\displaystyle t$};
\end{tikzpicture}
    \caption{The neighborhood $V_t$ constructed around  $\Gamma_t$.}
    \label{fig: V delta all properties}
\end{figure}
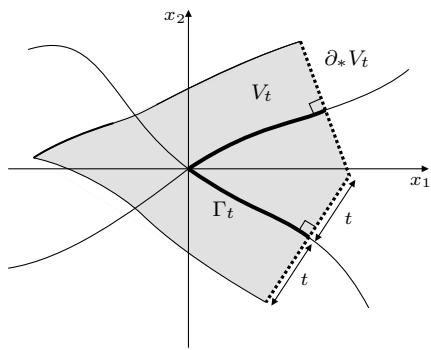

The construction of an appropriate neighborhood, which is illustrated in Fig.~\ref{fig: V delta all properties}, can be summarized by the following lemma. 
\begin{lemma}\label{lem: straigthen curved kite}There exist
\begin{itemize}
    \item[(i)] a $C^2$-smooth function $r$ defined in a neighborhood of $0$ with $r(0) = 0$ and $r'(0) = 1$,
    \item[(ii)] $C^2$-smooth functions $\lambda_\pm $ satisfying $\lambda_\pm(0) =0$ and $\lambda_\pm'(0) =\cotan\theta$,
    \item[(iii)] a bi-Lipschitz mapping $\phi_V: K^{r(t)}_\theta \rightarrow \phi_V(K^{r(t)}_\theta) =: V_t $,  with $ \phi'_V(x) = I_2 + \mathcal{O}(| x |)$ for $|x| \rightarrow 0$,
\end{itemize}
such that, defining 
    \begin{align*}
    A_{\pm}:= \gamma_\pm \circ \lambda_\pm,\quad  \Gamma_t := A_+([0, t ))\cup A_-([0, t )), \\
    \partial_* V_t := \{  A_\pm(t) + N_\pm(A_\pm(t)) \cdot \tau :  \, \tau \in (-t,t) \} ,
    \end{align*}
the following identities hold for all sufficiently small $t >0$
\[
\Gamma_t=  \phi_V(\Gamma_\theta^{r(t)}), \quad \partial_*V_t = \phi_V(\partial_* K^{r(t)}_\theta). 
\]
\end{lemma}
\begin{proof}

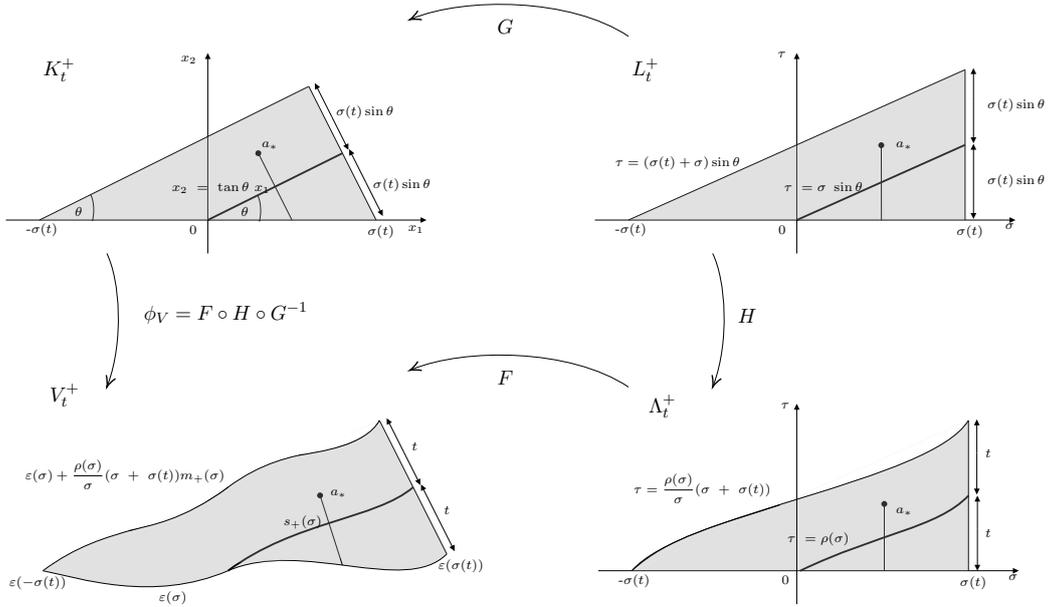
\begin{figure}
    \centering
    \scalebox{0.7}{\begin{tikzpicture}[x=0.9pt,y=0.9pt,yscale=-1,xscale=1]
%uncomment if require: \path (0,406); %set diagram left start at 0, and has height of 406

%Straight Lines [id:da32763084677628307] 
\draw    (30,140) -- (277,140) ;
\draw [shift={(280,140)}, rotate = 180] [fill={rgb, 255:red, 0; green, 0; blue, 0 }  ][line width=0.08]  [draw opacity=0] (3.57,-1.72) -- (0,0) -- (3.57,1.72) -- cycle    ;
%Straight Lines [id:da6891281773140133] 
\draw    (150,160) -- (150,43) ;
\draw [shift={(150,40)}, rotate = 90] [fill={rgb, 255:red, 0; green, 0; blue, 0 }  ][line width=0.08]  [draw opacity=0] (3.57,-1.72) -- (0,0) -- (3.57,1.72) -- cycle    ;
%Straight Lines [id:da8979191779971859] 
\draw [line width=1.0]   (150,140) -- (230,100) ;
%Straight Lines [id:da6617522524171178] 
\draw    (50,140) -- (210,60) ;
%Straight Lines [id:da6838500631446954] 
\draw    (210,60) -- (250,140) ;
%Straight Lines [id:da7105996286592238] 
\draw    (223.6,78) -- (232.26,95.32) ;
\draw [shift={(233.6,98)}, rotate = 243.43] [fill={rgb, 255:red, 0; green, 0; blue, 0 }  ][line width=0.08]  [draw opacity=0] (3.57,-1.72) -- (0,0) -- (3.57,1.72) -- cycle    ;
%Straight Lines [id:da032204291305640065] 
\draw    (223.6,78) -- (214.94,60.68) ;
\draw [shift={(213.6,58)}, rotate = 63.43] [fill={rgb, 255:red, 0; green, 0; blue, 0 }  ][line width=0.08]  [draw opacity=0] (3.57,-1.72) -- (0,0) -- (3.57,1.72) -- cycle    ;
%Straight Lines [id:da3819329151613692] polygon 1 
\draw    (244,117.6) -- (252.66,134.92) ;
\draw [shift={(254,137.6)}, rotate = 243.43] [fill={rgb, 255:red, 0; green, 0; blue, 0 }  ][line width=0.08]  [draw opacity=0] (3.57,-1.72) -- (0,0) -- (3.57,1.72) -- cycle    ;
%Straight Lines [id:da7072359353797377] 
\draw    (244,117.6) -- (235.34,100.28) ;
\draw [shift={(234,97.6)}, rotate = 63.43] [fill={rgb, 255:red, 0; green, 0; blue, 0 }  ][line width=0.08]  [draw opacity=0] (3.57,-1.72) -- (0,0) -- (3.57,1.72) -- cycle    ;
%Shape: Arc [id:dp7744665311977997] 
\draw  [draw opacity=0] (81.43,140.43) .. controls (81.91,138.64) and (82.2,136.41) .. (82.2,134) .. controls (82.2,130.13) and (81.46,126.73) .. (80.36,124.82) -- (78.3,134) -- cycle ; \draw   (81.43,140.43) .. controls (81.91,138.64) and (82.2,136.41) .. (82.2,134) .. controls (82.2,130.13) and (81.46,126.73) .. (80.36,124.82) ;  
%Shape: Arc [id:dp18122403017739797] 
\draw  [draw opacity=0] (180.63,140.83) .. controls (181.11,139.04) and (181.4,136.81) .. (181.4,134.4) .. controls (181.4,130.53) and (180.66,127.13) .. (179.56,125.22) -- (177.5,134.4) -- cycle ; \draw   (180.63,140.83) .. controls (181.11,139.04) and (181.4,136.81) .. (181.4,134.4) .. controls (181.4,130.53) and (180.66,127.13) .. (179.56,125.22) ;  
%Straight Lines [id:da3449022705852727] 
\draw    (200,140) -- (180,100) ;
\draw [shift={(180,100)}, rotate = 243.43] [color={rgb, 255:red, 0; green, 0; blue, 0 }  ][fill={rgb, 255:red, 0; green, 0; blue, 0 }  ][line width=0.75]      (0, 0) circle [x radius= 1.34, y radius= 1.34]   ;
%Straight Lines [id:da7198680878109306] 
\draw    (380,140) -- (627,140) ;
\draw [shift={(630,140)}, rotate = 180] [fill={rgb, 255:red, 0; green, 0; blue, 0 }  ][line width=0.08]  [draw opacity=0] (3.57,-1.72) -- (0,0) -- (3.57,1.72) -- cycle    ;
%Straight Lines [id:da3842328093181052] 
\draw    (500,160) -- (500,43) ;
\draw [shift={(500,40)}, rotate = 90] [fill={rgb, 255:red, 0; green, 0; blue, 0 }  ][line width=0.08]  [draw opacity=0] (3.57,-1.72) -- (0,0) -- (3.57,1.72) -- cycle    ;
%Straight Lines [id:da965161455236999] polygon 2 delta
\draw [line width=1.0]   (500,140) -- (600,95) ;
%Straight Lines [id:da9965796454963454] 
\draw    (400,140) -- (600,50) ;
%Straight Lines [id:da9365388559267768] 
\draw    (600,50) -- (600,140) ;
%Straight Lines [id:da6883545506842347] 
\draw    (550.3,140.1) -- (550.2,95.2) ;
\draw [shift={(550.2,95.2)}, rotate = 269.87] [color={rgb, 255:red, 0; green, 0; blue, 0 }  ][fill={rgb, 255:red, 0; green, 0; blue, 0 }  ][line width=0.75]      (0, 0) circle [x radius= 1.34, y radius= 1.34]   ;
%Straight Lines [id:da24918223085423652] 
\draw    (605,69) -- (605,52) ;
\draw [shift={(605,49)}, rotate = 90] [fill={rgb, 255:red, 0; green, 0; blue, 0 }  ][line width=0.08]  [draw opacity=0] (3.57,-1.72) -- (0,0) -- (3.57,1.72) -- cycle    ;
%Straight Lines [id:da9262034215702473] 
\draw    (605,69) -- (605,91.03) ;
\draw [shift={(605,94.03)}, rotate = 270] [fill={rgb, 255:red, 0; green, 0; blue, 0 }  ][line width=0.08]  [draw opacity=0] (3.57,-1.72) -- (0,0) -- (3.57,1.72) -- cycle    ;
%Straight Lines [id:da045753983410706356] 
\draw    (382,350) -- (629,350) ;
\draw [shift={(632,350)}, rotate = 180] [fill={rgb, 255:red, 0; green, 0; blue, 0 }  ][line width=0.08]  [draw opacity=0] (3.57,-1.72) -- (0,0) -- (3.57,1.72) -- cycle    ;
%Straight Lines [id:da588973482389302] 
\draw    (602,260) -- (602,350) ;
%Straight Lines [id:da8178916370195153] 
\draw    (552,350) -- (552,310) ;
\draw [shift={(552,310)}, rotate = 270] [color={rgb, 255:red, 0; green, 0; blue, 0 }  ][fill={rgb, 255:red, 0; green, 0; blue, 0 }  ][line width=0.75]      (0, 0) circle [x radius= 1.34, y radius= 1.34]   ;
%Curve Lines [id:da07351792943960567] polygon 4 delta
\draw [line width=1.0]   (502,350) .. controls (548.2,327.6) and (583.8,324.4) .. (602,305) ;
%Curve Lines [id:da20063891663048827] polygon 2
\draw [fill={rgb, 255:red, 155; green, 155; blue, 155 }  ,fill opacity=0.3 ]   (52,350) .. controls (73.8,354.4) and (117.4,369.2) .. (162,350) .. controls (206.6,330.8) and (267,365) .. (292,340) ;
%Straight Lines [id:da18960256703102352] 
\draw    (252,260) -- (292,340) ;
%Curve Lines [id:da7492566556625394] polygon 3 delta
\draw [line width=1.0]  (162,350) .. controls (202,320) and (249.8,319.2) .. (272,300) ;
%Curve Lines [id:da7774985610384032] polygon 2
\draw [fill={rgb, 255:red, 155; green, 155; blue, 155 }  ,fill opacity=0.3 ]   (52,350) .. controls (92,320) and (122,330) .. (162,300) .. controls (202,270) and (232.2,286.4) .. (252,260) ;
%Straight Lines [id:da9342304981783277] 
\draw    (605,114.33) -- (605,97.33) ;
\draw [shift={(605,94.33)}, rotate = 90] [fill={rgb, 255:red, 0; green, 0; blue, 0 }  ][line width=0.08]  [draw opacity=0] (3.57,-1.72) -- (0,0) -- (3.57,1.72) -- cycle    ;
%Straight Lines [id:da9570124295330317] 
\draw    (605,114.33) -- (605,136.37) ;
\draw [shift={(605,139.37)}, rotate = 270] [fill={rgb, 255:red, 0; green, 0; blue, 0 }  ][line width=0.08]  [draw opacity=0] (3.57,-1.72) -- (0,0) -- (3.57,1.72) -- cycle    ;
%Straight Lines [id:da22812348762474188] 
\draw    (607,279.33) -- (607,262.33) ;
\draw [shift={(607,259.33)}, rotate = 90] [fill={rgb, 255:red, 0; green, 0; blue, 0 }  ][line width=0.08]  [draw opacity=0] (3.57,-1.72) -- (0,0) -- (3.57,1.72) -- cycle    ;
%Straight Lines [id:da0724674508701707] 
\draw    (607,279.33) -- (607,301.37) ;
\draw [shift={(607,304.37)}, rotate = 270] [fill={rgb, 255:red, 0; green, 0; blue, 0 }  ][line width=0.08]  [draw opacity=0] (3.57,-1.72) -- (0,0) -- (3.57,1.72) -- cycle    ;
%Straight Lines [id:da8887754490188741] 
\draw    (607.33,325) -- (607.33,308) ;
\draw [shift={(607.33,305)}, rotate = 90] [fill={rgb, 255:red, 0; green, 0; blue, 0 }  ][line width=0.08]  [draw opacity=0] (3.57,-1.72) -- (0,0) -- (3.57,1.72) -- cycle    ;
%Straight Lines [id:da24177929155621547] 
\draw    (607.33,325) -- (607.33,347.03) ;
\draw [shift={(607.33,350.03)}, rotate = 270] [fill={rgb, 255:red, 0; green, 0; blue, 0 }  ][line width=0.08]  [draw opacity=0] (3.57,-1.72) -- (0,0) -- (3.57,1.72) -- cycle    ;
%Straight Lines [id:da0022802954322180824] 
\draw    (266.33,278.67) -- (274.99,295.98) ;
\draw [shift={(276.33,298.67)}, rotate = 243.43] [fill={rgb, 255:red, 0; green, 0; blue, 0 }  ][line width=0.08]  [draw opacity=0] (3.57,-1.72) -- (0,0) -- (3.57,1.72) -- cycle    ;
%Straight Lines [id:da4080831864261405] 
\draw    (266.33,278.67) -- (257.67,261.35) ;
\draw [shift={(256.33,258.67)}, rotate = 63.43] [fill={rgb, 255:red, 0; green, 0; blue, 0 }  ][line width=0.08]  [draw opacity=0] (3.57,-1.72) -- (0,0) -- (3.57,1.72) -- cycle    ;
%Straight Lines [id:da7864554807552789] 
\draw    (286.67,318.33) -- (295.33,335.65) ;
\draw [shift={(296.67,338.33)}, rotate = 243.43] [fill={rgb, 255:red, 0; green, 0; blue, 0 }  ][line width=0.08]  [draw opacity=0] (3.57,-1.72) -- (0,0) -- (3.57,1.72) -- cycle    ;
%Straight Lines [id:da8050444325441762] 
\draw    (286.67,318.33) -- (278.01,301.02) ;
\draw [shift={(276.67,298.33)}, rotate = 63.43] [fill={rgb, 255:red, 0; green, 0; blue, 0 }  ][line width=0.08]  [draw opacity=0] (3.57,-1.72) -- (0,0) -- (3.57,1.72) -- cycle    ;
%Curve Lines [id:da7349315647696791] G 3 to 1
\draw    (400,30) .. controls (365.53,6.8) and (302.91,6.46) .. (271.42,19.4) ;
\draw [shift={(270,20)}, rotate = 336.39] [color={rgb, 255:red, 0; green, 0; blue, 0 }  ][line width=0.75]    (6.56,-2.94) .. controls (4.17,-1.38) and (1.99,-0.4) .. (0,0) .. controls (1.99,0.4) and (4.17,1.38) .. (6.56,2.94)   ;
%Curve Lines [id:da1506174493085365] 1 to 2
\draw    (90,160) .. controls (97.84,176.12) and (100.11,212.51) .. (90.6,238.43) ;
\draw [shift={(90,240)}, rotate = 291.42] [color={rgb, 255:red, 0; green, 0; blue, 0 }  ][line width=0.75]    (6.56,-2.94) .. controls (4.17,-1.38) and (1.99,-0.4) .. (0,0) .. controls (1.99,0.4) and (4.17,1.38) .. (6.56,2.94)   ;
%Curve Lines [id:da022238749403166835] F 4 to 3
\draw    (400,240) .. controls (365.53,216.8) and (302.91,216.46) .. (271.42,229.4) ;
\draw [shift={(270,230)}, rotate = 336.39] [color={rgb, 255:red, 0; green, 0; blue, 0 }  ][line width=0.75]    (6.56,-2.94) .. controls (4.17,-1.38) and (1.99,-0.4) .. (0,0) .. controls (1.99,0.4) and (4.17,1.38) .. (6.56,2.94)   ;
%Curve Lines [id:da48913070093722344] H
\draw    (450,160) .. controls (457.84,176.12) and (460.11,212.51) .. (450.6,238.43) ;
\draw [shift={(450,240)}, rotate = 291.42] [color={rgb, 255:red, 0; green, 0; blue, 0 }  ][line width=0.75]    (6.56,-2.94) .. controls (4.17,-1.38) and (1.99,-0.4) .. (0,0) .. controls (1.99,0.4) and (4.17,1.38) .. (6.56,2.94)   ;
%Straight Lines [id:da03873728455159553] 
\draw    (230,346.7) -- (216.71,304.91) ;
\draw [shift={(216.71,304.91)}, rotate = 252.36] [color={rgb, 255:red, 0; green, 0; blue, 0 }  ][fill={rgb, 255:red, 0; green, 0; blue, 0 }  ][line width=0.75]      (0, 0) circle [x radius= 1.34, y radius= 1.34]   ;
%Shape: Polygon [id:ds2868494406513269] polygon 1
\draw  [draw opacity=0][fill={rgb, 255:red, 155; green, 155; blue, 155 }  ,fill opacity=0.3 ] (210,60) -- (250,140) -- (50,140) -- cycle ;
%Shape: Polygon [id:ds893638839956165] 
\draw  [draw opacity=0][fill={rgb, 255:red, 155; green, 155; blue, 155 }  ,fill opacity=0.3 ] (600,50) -- (600,140) -- (400,140) -- cycle ;
%Shape: Polygon [id:ds15572674032499256] 
\draw  [draw opacity=0][fill={rgb, 255:red, 155; green, 155; blue, 155 }  ,fill opacity=0.3 ] (602,260) -- (602,350) -- (402,350) -- cycle ;
%Curve Lines [id:da8325079329981528] 
\draw [fill={rgb, 255:red, 255; green, 255; blue, 255 }  ,fill opacity=1 ]   (402,350) .. controls (431,315.6) and (579.4,296.8) .. (602,260) ;
%Curve Lines [id:da3845234472114589] 
\draw [fill={rgb, 255:red, 155; green, 155; blue, 155 }  ,fill opacity=0.3 ]   (402,350) .. controls (423.57,329.77) and (453.57,322.63) .. (489.86,310.63) ;
%Straight Lines [id:da04372179285659905] 
\draw    (500,370) -- (500,253) ;
\draw [shift={(500,250)}, rotate = 90] [fill={rgb, 255:red, 0; green, 0; blue, 0 }  ][line width=0.08]  [draw opacity=0] (3.57,-1.72) -- (0,0) -- (3.57,1.72) -- cycle    ;
%Shape: Polygon [id:ds32081105959774103] polygon 2
\draw  [draw opacity=0][fill={rgb, 255:red, 155; green, 155; blue, 155 }  ,fill opacity=0.3 ] (252,260) -- (292,340) -- (52,350) -- cycle ;
%Curve Lines [id:da5508925766283312] polygon 2
\draw [fill={rgb, 255:red, 255; green, 255; blue, 255 }  ,fill opacity=1 ]   (103,326.07) .. controls (116.25,322.57) and (134.75,320.25) .. (162,300) ;
%Curve Lines [id:da25946573013558816] 
\draw [fill={rgb, 255:red, 255; green, 255; blue, 255 }  ,fill opacity=1 ]   (204.25,280.82) .. controls (216,278.82) and (243.5,275.32) .. (252,260) ;

% Text Node
\draw (268,142) node [anchor=north west][inner sep=0.75pt]  [font=\tiny] [align=left] {$\displaystyle x_{1}$};
% Text Node
\draw (133,40.6) node [anchor=north west][inner sep=0.75pt]  [font=\tiny] [align=left] {$\displaystyle x_{2}$};
% Text Node
\draw (51,42) node [anchor=north west][inner sep=0.75pt]  [font=\normalsize,color={rgb, 255:red, 0; green, 0; blue, 0 }  ,opacity=1 ] [align=left] {$\displaystyle K_{t}^{+}$};
% Text Node
\draw (225.2,70.6) node [anchor=north west][inner sep=0.75pt]  [font=\tiny] [align=left] {$\displaystyle \sigma ( t)\sin \theta $};
% Text Node
\draw (247.2,113.8) node [anchor=north west][inner sep=0.75pt]  [font=\tiny] [align=left] {$\displaystyle \sigma ( t)\sin \theta $};
% Text Node
\draw (70,131) node [anchor=north west][inner sep=0.75pt]  [font=\tiny] [align=left] {$\displaystyle \theta $};
% Text Node
\draw (168.8,131) node [anchor=north west][inner sep=0.75pt]  [font=\tiny] [align=left] {$\displaystyle \theta $};
% Text Node
\draw (244,142) node [anchor=north west][inner sep=0.75pt]  [font=\tiny] [align=left] {$\displaystyle \sigma ( t)$};
% Text Node
\draw (41,141) node [anchor=north west][inner sep=0.75pt]  [font=\tiny] [align=left] {\mbox{-}$\displaystyle \sigma ( t)$};
% Text Node
\draw (138,142) node [anchor=north west][inner sep=0.75pt]  [font=\tiny] [align=left] {$\displaystyle 0$};
% Text Node
\draw (594,142) node [anchor=north west][inner sep=0.75pt]  [font=\tiny] [align=left] {$\displaystyle \sigma ( t)$};
% Text Node
\draw (391,141) node [anchor=north west][inner sep=0.75pt]  [font=\tiny] [align=left] {\mbox{-}$\displaystyle \sigma ( t)$};
% Text Node
\draw (488,142) node [anchor=north west][inner sep=0.75pt]  [font=\tiny] [align=left] {$\displaystyle 0$};
% Text Node
\draw (127.6,116.2) node [anchor=north west][inner sep=0.75pt]  [font=\tiny,color={black}  ,opacity=1 ] [align=left] {$\displaystyle x_{2} \ =\ \tan \theta \ x_{1}$};
% Text Node
\draw (487.6,38.2) node [anchor=north west][inner sep=0.75pt]  [font=\tiny] [align=left] {$\displaystyle \tau $};
% Text Node
\draw (622.8,141) node [anchor=north west][inner sep=0.75pt]  [font=\tiny] [align=left] {$\displaystyle \sigma $};
% Text Node
\draw (401,42) node [anchor=north west][inner sep=0.75pt]  [font=\normalsize] [align=left] {$\displaystyle L_{t}^{+}$};
% Text Node
\draw (491.33,116) node [anchor=north west][inner sep=0.75pt]  [font=\tiny,color={black}  ,opacity=1 ] [align=left] {$\displaystyle \tau \ =\sigma \ \sin \theta \ $};
% Text Node
\draw (612.33,66) node [anchor=north west][inner sep=0.75pt]  [font=\tiny] [align=left] {$\displaystyle \sigma ( t)\sin \theta $};
% Text Node
\draw (596,352) node [anchor=north west][inner sep=0.75pt]  [font=\tiny] [align=left] {$\displaystyle \sigma ( t)$};
% Text Node
\draw (393,351) node [anchor=north west][inner sep=0.75pt]  [font=\tiny] [align=left] {\mbox{-}$\displaystyle \sigma ( t)$};
% Text Node
\draw (490,352) node [anchor=north west][inner sep=0.75pt]  [font=\tiny] [align=left] {$\displaystyle 0$};
% Text Node
\draw (489.6,248.2) node [anchor=north west][inner sep=0.75pt]  [font=\tiny] [align=left] {$\displaystyle \tau $};
% Text Node
\draw (624.8,351) node [anchor=north west][inner sep=0.75pt]  [font=\tiny] [align=left] {$\displaystyle \sigma $};
% Text Node
\draw (493.33,326) node [anchor=north west][inner sep=0.75pt]  [font=\tiny,color={black}  ,opacity=1 ] [align=left] {$\displaystyle \tau \ =\rho ( \sigma ) \ $};
% Text Node
\draw (402,291) node [anchor=north west][inner sep=0.75pt]  [font=\tiny] [align=left] {$\displaystyle \tau =\frac{\rho ( \sigma )}{\sigma }( \sigma \ +\ \sigma ( t))$};
% Text Node
\draw (411,243.5) node [anchor=north west][inner sep=0.75pt]   [align=left] {$\displaystyle \Lambda_{t}^{+}$};
% Text Node
\draw (120,361) node [anchor=north west][inner sep=0.75pt]  [font=\tiny] [align=left] {$\displaystyle \varepsilon ( \sigma )$};
% Text Node
\draw (55,237) node [anchor=north west][inner sep=0.75pt]   [align=left] {$\displaystyle V_{t}^{+}$};
% Text Node
\draw (31,352) node [anchor=north west][inner sep=0.75pt]  [font=\tiny] [align=left] {$\displaystyle \varepsilon ( -\sigma ( t))$};
% Text Node
\draw (286,342) node [anchor=north west][inner sep=0.75pt]  [font=\tiny] [align=left] {$\displaystyle \varepsilon ( \sigma ( t))$};
% Text Node
\draw (612.33,111.33) node [anchor=north west][inner sep=0.75pt]  [font=\tiny] [align=left] {$\displaystyle \sigma ( t)\sin \theta $};
% Text Node
\draw (611,276.53) node [anchor=north west][inner sep=0.75pt]  [font=\tiny] [align=left] {$\displaystyle t$};
% Text Node
\draw (611,323.87) node [anchor=north west][inner sep=0.75pt]  [font=\tiny] [align=left] {$\displaystyle t$};
% Text Node
\draw (270.33,271.87) node [anchor=north west][inner sep=0.75pt]  [font=\tiny] [align=left] {$\displaystyle t$};
% Text Node
\draw (289.33,310.2) node [anchor=north west][inner sep=0.75pt]  [font=\tiny] [align=left] {$\displaystyle t$};
% Text Node
\draw (194,315.87) node [anchor=north west][inner sep=0.75pt]  [font=\tiny,color={black}  ,opacity=1 ] [align=left] {$\displaystyle s_{+}( \sigma )$};
% Text Node
\draw (40.5,282.5) node [anchor=north west][inner sep=0.75pt]  [font=\tiny] [align=left] {$\displaystyle \varepsilon ( \sigma ) +\frac{\rho ( \sigma )}{\sigma }( \sigma \ +\ \sigma ( t)) m_{+}( \sigma ) \ $};
% Text Node
\draw (321,19) node [anchor=north west][inner sep=0.75pt]   [align=left] {$\displaystyle G$};
% Text Node
\draw (464,192) node [anchor=north west][inner sep=0.75pt]   [align=left] {$\displaystyle H$};
% Text Node
\draw (321,229) node [anchor=north west][inner sep=0.75pt]   [align=left] {$\displaystyle F$};
% Text Node
\draw (111,188) node [anchor=north west][inner sep=0.75pt]   [align=left] {$\displaystyle \phi_{V} =F\circ H\circ G^{-1}$};
% Text Node
\draw (391,101) node [anchor=north west][inner sep=0.75pt]  [font=\tiny] [align=left] {$\displaystyle \tau =( \sigma ( t) +\sigma )\sin \theta $};
% Text Node
\draw (181,92) node [anchor=north west][inner sep=0.75pt]  [font=\tiny] [align=left] {$\displaystyle a_{*}$};
% Text Node
\draw (558,92) node [anchor=north west][inner sep=0.75pt]  [font=\tiny] [align=left] {$\displaystyle a_{*}$};
% Text Node
\draw (558,311) node [anchor=north west][inner sep=0.75pt]  [font=\tiny] [align=left] {$\displaystyle a_{*}$};
% Text Node
\draw (221.81,300.52) node [anchor=north west][inner sep=0.75pt]  [font=\tiny] [align=left] {$\displaystyle a_{*}$};

\end{tikzpicture}}
    \caption{A breakdown of the construction of $V_t$}
    \label{fig: Vt construction}
\end{figure}
 We can assume, without loss of generality, that there exists $t_0>0$ such that $Y'(t) \neq 0$ for all $t \in [-t_0,t_0]$. Our goal is to parametrize the curve $Y$ by arch-length. To this end, consider the function $\sigma $ defined by
 \[
 \sigma(0) = 0, \quad \sigma'(0) = |Y'(0)| = \frac{1}{\sin \theta}, \quad \sigma' = |Y'| > 0.
 \]
Then, the map $\sigma:(-t_0,t_0) \rightarrow (- \sigma_-,  \sigma_+) $ is $C^2$-diffeomorphism for some $\sigma_\pm >0$, and its inverse $\rho := \sigma^{-1}: [-\sigma_-,\sigma_+] \rightarrow [-t_0,t_0]$ satisfies the properties $\rho(0) = 0$ and $\rho'(0) = \frac{1}{\sigma'(0)} = \sin \theta$. Using $\rho$, we reparametrize the curve $Y$ and its associated functions as 
\[
\varepsilon:= Y \circ \rho, \quad s_\pm := \lambda_\pm \circ \rho, \quad B_\pm := \gamma_\pm \circ s_\pm, \quad m_\pm := n_\pm \circ s_\pm,
\]
where the existence of $\lambda_\pm$ follows from the discussion preceding equation \eqref{def lambda func}.
With these definitions, we now proceed to construct $V_t$ and $\phi_V$. To do so, we begin by cutting the kite $K^{\sin \theta \sigma(t)}_\theta$ along the $x_1$-axis, and define its upper and lower halves as $K^\pm_t := K^{\sin \theta \sigma(t)}_\theta \cap \{ \pm x_2 >0 \}$ as illustrated in Figure \ref{fig: Vt construction}. 
Define the linear map $G_\pm: L^\pm_t \rightarrow K^\pm_t$ by 
\begin{align*}
G_\pm (\sigma, \tau ) &= \sigma \begin{pmatrix}
    1\\0
\end{pmatrix} + \tau \begin{pmatrix}
    -\sin \theta\\ \pm \cos \theta 
\end{pmatrix}.
  \end{align*}
Since $G_\pm$ is a bijective, linear map, it follows that $G_\pm: L^\pm_t \rightarrow K^\pm_t$ is a smooth diffeomorphism. Next, we define the triangles $L^\pm$ as 
\begin{align*}
  L_t^\pm &= \{ (\sigma, \tau) : - \sigma(t)< \sigma < \sigma(t), \,  0 < \pm \tau < (\sigma + \sigma(t)) \sin \theta  \}.
  \end{align*}
Using the map $ H(\sigma,\tau) = (\sigma, \frac{\tau \rho(\sigma )}{\sigma \sin \theta})$, we can further transform $L_t^\pm$ into the curved triangle
\[
\Lambda_t^\pm = \{ (\sigma, \tau ) : -\sigma(t) < \sigma < \sigma(t), \,  0 < \pm \tau < \frac{\rho(\sigma)}{\sigma} (\sigma + \sigma(t))  \}.
\]
Note that
\[
H'(\sigma,\tau) = \begin{pmatrix}
    1 & 0 \\
    \frac{\tau( \sigma \rho(\sigma)- \rho'(\sigma))}{\sigma^2 \sin \theta} & \frac{\rho(\sigma)}{\sigma  \sin \theta}
\end{pmatrix}, \quad \det (H'(\sigma,\tau)) =  \frac{\rho(\sigma)}{\sigma  \sin \theta} = 1 + \mathcal{O}(\sigma) \neq 0 \text{ for $\sigma $ near 0}.
\]
Thus, by the inverse function theorem, it follows that $H: L_t^\pm \rightarrow \Lambda_t^\pm $ is a diffeomorphism for sufficiently small $t$. We can now parametrize a neighborhood of $\Gamma_t$ via the maps 
\[
F_\pm: \Lambda_t \rightarrow F_\pm(\Lambda_t), \quad  (\sigma,\tau) \mapsto \varepsilon(\sigma) + \tau m_\pm (\sigma),
\]
and define the open set 
\[
V_t := \left(\overline{F^+(\Lambda_t^+)} \cup \overline{F^-(\Lambda_t^-)}\right)^\circ.
\]
Using again the inverse function theorem, one sees that $F_\pm$ is a diffeomorphism since we have
\[
F_\pm'(0,0) = (\varepsilon'(0) \quad  m_\pm(0)) = \begin{pmatrix}
    1 & -\sin \theta \\
    0 & \pm \cos \theta
\end{pmatrix}, \quad \det (F_\pm') = \pm \cos \theta \neq 0.
\]
With this, we are now able to define the bi-Lipschitz diffeomorphism 
 \[
 \phi_V = F_\pm \circ H\circ G^{-1}_\pm: K^R_\theta \rightarrow V_t \text{ for } \pm x_2 >0
 \]
which can be extended continuously to $x_2  = 0$ by
\[
F_\pm \circ H\circ G^{-1}_\pm(x_1,0) = F_\pm \circ H (x_1,0) = F_\pm (x_1,0) = \varepsilon(x_1)  .
\]
By construction, $\phi_V$ is $C^2$-smooth on $K^t_\pm$ and continuous up to the boundary $x_2 = 0$. Therefore, $\phi_V$ is Lipschitz and maps $\Gamma_\theta^{r(t)} = \{ x \cdot(\cos \theta, \pm\sin \theta  ): x \in (0, \sigma(t) / \cos \theta ) \}$ to $\Gamma_t$ and $ \partial_*K^R_\theta $ to $\partial_* V_t$, i.e.,
\begin{align*} 
F \circ H \circ G(\Gamma^{r(t)}_\theta) &= F(\{ (\pm\rho(\sigma), \sigma) \mid \sigma \in (0,\sigma(t))  \}) = \Gamma_t ,\\
 F \circ H \circ G(\partial_* K^{r(t)}_\theta) &= F(\{ (\sigma(t), \sigma): \sigma \in (0,2t)  \}) = \partial_*V_t.
 \end{align*}
%It is evident by construction that $V_t$ is a suitable neighborhood of $\Gamma_t$.
Now we compute
\[
\phi_V'(0,0) = F'(0,0)H'(0,0)G(0,0) = \begin{pmatrix}
    1 & -\sin \theta\\ 0 & \pm \cos \theta
\end{pmatrix}
\begin{pmatrix}
    1&0\\0&1
\end{pmatrix}
\begin{pmatrix}
    1 & \pm \tan \theta \\ 0 & \pm \frac{1}{\cos \theta}
\end{pmatrix}
=I_2,
\]
and since $ F_\pm \circ H \circ G^{-1}_\pm $ is $C^1$-smooth, it follows that $\phi'_V = I_2 + \mathcal{O}(| x |)$ as $|x| \rightarrow 0$. The inverse function theorem implies that $\phi_V^{-1}$ is also $C^1$-smooth on $ V^\pm_t$. Moreover, the continuity of $\phi_V^{-1}$ along the curve $\varepsilon(\sigma)$ ensures that it is Lipschitz continuous. Therefore, $\phi_V$ is bi-Lipschitz, and (iii) is proved.

Finally, (i) follows by defining $r(t) := \sin \theta \sigma(t)$, and we see that $r'(0) = \frac{\sin \theta}{\sin \theta} =1 $, and this concludes the proof.
\end{proof}
The following lemma provides further properties of $V_t$.
\begin{lemma}\label{lem: norm estimate in V_ t}
There exist $0 < a < b$ such that, for all sufficiently small $t>0$, we have $B_{at }(0) \subset V_t \subset B_{bt}(0)$. In particular, for any $c \in (0,1)$ there exist $0<\tilde{a}<\tilde{b}$ such that $V_{\tilde{a}t} \subset B_{ct}(0) \subset B_t(0) \subset V_{\tilde{b}t}$ holds for sufficiently small $t$.
\end{lemma}
\begin{proof}   Let $\phi_V$ and $r$ be as in Lemma \ref{lem: straigthen curved kite}. Then,  $y\in V_t \Leftrightarrow y=\phi_V(x)$ for some $x \in K_\theta^{r(t)}$ with $\phi'_V(x) = I_2 + \mathcal{O}(| x |)$. Thus, for $| x |$ sufficiently small, a Taylor expansion of $\phi_V$ shows that $\frac{1}{2}| y| \leq | x | \leq 2 | y|$. Note that $B_{r(t)}(0) \subset K^{r(t)}_\theta \subset B_{\frac{r(t)}{\sin \theta}} (0)$,  and since $r(t) = t+\mathcal{O}(t^2)$ as $t \rightarrow 0$, it follows by another Taylor expansion that there is $0<a\leq 1$ such that
    \[
   a t < |x|\leq  2| y| \quad\text{and} \quad    \frac{1}{2} | y| \leq|x|< \frac{2 t}{\sin\theta}=: \frac{b}{2},
    \]
 holds for sufficiently small $t$. Now, let  $c\in(0,1)$ and remark that for $t$ small enough,  $r$ is invertible with $r^{-1}(t) = t + \mathcal{O}(t^2)$. This implies that $K^t_\theta = K^{r(r^{-1}(t))}_\theta$  and thus $\phi_V(K^t_\theta) = V_{r^{-1}(t)} \subset V_{2t}$ for sufficiently small $t>0$. Therefore, similar arguments as before show that $V_{\tilde{a}t}\subset B_{ct}\subset B_t\subset V_{\tilde{b}}$ holds true with $\tilde{a}=\frac{c\sin\theta }{2}$ and $\tilde{b}=\frac{2}{a}$. This concludes the proof.
\end{proof}
Next, we construct a family of smooth cutoff functions in $V_t$ that satisfy the transmission condition on $\Gamma_t$.
\begin{lemma}
\label{lem: cuttoff function for V delta}
There exist constants $\eta > 0$ and $0< a < b <1 $ such that for some $\delta_0 > 0 $, $ C > 0 $, and for every $\delta \in (0,\delta_0)$, there exist $C^2$-smooth functions $\chi_\delta: \overline{V_\eta} \rightarrow \rr$ satisfying the following properties:
\begin{itemize}
    \item[(1)] $0 \leq \chi_\delta \leq 1$.
    \item[(2)] For all $\beta \in \mathbb{N}^2$ the uniform estimate $\| \partial^\beta \chi_\delta \|_{\infty} \leq C \delta^{- | \beta |}$ holds.
    \item[(3)] $\chi_\delta \equiv 1$ in $V_{a \delta}$.
    \item[(4)] $\supp \chi_\delta \subset V_{b \delta}$.
    \item[(5)] The normal derivative of $\chi_\delta$ vanishes on $\Gamma_\pm$.
\end{itemize}
\end{lemma}
\begin{proof} Recall the definition of $\gamma_\pm$, $T_\pm$, and $n_\pm$ from Notation \ref{notation param}. Let $t_0 > 0 $ and $s_0 > 0$, and consider the coordinates maps
    \[
    \phi_\pm: (-s_0,s_0) \times (-t_0,t_0) \rightarrow \rr^2 , \quad \phi_\pm(s,t) = \gamma_\pm(s) - t N_\pm(s).
    \]
It is straightforward to verify that the Jacobian matrix $\phi'$ of $\phi_\pm$ and its determinant are given by 
    \[ \phi_\pm'(s,t) = \begin{pmatrix}
        (1-tk_\pm(s))T_1^\pm(s) & -n_1^\pm(s) \\  (1-tk_\pm(s))T_2^\pm(s) & -n_2^\pm(s)
    \end{pmatrix}, \quad \det(\phi_\pm'(s,t)) = (1-tk_\pm(s))^2.
    \]
    Choosing $s_0$ and $t_0$ sufficiently small, it follows that $\phi$ is invertible. Consequently, there exists a constant $\eta > 0$ such that the maps $(s,t) \mapsto \phi_\pm(s,t)$ can be inverted to yield $C^2$-smooth functions $s_\pm$ and $t_\pm$ with the following properties, as guaranteed by the inverse function theorem:
    \begin{align*}
    &\nabla s_\pm(x) = \pm \frac{1}{1-t_\pm(x) K_\pm(x)} (N_2^\pm(x), - N_1^\pm (x)) , \quad K_{\pm} := k_\pm \circ s_\pm , \quad N^\pm_j := n^\pm_j \circ s_\pm.
     \end{align*}
     Hence, $s_\pm(0,0) = (0,0)$ and $\nabla s_\pm(0,0) = (\cos\theta, \pm \sin \theta)$, and we deduce that 
\begin{align*}
     s_\pm (x_1,x_2) = \langle (\cos \theta, \pm \sin\theta),(x_1,x_2) \rangle + \mathcal{O} (x_1^2 + x_2^2) \quad  \text{as } (x_1,x_2) \rightarrow (0,0).
     \end{align*}
  Moreover, we have $s_\pm(\gamma_\pm(s)) = s$ and as $s\to0$ there holds
    \[
    s_\pm(\gamma_\mp(s)) =  (\cos\theta, \pm \sin \theta)\cdot \gamma'_\pm(0)s + \mathcal{O}(s^2) = \cos(2\theta)s + \mathcal{O}(s^2). \quad 
    \]
    Keeping this notation in mind, we further set $c:= \min(\sin \theta, \cos \theta) \cdot \max (\frac{1}{2}, \left| \cos 2 \theta \right| ) \in (0,1) $. In view of Lemma \ref{lem: norm estimate in V_ t}, there are $0 < \tilde{a} < \tilde{b}$ such that
    \[
    V_{\tilde{a} \tilde\delta} \subset B_{c \tilde{\delta} }(0) \subset B_{\tilde{\delta}}(0) \subset V_{\tilde{b} \tilde{\delta}}
    \]
    holds for all sufficiently small $\tilde{\delta} > 0$. Now, let $d_0 > 0$ be such that $b := \tilde{b} \cdot d_0 <1 $ and set $a : = \tilde{a} \cdot d_0 $, $ \delta := \tilde{\delta} /d_0  $. Then, one has
    \[
    V_{a \delta} \subset B_{c d_0 \delta}(0) \subset B_{d_0 \delta}(0) \subset V_{b \delta}.
    \]
    We are now going to construct a cutoff function $\chi_\delta$ that satisfies assertions  (1), (2), (5), and such that
    \[
    \chi_\delta \equiv 1 \quad \text{in } B_{c d_0 \delta}(0) \quad \text{ and } \quad 
    \supp\chi_\delta \subset B_{d_0 \delta}(0),
    \]
    which then yields the desired result. For this, we define the constants
    \[
    c_0 := c \cdot d_0 = \min(\sin \theta, \cos \theta) \cdot \max (\frac{1}{2}, \left| \cos 2 \theta \right| ) \cdot d_0 ,  \quad c_1 := \min(\sin \theta, \cos \theta) \cdot d_0 > c_0,
    \]
    and for a sufficiently small fixed $\varepsilon > 0$, consider a smooth  function $\chi: \mathbb{R} \rightarrow [0,1]$ with
    \[
    \chi \equiv 1 \text{ in } [- c_0 - \varepsilon, c_0 + \varepsilon] \quad \text{ and }\quad \supp \chi \subseteq [ - c_1 + \varepsilon, c_1 - \varepsilon],
    \]
    and define the cutoff function $\chi_\delta$ by
    \[
    \chi_\delta (x) := \chi\left(\frac{s_+(x)}{\delta}\right) \cdot \chi\left(\frac{s_-(x)}{\delta}\right).
    \]
    By construction,  $\chi_\delta$ satisfies assertions (1) and (2). To verify assertion (3),  note that for $x \in B_{c_0 \delta}(0) $ and sufficiently small $\delta>0$,  we have
    \[
    \frac{\left|s_\pm(x)\right|}{\delta} = \frac{\left|\langle(\cos \theta, \pm \sin\theta),(x_1, x_2) \rangle + \mathcal{O}(x_1^2 + x_2^2) \right|}{\delta} \leq \frac{| x | + \mathcal{O}(| x |^2)}{\delta} \leq \frac{c_0 \delta + \mathcal{O}(\delta^2)}{\delta} \leq  c_0 + \varepsilon,
    \]
    which yields $\chi(s_\pm(x) /\delta) = 1$ in $ B_{c_0 \delta}(0)$, and this implies (3). To prove assertion (4), it suffices to show that outside the ball $B_{d_0\delta}(0)$, either $\chi(s_+(x) /\delta) \equiv 0$ or $\chi(s_-(x) /\delta) \equiv 0$. Let $ | x | \geq d_0 \delta$. For $x_1 \geq 0$,  $x_2 < 0$, and sufficiently small $\delta > 0$, we have
    \begin{align*}
    \frac{\left|s_-(x)\right|}{\delta} &=  \frac{\left|\cos \theta x_1 - \sin\theta x_2 + \mathcal{O}(x_1^2 + x_2^2)\right|}{\delta}  = \frac{\left|\cos \theta |x_1| + \sin \theta |x_2|\right|}{\delta} + \mathcal{O}(\delta) \\
    &\geq \min(\sin \theta, \cos\theta) \frac{|x_1 | + |x_2|}{\delta} + \mathcal{O}(\delta) \geq \frac{c | x |}{\delta} + \mathcal{O}(\delta) \geq c_1 - \varepsilon. 
    \end{align*}
    An analogous argument applies for $x_1 <0  $ and $x_2 \geq 0$. In the case where $x_1$ and $x_2$ have the same sign, the same estimate can be established similarly for $|s_+(x)/\delta|$, which concludes the proof of (4).
    
    We now prove (5) for $\Gamma_+$ with $s>0$, as the case $\Gamma_-$ with $s < 0$ follows analogously. For $s>0$ and $x=\gamma_+(s)\in\Gamma_+$ we have
    \begin{align*}
    \frac{\partial \chi_\delta}{\partial n_+} (x) =& \langle n_+(s), (\nabla\chi_\delta) (\gamma_+(s))  \rangle 
    = \frac{1}{\delta} \chi'\left(s_+(\gamma_+(s))/\delta\right) \chi(s_-(\gamma_+(s))/\delta) \cdot  \langle n_+(s), (\nabla s_+)(\gamma_+(s)) \rangle  \\
    &+ \frac{1}{\delta} \chi'(s_-(\gamma_+(s))/\delta) \chi(s_+(\gamma_+(s))/\delta)  \cdot  \langle n_+(s), (\nabla s_-)(\gamma_+(s)) \rangle.  
    \end{align*}
    As  $(\nabla s_+)(\gamma_+(s)) = (n_2^+ (s), - n_1^+(s))$, it follows that $\langle n_+(s), (\nabla s_+)(\gamma_+(s)) \rangle = 0$, and thus
    \[
    \frac{\partial \chi_\delta}{\partial n_+} (x) =\frac{1}{\delta} \chi'(s_-(\gamma_+(s))/\delta) \chi(s/\delta)  \cdot  \langle n_+(s), (\nabla s_-)(\gamma_+(s)) \rangle.
    \]
    Since $ \chi(s/\delta) = 0 $ for $ \left|s/\delta\right| \geq c_1 - \varepsilon$, it suffices to prove that  $\chi'(s_-(\gamma_+(s))/\delta)$ vanishes for all $|s| < \delta(c_1 - \varepsilon)$. By expanding $s_-(\gamma_+ (s))$, we have
    \[
    s_-(\gamma_+ (s)) = \langle (\cos \theta, - \sin \theta), \gamma_+'(0) \rangle + \mathcal{O}(s^2) = \cos(2 \theta)s + \mathcal{O}(s^2).
    \]
    For sufficiently small $\delta$, this implies
    \[
    \frac{\left|s_-(\gamma_+(s))\right|}{\delta} = \frac{\left|\cos(2 \theta )s + \mathcal{O}(s^2)\right|}{\delta} \leq \left|\cos(2 \theta)\right| (c_1 - \varepsilon) + \mathcal{O}(\delta) \leq c_0 + \varepsilon.
    \]
    As $\chi$ is constant in $[-c_0- \varepsilon, c_0 + \varepsilon ]$, its derivative vanishes there. Consequently, the normal derivative of $\chi_\delta$ vanishes identically on $\Gamma_+$, which shows (5) and completes the proof of the lemma.
\end{proof}

\subsection{Spectral properties}
After defining appropriate neighborhoods near each corner of the curve supporting the $\delta$-interaction, we proceed to analyze the analogues of the operators introduced in Section \ref{ope straight} within this context. In particular, we derive the asymptotics of their first $\kappa(\theta)$ eigenvalues (see Corollary \ref{lem: ev dirichlet neuman truncated curved} below), which constitutes a key step in proving Theorem \ref{theo: corner induced} on the asymptotic behavior of corner-induced eigenvalues.
\begin{definition}
 Let $V_\delta$ be as in Lemma \ref{lem: straigthen curved kite}. For $x\in\{ d,n \}$, we define the sesquilinear form of the Dirichlet/Neumann Laplacian with a strong $\delta$-interaction on $\Gamma_\delta$ by
    \[ x^\delta_{\Gamma, \alpha} (u) = \int_{V_\delta} |\nabla u|^2 \mathrm{d}x  - \alpha \int_{\Gamma_\delta} |u|^2 \dS, \quad  D(d^\delta_{\Gamma,\alpha}) = H^1_0(V_\delta), \quad  D(n^\delta_{\Gamma,\alpha}) = H^1(V_\delta).
    \]
We further set  $V_{\delta,\rho} := V_\delta \setminus \overline{V_\rho} $ and $ \Gamma_{\delta,\rho}: = \Gamma_\delta \cap V_{\delta,\rho}\, $ for $0<\rho<\delta$, and define
        \[
        p^{\delta,\rho}_{\Gamma,\alpha}(u) = \int_{V_{\delta,\rho}} |\nabla u|^2 \mathrm{d}x - \alpha \int_{\Gamma_{\delta,\rho}} |u|^2 \dS, \quad D(p^{\delta,\rho}_{\Gamma,\alpha}) = H^1(V_{\delta,\rho}).
        \]
\end{definition}
The existence of the diffeomorphism between $V_\delta$ and $K^{r(t)}_{\theta}$ allows us to apply a change of variables, which facilitates a direct comparison between the operators arising in these respective domains.
\begin{lemma}
\label{lem: unitary equivalence for V delta }
    There exist $a_0,a_1,\delta_0 >0$ such that for all $\delta \in (0,\delta_0)$, $\rho \in (0,\delta)$, and  $n \in \mathbb{N}$, the following inequalities hold:
    \begin{align*}
        (1-a_0\delta) E_n(N^{r(\delta)}_{\theta,\alpha(1+a_1 \delta)})& \leq E_n(N^\delta_{\Gamma,\alpha}) \leq(1+a_0\delta) E_n(N^{r(\delta)}_{\theta,\alpha(1-a_1 \delta)}), \\
    (1-a_0\delta) E_n(D^{r(\delta)}_{\theta,\alpha(1+a_1 \delta)}) &\leq E_n(D^\delta_{\Gamma,\alpha}) \leq(1+a_0\delta) E_n(D^{r(\delta)}_{\theta,\alpha(1-a_1 \delta)}),\\
         (1-a_0\delta) E_n(P^{r(\delta),r(\rho)}_{\theta,\alpha(1+a_1 \delta)}) &\leq E_n(P^{\delta,\rho}_{\Gamma,\alpha}) \leq(1+a_0\delta) E_n(P^{r(\delta),r(\rho)}_{\theta,\alpha(1-a_1 \delta)}),
         \end{align*}
        as $\alpha \to \infty$, $\delta \to 0^+ $, and $ \alpha \delta \to \infty$.
\end{lemma}
\begin{proof}We provide the detailed estimates only for $D^\delta_{\Gamma,\alpha}$, as the same argument applies analogously to $N^\delta_{\Gamma,\alpha}$ and $P^{\delta,\rho}_{\Gamma,\alpha}$. By Lemma \ref{lem: straigthen curved kite}, there exist $\delta_0 > 0 $ and a function $r \in C^2$ satisfying $r(0)=0$ and $r'(0) = 1$, such that for every $\delta \in (0, \delta_0)$, the map $\phi_V: K^{r(\delta)}_\theta \rightarrow V_\delta $ is bi-Lipschitz with $\phi_V'(x) = I_2 + \mathcal{O}(x)$.  Consider the map $\Phi : L^2(V_\delta) \rightarrow L^2(K^{r(\delta)}_\theta)$ defined by $ \Phi(v)= v \circ \phi_V =: u$. Then, $\Phi: H^1(V_\delta) \mapsto H^1(K^{r(\delta)}_\theta)$ is bijective, and the boundary condition $u = 0$ on $ \partial K^{r(\delta)}_\theta $ holds if and only if $v = 0$ in $ \partial V_\delta$. To apply the Min-Max principle for the required estimates, we first need to find suitable approximations for the norms $ \| v \|_{V_\delta}$, $\| \nabla v \|_{V_\delta}$, and $\| v \|_{L^2(\Gamma_\delta)}$.
     
    We start with $ \| v \|^2_{V_\delta}$. Using the change of variables $u=v \circ \phi_V $, we obtain 
    \[
    \int_{V_\delta} | v |^2 \mathrm{d}x  = \int_{K^{r(\delta)}_\theta}|u|^2 |\det (\phi_V')| \mathrm{d} x  = \int_{K^{r(\delta)}_\theta}|u|^2 ({1 + \mathcal{O}(|x|)}) \mathrm{d} x.
    \]
    Hence, there exists $b_1>0$ such that $ 1- b_1 \delta \leq |\det (\phi_V')| \leq 1 + b_1 \delta $,
    and thus
    \[
    (1-b_1\delta)\int_{K^{r(\delta)}_\theta} |u|^2 \mathrm{d}x \leq \int_{V_\delta} | v |^2 \mathrm{d}x \leq (1+b_1\delta)\int_{K^{r(\delta)}_\theta}|u|^2 \mathrm{d}x,
    \]
    which is equivalent to $(1-b_1\delta) \| u \|^2_{K^{r(\delta)}_\theta} \leq \| v \|^2_{V_\delta} \leq (1+b_1\delta) \| u \|^2_{K^{r(\delta)}_\theta}$.
    Similarly,
    \[
    \int_{V_\delta} |\nabla v|^2 \mathrm{d}x  =  \int_{K^{r(\delta)}_\theta} \sum_{j,k = 1}^2 G^{j,k} \partial_j u \, \partial_k u  |\det (\phi_V')| \mathrm{d} x,
    \]
    where the matrix $G^{j,k}$ is the inverse of $G_{j,k} = \langle \partial_j \phi_V, \partial_k \phi_V \rangle =\delta_{j,k} + \mathcal{O}(\delta) $, and therefore satisfies
    \[
    G^{j,k}=(1 + \mathcal{O}(\delta))^{-1}(\delta_{j,k} + \mathcal{O}(\delta)) =(1 + \mathcal{O}(\delta))(\delta_{j,k} + \mathcal{O}(\delta)) = \delta_{j,k} + \mathcal{O}(\delta).
    \]
    Consequently,  there exists $b_2 > 0$ such that for any $u\in H^1(K^{r(\delta) }_\theta)$ one has
    \[
    (1-b_2\delta) |\nabla u|^2 \leq \sum_{j,k = 1}^2 G^{j,k} \partial_j u \, \partial_k u \leq (1+b_2\delta) |\nabla u|^2,
    \]
    and we deduce that
        \[
        (1-b_2\delta) \int_{K^{r(\delta) }_\theta} |\nabla u|^2 \mathrm{d}x
        \leq \| \nabla v \|^2_{V_\delta}
        \leq  (1+b_2\delta) \int_{K^{r(\delta) }_\theta} |\nabla u|^2 \mathrm{d}x.
        \]
    To estimate $\| v \|_{L^2(\Gamma_\delta)}$,  consider an arc-length parametrization $\gamma:I_\delta \rightarrow \Gamma^\delta_\theta$. Then, $ \phi(\gamma (t)) $ is a parametrization of $\Gamma_\delta $, and we have
    \[
    \int_{\Gamma_\delta} | v |^2 \dS = \int_{I_\delta}  | u(t)|^2 |\phi'(\gamma(t)) \gamma'(t) | \mathrm{d}t.
    \]
Using the fact that  $\phi'(x) = I_2 + \mathcal{O}(|x|)$, it follows that there exists $b_3 > $ such that
\[
1- b_3 \delta \leq |\phi'(\gamma(t)) \gamma'(t)| \leq 1 + b_3 \delta,
\]
Therefore,
\[
(1- b_3 \delta) \int_{\Gamma_\theta^{r(\delta)}} |u|^2 \dS 
\leq \int_{\Gamma_\delta} | v |^2 \dS \leq  (1+ b_3 \delta) \int_{\Gamma_\theta^{r(\delta)}} |u|^2 \dS,
\]
which yields the desired estimate for $\| v \|_{L^2(\Gamma_\delta)}$. By combining the previous estimates, we obtain
\[
\frac{d^\delta_{\Gamma,\alpha}(v)}{\| v  \|^2_{L^2(V_\delta)}} 
\geq \frac{1-b_2\delta}{1+b_1  \delta} \frac{\int_{K^{r(\delta)}_\theta} |\nabla u|^2 \mathrm{d}x }{\| u \|^2_{L^2(K^{r(\delta)}_\theta)}} 
- \alpha \frac{1+ b_3 \delta}{1-b_1 \delta} \frac{\int_{\Gamma_\theta^{r(\delta)}} |u|^2 \dS}{\| u \|^2_{L^2(K^{r(\delta)}_\theta)}},\quad \forall v\neq0.
\]
Moreover, there exist $a_0, a_1 > 0$ such that for all sufficiently small $\delta > 0 $, 
 \[
 \frac{1-b_2\delta}{1+b_1  \delta}=1+ \mathcal{O}(\delta) \geq 1- a_0 \delta\quad\text{and}\quad  \frac{1}{1-a_0 \delta} \cdot  \frac{1+b_3\delta}{1-b_1  \delta} \leq  1+ a_1 \delta,
 \]
which allows us to write the further bound
 \[
 \frac{d^\delta_{\Gamma,\alpha}(v)}{\| v  \|^2_{L^2(V_\delta)}} 
\geq (1-a_0 \delta) \frac{ \| \nabla u  \|^2_{K ^{r(\delta)}_\theta } - \alpha(1+ a_1 \delta) \int_{\Gamma_\theta^{r(\delta)}} |u|^2 \dS }{\| u \|^2_{K_\theta^{r(\delta)}} } = (1- a_0 \delta) \frac{d^{r(\delta)}_{\theta, \alpha(1+a_1\delta) } (\Phi v)}{\| \Phi v\|^2_{K^{r(\delta) }_{\theta} } }.
\]
Since $\Phi_\delta$ is bijective, the Min-Max principle implies
\[
E_n(D^\delta_{\Gamma,\alpha}) \geq (1-a_0 \delta) \cdot E_n(D^{r(\delta)}_{\theta,\alpha(1+a_1\delta)} ).
\]
\end{proof}
The following eigenvalue asymptotics follow directly from the lemma above.
\begin{corollary}\label{lem: ev dirichlet neuman truncated curved}
As $\alpha \to \infty$, $ \delta \to 0^+$, and $\alpha \delta \to \infty$, the following asymptotics hold:
\begin{align*}
E_n(D^\delta_{\Gamma,\alpha}) &= \alpha^2(\mathcal{E}(\theta) + \mathcal{O}( \delta + e^{-c \alpha \delta})), \quad \forall n \in \{1,\dots,\kappa(\theta) \},\\
 E_n(N^\delta_{\Gamma,\alpha}) &= \alpha^2(\mathcal{E}(\theta) + \mathcal{O}( \delta + 1/(\alpha\delta)^2)), \quad  \forall n \in \{1, \dots,\kappa(\theta)\},\\
 E_{k(\theta) + 1} (D^\delta_{\Gamma,\alpha}) &\geq E_{\kappa(\theta) +1 }(N^\delta_{\Gamma,\alpha}) \geq -\frac{\alpha^2}{4} + o(\alpha^2) .
 \end{align*}
\end{corollary}
\begin{proof}
By Lemma \ref{lem: unitary equivalence for V delta }, for $X \in \{ N,D \}$ and any $n \in \nn $, we have
\[
(1 + \mathcal{O}(\delta))E_n(X^{r(\delta)}_{\theta,\alpha(1+a \delta)}) 
\leq E_n(X^\delta_{\Gamma,\alpha}) \leq 
(1 + \mathcal{O}(\delta))E_n(X^{r(\delta)}_{\theta,\alpha(1-a \delta)}).
\]
Since $r(\delta) = \mathcal{O}(\delta) $ as $\delta \to 0$, it follows that 
$r(\delta) \alpha \to \infty $ as $\alpha \delta \to \infty$. This allows us to apply Lemma \ref{lem: eigenvalues of dirichlet on karl} and Lemma \ref{lem: eigenvalues of neumann on karl}; in particular, there exists $c>0$ such that
\begin{align*}E_n(D^\delta_{\Gamma,\alpha}) &= (1 + \mathcal{O}(\delta))^3 (\mathcal{E}(\theta) + \mathcal{O}(e^{-c \alpha r(\delta)})) \alpha^2 = (\mathcal{E}(\theta) + \mathcal{O}(\delta + e^{-c \alpha \delta})) \alpha^2,\\ E_n(N^\delta_{\Gamma,\alpha}) &= (1 + \mathcal{O}(\delta))^3 \left(\mathcal{E}(\theta) 
+ \mathcal{O}\left(\frac{1}{(\alpha r(\delta))^2}\right)\right) \alpha^2 
= \left(\mathcal{E}(\theta)  + \mathcal{O}\left(\alpha^2 \delta + \frac{1}{\delta^2}\right) \right) \alpha^2,
\end{align*}
and furthermore, 
\begin{align*}
   E_{k(\theta) + 1} (D^\delta_{\Gamma,\alpha}) &\geq  E_{\kappa(\theta) +1 }(N^\delta_{\Gamma,\alpha}) \geq (1+ a_0 \delta)E_{\kappa(\theta)+1} (N^{r(\delta)}_{\theta,\alpha(1+ a_0 \delta)}) \\
   &\geq (1+ a_0 \delta)( -1/4 + o(1) ) \alpha^2 (1+ a_1 \delta)^2 =  -\frac{\alpha^2}{4} + o(\alpha^2).
    \end{align*}
   where the first inequality follows by monotonicity.
    \end{proof}
Note that the eigenvalues depend asymptotically only on the angle $\theta$  and are independent of the specific geometry of the curves beyond their intersection point. This result also extends to the following case.
\begin{corollary}\label{lem: ev p truncated curved} It holds that $E_1(P^{\delta,\rho}_{\Gamma,\alpha}) \geq -\alpha^2/4 + o(\alpha^2)$.
\end{corollary}
\begin{proof} The proof follows the same lines as the one in  Corollary \ref{lem: ev dirichlet neuman truncated curved}, by combining Lemma \ref{lem: unitary equivalence for V delta } and Lemma \ref{lem: Kite lower bound robin bc}.
\end{proof}
Similarly to the operator $H^\alpha_\theta$ (see Corollary \ref{lem: agmon estimate HAlphaTheta}), one can show that the eigenfunctions of $N^\delta_{\Gamma,\alpha}$ satisfy an Agmon-type estimate.
\begin{lemma}
\label{lem: agmon truncated curved}
Let $\psi^{\delta,n}_{\Gamma,\alpha}$ be an eigenfunction corresponding to the $n$th eigenvalue of $N^\delta_{\Gamma,\alpha}$. Then there exist $c,C >0$ such that for $\delta \to 0^+$ and $\alpha \delta \to \infty$, it holds that
\[
\int_{V_\delta} e^{c \alpha | x |} \left(\frac{1}{\alpha^2} | \nabla \psi^{\delta,n}_{\Gamma,\alpha}|^2 + |\psi^{\delta,n}_{\Gamma,\alpha}| \right) \mathrm{d} x \leq C \| \psi^{\delta,n}_{\Gamma,\alpha}\|_{L^2(V_\delta)}.
\]
\end{lemma}
\begin{proof}
Consider the function $f: V_\delta \rightarrow \rr$ defined by $f(x)= b | x |$ for some $b>0$ to be chosen later. 
Analogously to the proof of Lemma $\ref{lem: agmon estimate H1Theta}$ we can write:
\begin{align*}
    n^\delta_{\Gamma,\alpha}(e^{\alpha f} \psi ) &= \int_{V_\delta} |\nabla(e^{\alpha f}\psi)| \mathrm{d}x - \alpha \int_{\Gamma_\delta} e^{2 \alpha f} | \psi |^2 \dS \\
   & = \int_{V_\delta} e^{2 \alpha f} ((-\Delta\psi) \psi + \alpha^2 |\nabla f|^2 \psi^2 ) \mathrm{d}x
= \int_{V_\delta} e^{2\alpha f} (E_n(N^\delta_{\Gamma,\alpha}) +b^2\alpha^2 ) | \psi |^2 \mathrm{d}x.
\end{align*}
By Corollary \ref{lem: ev dirichlet neuman truncated curved}, we have $E_n(N^\delta_{\Gamma,\alpha}) = (\mathcal{E}_n(\theta) + o(1))\alpha^2$, and for any $\varepsilon>0$ it holds as $\alpha \to \infty$ that
\begin{equation} \label{eq: agmon curved 1}
    n^\delta_{\Gamma,\alpha}(e^{\alpha f}) \leq (\mathcal{E}_n(\theta) + b^2 + \varepsilon) \alpha^2 \int_{V_\delta} e^{2 \alpha f} | \psi |^2 \mathrm{d}x.
\end{equation}
Let $\eta \in (0,1)$ and set $\rho = \frac{L}{\alpha}$, where both $\eta$ and $L>0$ will be chosen later. Then we have
\begin{align*}
    n^\delta_{\Gamma,\alpha}(e^{\alpha f} \psi ) &=  \int_{V_\delta} |\nabla(e^{\alpha f})|^2 \mathrm{d}x - \alpha \int_{\Gamma_\delta} e^{2 \alpha f} | \psi|^2 \dS \\
    &= \eta \int_{V_\delta} | \nabla(e^{\alpha f} \psi) |^2 \mathrm{d}x  + (1-\eta)
 (n^\rho_{\Gamma, \frac{\alpha}{1-\eta}} { (e^{\alpha f} \psi) } + p^{\rho,\delta}_{\Gamma, \frac{\alpha}{1-\eta}} { (e^{\alpha f} \psi) } ) \\
 &\geq  \eta \int_{V_\delta} | \nabla(e^{\alpha f} \psi) |^2 \mathrm{d}x  
  + (1-\eta)
 \left(E_1(N^\rho_{\Gamma, \frac{\alpha}{1-\eta}}) \| e^{\alpha f} \psi \|^2_{L^2 (V_\rho)} + E_1(P^{\rho,\delta}_{\Gamma, \frac{\alpha}{1-\eta}} ) \| e^{\alpha f} \psi \|^2_{L^2(V_{\delta,\rho})}\right).
\end{align*}
By Corollary \ref{lem: ev dirichlet neuman truncated curved}  and Corollary \ref{lem: ev p truncated curved}, it holds that
\[
E_1(N^\rho_{\Gamma, \frac{\alpha}{1-\eta}}) \geq (\mathcal{E}_n(\theta) - \varepsilon) \frac{\alpha^2}{(1-\eta)^2}, \quad  E_1(P^{\rho,\delta}_{\Gamma, \frac{\alpha}{1-\eta}} )  \geq - \frac{(\frac{1}{4}+ \varepsilon) \alpha^2}{(1-\eta)^2},
\]
which can be substituted into the previous inequality to yield
\[
n^\delta_{\Gamma,\alpha}(e^{\alpha f} \psi ) \geq   \eta \int_{V_\delta} | \nabla(e^{\alpha f} \psi) |^2 \mathrm{d}x  
  + \alpha^2 \frac{\mathcal{E}_n(\theta) - \varepsilon}{(1-\eta)} \|e^{\alpha f} \psi \|^2_{L^2 (V_\rho)} 
  - \alpha^2 \frac{\frac{1}{4} - \varepsilon}{1- \eta} \|e^{\alpha f} \psi \|^2_{L^2(V_{\delta,\rho})}.
  \]
  Incorporating this bound into inequality  \eqref{eq: agmon curved 1} and rearranging terms yields
\begin{align*}
\eta \int_{V_\delta} | \nabla(e^{\alpha f})|^2 \mathrm{d}x  + \left( - \mathcal{E}_n(\theta) - b^2 - \varepsilon - \frac{\frac{1}{4} + \varepsilon}{1-\eta}\right)& \alpha^2 \|e^{\alpha f} \psi \|^2_{L^2 (V_{\delta,\rho})}  \\
&\leq\left(\mathcal{E}_n(\theta) + b^2 + \varepsilon - \frac{\mathcal{E}_1(\theta)-\varepsilon}{1-\eta}\right)\alpha^2\|e^{\alpha f} \psi \|^2_{L^2 (V_\rho)},
\end{align*}
which can be equivalently expressed as
 \[
 \eta \int_{V_\delta} |\nabla(e^{\alpha f})|^2 \mathrm{d}x  + a_0 \alpha^2 \|e^{\alpha f} \psi \|^2_{L^2 (V_{\delta,\rho})} \leq b_0 \alpha^2\|e^{\alpha f} \psi \|^2_{L^2 (V_\rho)},
 \]
where the constants are defined by 
\begin{align*}
a_0&:= - \mathcal{E}_n(\theta) - b^2 - \varepsilon - \frac{\frac{1}{4} + \varepsilon}{1-\eta}
= \frac{- \mathcal{E}_n(\theta)   -\frac{1}{4}  + (\eta b^2 - b^2 + \eta \mathcal{E}_n(\theta) -2 \varepsilon + \varepsilon \eta)  }{1-\eta},\\
b_0 &:= \mathcal{E}_n(\theta) + b^2 + \varepsilon - \frac{ \mathcal{E}_1(\theta)-\varepsilon}{1-\eta} = b^2 + \frac{ \mathcal{E}_n(\theta) - \mathcal{E}_1(\theta) - \eta \mathcal{E}_n(\theta) + 2 \varepsilon - \varepsilon \eta}{1-\eta}.
\end{align*}
Recall that $- \mathcal{E}_n(\theta)    -1/4> 0$ and $\mathcal{E}_n(\theta)(1-\eta) - \mathcal{E}_1(\theta)>0$. Hence, by choosing $\varepsilon>0$, $\eta\in(0,1)$ , and $b >0$ sufficiently small, we ensure that $ a_0 > 0 $ and $b_0 > b^2 > 0 $. Consequently, by Lemma \ref{lem: norm estimate in V_ t} there exists $a > 0 $ such that $ V_\rho \subset B_{a \rho}(0) $, and therefore, we get $ \alpha f(x) = \alpha b | x | \leq \alpha b a(L/\alpha) = abL$ for any $x \in  V_\rho$. Substituting this into the previous inequality, we obtain, for $b_1 = b_0 e^{2b aL} $, the estimate
\[
\eta \int_{V_\delta} | \nabla (e^{\alpha f}  \psi)|^2 \mathrm{d}x + a_0 \alpha^2 \int_{V_\delta,\rho} e^{2\alpha f} | \psi |^2 \mathrm{d} x \leq b_1 \alpha^2 \int_{V_\rho} | \psi |^2\mathrm{d}x.
\]
From here, it follows that
\begin{align} \label{eq: agmon curved 2}
\begin{split}
    \int_{V_\delta} | \nabla(e^{\alpha f} \psi)|^2 + &2b^2 \alpha^2 {e^{2\alpha f}} | \psi |^2 \mathrm{d} x \\
    &= \frac{\eta}{\eta} \int_{V_\delta} |\nabla(e^{\alpha f} \psi  )|^2 \mathrm{d} x + \frac{2b^2}{a_0}  a_0 \alpha^2 \int_{V_{\delta,\rho}} e^{2 \alpha f} | \psi |^2 \mathrm{d}x  + 2 b^2 \alpha^2 \int_{V_\rho} | \psi |^2 \mathrm{d}x\\
     &\leq \left(\frac{1}{\eta} b_1 + \frac{2b^2}{a_0} + 2b^2 \right) \alpha^2 \int_{V_\rho} | \psi |^2 \mathrm{d}x=: b_2 \alpha^2 \int_{V_\rho} | \psi |^2 \mathrm{d}x  \leq b_2 \alpha^2 \| \psi \|^2_{L^2(V_\delta)}.
     \end{split}
\end{align}
Next, we estimate the gradient term:
\begin{align*}
    |\nabla(e^{\alpha f } \psi)|^2 &\geq |e^{\alpha f} \nabla \psi| + b^2\alpha^2 |e^{\alpha f} \psi| - 2  |e^{\alpha f} \nabla \psi| |b \alpha e^{\alpha f} \psi| \\
    &\geq |e^{\alpha f} \nabla \psi| + b^2\alpha^2 |e^{\alpha f} \psi| - 2 ( \frac{1}{4} |e^{\alpha f} \nabla \psi|^2  + | b \alpha e^{\alpha f} \psi|^2)\\
    &\geq \frac{1}{2} |e^{\alpha f} \nabla \psi|^2 - b^2 \alpha^2|e^{\alpha f} \psi|^2. 
\end{align*}
Substituting this estimate into \eqref{eq: agmon curved 2} yields
\[\int_{V_\delta} e^{2 b \alpha | x |} (\frac{1}{2} | \nabla \psi |^2 + b^2 \alpha^2 | \psi |^2 ) \mathrm{d}x \leq b_2 \alpha^2 \| \psi \|^2_{L^2(V_\delta)},
\]
and the desired estimate follows with constants $ c =  2b$ and $C = b_2 (2+1/b^2)$.
\end{proof}

\subsection{Non-resonance condition in the curved setting} Recall that the non-resonant condition is specified in Definition \ref{def non_resonant}. In this section, we establish a crucial estimate in Corollary \ref{lem: upper bound LT in V delta }, which is essential for the proof of Theorem \ref{theo: side induced} regarding edge-induced eigenvalues. Theorem \ref{theo: side induced} requires the non-resonance condition to hold for all angles, underscoring that the proof of Corollary \ref{lem: upper bound LT in V delta } relies fundamentally on this assumption. Therefore, throughout this section, we assume that the angle $\theta$ is non-resonant.

First, consider the following sesquilinear form and its associated lower bound:
\begin{definition} Let $\partial_* V_\delta$ be as in Lemma \ref{lem: straigthen curved kite}. Define the sesquilinear form
    \[
    r^\delta_{\Gamma, \alpha} (u) = \int_{V_\delta} |\nabla u|^2 \mathrm{d}x  - \alpha \int_{\Gamma_\delta} |u|^2 \dS- \alpha \int_{\partial_* V_\delta} |u|^2 \dS , \quad   D(r^\delta_{\Gamma,\alpha}) = H^1(V_\delta).
    \]
\end{definition}
\begin{lemma}\label{lem: lower bound robin bc V delta} There exists a constant $c>0$ such that $R^\delta_{\Gamma,\alpha} \geq - c \alpha^2$.
\end{lemma}
\begin{proof}
    Arguing as in the proof of Corollary \ref{lem: ev dirichlet neuman truncated curved}, one can show that there exist $a_0,a_1 >0$ such that
    \[
    E_1(R^\delta_{\Gamma,\alpha}) \geq (1-a_0) E_1(R^{r(\delta)}_{\theta, \alpha(1+a_1\delta)}).
    \]
    Since $r(\delta) = \delta + \mathcal{O}(\delta^2)$, one has $r(\delta)\alpha \to \infty$ as $\delta \to 0$. Thus, by Lemma \ref{lem: Kite lower bound robin bc}, there exists $c>0$ such that 
    \[
    (1-a_0) E_1(R^{r(\delta)}_{\theta, \alpha(1+a_1\delta)})\geq  (1-a_0\delta)(-c\alpha^2)(1+a_1\delta) = -{c} \alpha^2(1+(a_1-a_0)\delta -\delta^2).
    \]
    Since $|(a_1-a_0)\delta -\delta^2|<1$ for sufficiently small $\delta$, we obtain 
    \[
    E_1(R^\delta_{\Gamma,\alpha}) \geq (1-a_0) E_1(R^{r(\delta)}_{\theta, \alpha(1+a_1\delta)})\geq   - \tilde{c}\alpha^2,
    \]
    for some $\tilde{c} >0$ when $\alpha$ is sufficiently large.
\end{proof}
Thanks to Lemma \ref{lem: straigthen curved kite}, we can also translate the non-resonance condition to the curved setting.
\begin{corollary}\label{lem: idk man i just got here} 
    There exists $c>0 $ such that for $\alpha \delta \to \infty$, $\delta \to 0 $, $\alpha \to \infty$, and $\alpha^2 \delta^3 \to 0$, there holds
    \[
    E_{\kappa(\theta) +1 } (D^\delta_{\Gamma,\alpha}) \geq E_{\kappa(\theta) +1 } (N^\delta_{\Gamma,\alpha}) \geq - \frac{\alpha^2}{4} + \frac{c}{\delta^2}.
    \]
\end{corollary}
\begin{proof} By Lemma \ref{lem: straigthen curved kite} there exist $a_0,a_1 > 0 $ and a $C^2$-smooth function $r$ with $r(0) = 0$, $r'(0) =1$ such that for all $n\in\nn$,  $E_n(N^\delta_{\Gamma, \alpha}) \geq (1-a_0\delta) E_n(N^{r(\delta)}_{\theta,\alpha(1+ a_1\delta)})$. Since $\theta$ is non-resonant and $|r(\delta)| = | \delta + \mathcal{O}(\delta) | \geq \delta/2$ for sufficiently small $\delta$, it follows that
    \begin{align*}E_{\kappa(\theta +1)} (N^{r(\delta)}_{\theta,\alpha(1+ a_1\delta)}) &\geq -\frac{\alpha^2}{4}(1+a_1\delta)^2 + \frac{c}{r(\delta)^2}\\
    &\geq - \frac{\alpha^2}{4} + \frac{1}{\delta^2}(\frac{c}{2} - \alpha^2 \delta^3 \frac{a_1}{2} - \alpha^2 \delta^4 \frac{a_1^2}{4} ) \geq -\frac{\alpha^2}{4} + \frac{c_1}{\delta^2} 
    \end{align*}
    for some $c_1 >0$, using the asymptotics $\alpha^2 \delta^3 \to 0$.  Combining the above with the initial bound, we get
    \begin{align*}
        E_{\kappa(\theta)+1}(N^\delta_{\Gamma, \alpha}) &\geq (1-a_0\delta) E_{\kappa(\theta)+1}(N^{r(\delta)}_{\theta,\alpha(1+ a_1\delta)}) \geq (1-a_0\delta)(-\frac{\alpha^2}{4} + \frac{c_0}{\delta^2})\\
     &= -\frac{\alpha^2}{4} + \frac{c_1}{\delta^2}(1-a_0\delta + a_0 \delta\frac{\alpha^2}{4} \geq -\frac{\alpha^2}{4} + \frac{c_2}{\delta^2}.
     \end{align*}
     since 
    for some $c_2 >0$ and sufficiently small $\delta$,  which completes the proof.
\end{proof}
We are now ready to prove the estimate needed for the proof of Theorem \ref{theo: side induced}.
\begin{corollary}\label{lem: upper bound LT in V delta }
    Let $ \mathcal{L}$ be the subspace spanned by eigenfunctions corresponding to the first $\kappa(\theta)$ eigenvalues of $N^\delta_{\Gamma,\alpha}$. Then there exists $b> 0 $ such that as $\alpha \delta \to \infty$, $\delta \to 0 $, $\alpha \to \infty$, and $\alpha^2 \delta^3 \to 0$, there hold
    \begin{align*} \| v \|^2_{L^2(V_\delta)} &\leq  b \delta^2 \left(n^\delta_{\Gamma,\alpha} (v) + \frac{\alpha^2}{4} \| v  \|^2_{L^2(V_\delta)} \right), \quad \forall v \in H^1(V_\delta) \cap \mathcal{L}^{\perp},\\
    \int_{\partial_* V_\delta} | v |^2 \dS  &\leq b \alpha \delta^2 \left(n^\delta_{\Gamma,\alpha}(v) + \frac{\alpha^2}{4} \| v  \|^2_{L^2(V_\delta)} \right) , \quad \forall v \in H^1(V_\delta) \cap \mathcal{L}^{\perp} .
    \end{align*}
\end{corollary}
\begin{proof} Let $v \in H^1(V_\delta) \cap \mathcal{L}^{\perp}$. Then, Corollary \ref{lem: idk man i just got here} together with the spectral theorem yields 
\[
n^\delta_{\Gamma,\alpha}(v) \geq E_{\kappa(\theta)+1}(N^\delta_{\Gamma,\alpha}) \geq -\frac{\alpha^2}{4} \| v  \|^2_{L^2(V_\delta)} + \frac{c}{\delta^2} \| v  \|^2_{L^2(V_\delta)},
\]
which gives the first inequality.  To prove the second inequality, note that for any $u \in  D(N^\delta_{\Gamma,\alpha}) = D(R^\delta_{\Gamma,\alpha})=H^1(V_\delta)$, Lemma \ref{lem: lower bound robin bc V delta} implies 
\[
n^\delta_{\Gamma,\alpha}(u) - \alpha \int_{\partial_* V_\delta} |u|^2 \dS  = r^\delta_{\Gamma,\alpha}(u) \geq -c_0 \alpha^2 \| u \|^2_{L^2(V_\delta)},
\]
which is equivalent to
\[
 \int_{\partial_* V_\delta} |u|^2 \dS \leq \frac{1}{\alpha} n^\delta_{\Gamma,\alpha}(u) + c_0 \alpha \| u \|^2_{V_\delta}.
\]
Applying the first inequality for $u \in H^1(V_\delta) \cap \mathcal{L}^{\perp}$ yields
\begin{align*}\int_{\partial_* V_\delta} |u|^2 \dS &\leq \frac{1}{\alpha} n^\delta_{\Gamma,\alpha}(u) + c_0 \alpha b\delta^2\left(n^\delta_{\Gamma,\alpha}(u) + \frac{\alpha^2}{4}\| u \|^2_{L^2(V_\delta)} \right)\\
	&\leq \big(\frac{1}{\alpha} + c_0b \alpha \delta^2\big)\left(n^\delta_{\Gamma,\alpha}(u) + \frac{\alpha^2}{4}\| u \|^2_{L^2(V_\delta)} \right),
\end{align*}
which gives the second inequality by noting that $ \frac{1}{\alpha} = \alpha \delta^2 \cdot (\frac{1}{\delta \alpha})^2 = o(\alpha \delta^2)$ as $\alpha \delta \to \infty$ and $\alpha \to \infty$.
\end{proof}

\section{Neighborhoods of curved edges}\label{sec: W delta}
The main goal of this section is to construct an appropriate neighborhood around a smooth open arc contained in an edge of the piecewise smooth curve supporting the $\delta$-interaction. Within these neighborhoods, we analyze the spectral properties of Dirichlet and Neumann Laplacians subjected to the $\delta$-interaction localized on the specific open arcs. 

\subsection{Geometric setting and change of variables} We begin with the geometric construction of a tubular neighborhood around a given arc of a curve. Namely, throughout this section, for some $l>0$, we consider an open arc $\Gamma$ defined by an arc-length parametrization $\gamma: [0,l] \rightarrow \rr^2$, where $\gamma$ is $C^{3}$-smooth injective function with $|\gamma^\prime|=1$. At each point $\gamma(s)\in\Gamma$, we denote by $\tau(s) := \gamma'(s)$ the tangent vector and by $\nu(s)$ the normal vector with the convention that $\tau(s)\wedge\nu(s)=-1$. We further denote by $k(s)$ the curvature of $\Gamma$ defined by $\nu'(s)=k(s)\tau(s)$ at each point $\gamma(s)$, and we set $k_{max}:= \| k\|_\infty$ for the maximal curvature of $\Gamma$. Note that the restriction of $k$ onto any subinterval strictly contained in $(0,l)$ will still be denoted by $k$, the meaning being clear from the context.

To construct a tubular neighborhood $W_\delta$ around parts of the curve $\Gamma$ that shrinks as the parameter $\delta>0$ tends to zero, we proceed as follows. Let $\delta_0 > 0$ be fixed, and define the $C^1$-smooth adjustment functions 
\[
\lambda_{0}, \, \lambda_{l} : [0, \delta_0) \rightarrow [0,\infty), \quad \lambda_{0}(0)=\lambda_{l}(0) = 0, \quad \lambda_{0}'(0), \lambda_{l}'(0) \geq 0,
\]
along with the mapping 
\[
\phi_W:(0,l) \times (- \delta_0,\delta_0) \rightarrow \rr^2, \quad (s,t) \mapsto \gamma(s)-t \nu(s).
\]
Finally, define 
 \[
 I_\delta := (\lambda_0(\delta), l - \lambda_l(\delta)) , \quad  \Pi_\delta := I_\delta \times (-\delta, \delta), \quad W_\delta:= \phi(\Pi_\delta) , \quad \Gamma_\delta := \phi(I_\delta \times \{0\}).
 \]
A visualization of $\Pi_\delta$ and $W_\delta$ can be found in Figure \ref{fig: W delta}.
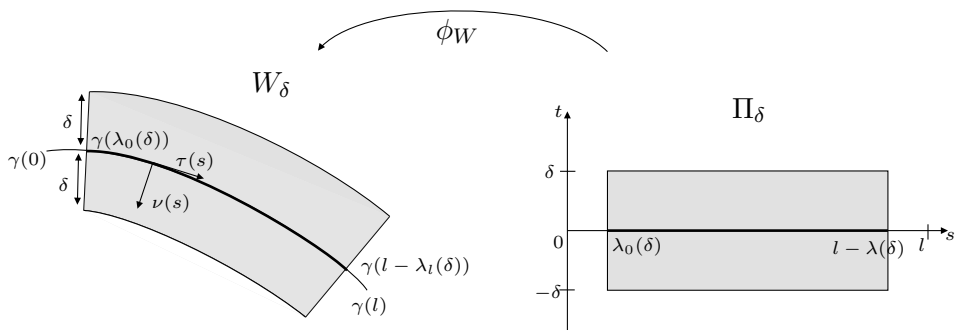
\begin{figure}
    \centering
    \begin{tikzpicture}[x=0.75pt,y=0.75pt,yscale=-1,xscale=1]
%uncomment if require: \path (0,300); %set diagram left start at 0, and has height of 300

%Curve Lines [id:da09207474685827122] 
\draw    (100,110) .. controls (159,101.83) and (250,163.83) .. (260,180) ;
%Straight Lines [id:da9217741350401035] 
\draw    (121.42,80.17) -- (118.58,139.83) ;
%Straight Lines [id:da7735343976211826] 
\draw    (271.33,142.83) -- (229.33,193.5) ;
%Curve Lines [id:da04988071559273788] 
\draw  [fill={rgb, 255:red, 155; green, 155; blue, 155 }  ,fill opacity=0.3 ] (121.42,80.17) .. controls (168.67,78.5) and (238.67,115.17) .. (271.33,142.83) ;
%Curve Lines [id:da9451936866487203] 
\draw    (118.58,139.83) .. controls (142.33,141.5) and (197,167.83) .. (229.33,193.5) ;
%Shape: Rectangle [id:dp788111409637603] 
\draw  [fill={rgb, 255:red, 155; green, 155; blue, 155 }  ,fill opacity=0.3 ] (380,120) -- (520,120) -- (520,180) -- (380,180) -- cycle ;
%Straight Lines [id:da5673069852481477] 
\draw    (360,200) -- (360,93) ;
\draw [shift={(360,90)}, rotate = 90] [fill={rgb, 255:red, 0; green, 0; blue, 0 }  ][line width=0.08]  [draw opacity=0] (3.57,-1.72) -- (0,0) -- (3.57,1.72) -- cycle    ;
%Straight Lines [id:da6147474316769779] 
\draw    (520,150) -- (547,150) ;
\draw [shift={(550,150)}, rotate = 180] [fill={rgb, 255:red, 0; green, 0; blue, 0 }  ][line width=0.08]  [draw opacity=0] (3.57,-1.72) -- (0,0) -- (3.57,1.72) -- cycle    ;
%Straight Lines [id:da5147652978598899] 
\draw    (360,150) -- (380,150) ;
%Straight Lines [id:da23035522661337948] 
\draw    (356.33,120.17) -- (364,120.17) ;
%Straight Lines [id:da20947384931732926] 
\draw    (356,179.83) -- (363.67,179.83) ;
%Straight Lines [id:da4957618288145059] 
\draw    (540,146.83) -- (540,154.5) ;
%Shape: Polygon [id:ds9158826308808281] 
\draw  [draw opacity=0] (121.42,80.17) -- (271.33,142.83) -- (229.33,193.5) -- (118.58,139.83) -- cycle ;
%Shape: Polygon [id:ds14988591407239038] 
\draw  [draw opacity=0][fill={rgb, 255:red, 155; green, 155; blue, 155 }  ,fill opacity=0.27 ] (121.42,80.17) -- (271.33,142.83) -- (229.33,193.5) -- (118.58,139.83) -- cycle ;
%Curve Lines [id:da9389038257041095] 
\draw [fill={rgb, 255:red, 255; green, 255; blue, 255 }  ,fill opacity=1 ]   (118.58,139.83) .. controls (143.4,141.8) and (197.8,168.6) .. (229.33,193.5) ;
%Straight Lines [id:da45012588596822845] 
\draw    (117.53,84.2) -- (117.07,104) ;
\draw [shift={(117,107)}, rotate = 271.33] [fill={rgb, 255:red, 0; green, 0; blue, 0 }  ][line width=0.08]  [draw opacity=0] (3.57,-1.72) -- (0,0) -- (3.57,1.72) -- cycle    ;
\draw [shift={(117.6,81.2)}, rotate = 91.33] [fill={rgb, 255:red, 0; green, 0; blue, 0 }  ][line width=0.08]  [draw opacity=0] (3.57,-1.72) -- (0,0) -- (3.57,1.72) -- cycle    ;
%Straight Lines [id:da6358887018454019] 
\draw    (115.93,114.2) -- (115.47,134) ;
\draw [shift={(115.4,137)}, rotate = 271.33] [fill={rgb, 255:red, 0; green, 0; blue, 0 }  ][line width=0.08]  [draw opacity=0] (3.57,-1.72) -- (0,0) -- (3.57,1.72) -- cycle    ;
\draw [shift={(116,111.2)}, rotate = 91.33] [fill={rgb, 255:red, 0; green, 0; blue, 0 }  ][line width=0.08]  [draw opacity=0] (3.57,-1.72) -- (0,0) -- (3.57,1.72) -- cycle    ;
%Curve Lines [id:da9766822288668527] delta support
\draw [line width=1.0]    (120,110) .. controls (163,109.5) and (237.67,157.5) .. (250,170) ;
%Straight Lines [id:da28349019527043784] delta support rectangle
\draw [line width=1.0]    (380,150) -- (520,150) ;
%Straight Lines [id:da3724506706371713] 
\draw    (153.11,115.72) -- (175.74,122.89) ;
\draw [shift={(178.6,123.8)}, rotate = 197.58] [fill={rgb, 255:red, 0; green, 0; blue, 0 }  ][line width=0.08]  [draw opacity=0] (3.57,-1.72) -- (0,0) -- (3.57,1.72) -- cycle    ;
%Straight Lines [id:da10487252968298488] 
\draw    (153.11,115.72) -- (146.34,136.15) ;
\draw [shift={(145.4,139)}, rotate = 288.32] [fill={rgb, 255:red, 0; green, 0; blue, 0 }  ][line width=0.08]  [draw opacity=0] (3.57,-1.72) -- (0,0) -- (3.57,1.72) -- cycle    ;
%Curve Lines [id:da6390540542372299] 
\draw    (380,60) .. controls (352.84,32.6) and (260.76,33.66) .. (237.46,57.25) ;
\draw [shift={(235.5,59.5)}, rotate = 307.14] [fill={rgb, 255:red, 0; green, 0; blue, 0 }  ][line width=0.08]  [draw opacity=0] (3.57,-1.72) -- (0,0) -- (3.57,1.72) -- cycle    ;

% Text Node
\draw (79.4,108.8) node [anchor=north west][inner sep=0.75pt]  [font=\tiny] [align=left] {$\displaystyle \gamma ( 0)$};
% Text Node
\draw (251,182) node [anchor=north west][inner sep=0.75pt]  [font=\tiny] [align=left] {$\displaystyle \gamma ( l)$};
% Text Node
\draw (121,98.73) node [anchor=north west][inner sep=0.75pt]  [font=\tiny] [align=left] {$\displaystyle \gamma ( \lambda_{0}( \delta ))$};
% Text Node
\draw (255.6,162.2) node [anchor=north west][inner sep=0.75pt]  [font=\tiny] [align=left] {$\displaystyle \gamma ( l-\lambda_{l}( \delta ))$};
% Text Node
\draw (106.67,90.47) node [anchor=north west][inner sep=0.75pt]  [font=\tiny] [align=left] {$\displaystyle \delta $};
% Text Node
\draw (351.67,151.33) node [anchor=north west][inner sep=0.75pt]  [font=\tiny] [align=left] {$\displaystyle 0$};
% Text Node
\draw (381,152) node [anchor=north west][inner sep=0.75pt]  [font=\tiny] [align=left] {$\displaystyle \lambda_{0}( \delta )$};
% Text Node
\draw (534,152.33) node [anchor=north west][inner sep=0.75pt]  [font=\tiny] [align=left] {$\displaystyle l$};
% Text Node
\draw (491,153.33) node [anchor=north west][inner sep=0.75pt]  [font=\tiny] [align=left] {$\displaystyle l-\lambda ( \delta )$};
% Text Node
\draw (347.33,116) node [anchor=north west][inner sep=0.75pt]  [font=\tiny] [align=left] {$\displaystyle \delta $};
% Text Node
\draw (342.33,176) node [anchor=north west][inner sep=0.75pt]  [font=\tiny] [align=left] {$\displaystyle -\delta $};
% Text Node
\draw (547.33,149.67) node [anchor=north west][inner sep=0.75pt]  [font=\tiny] [align=left] {$\displaystyle s$};
% Text Node
\draw (352.33,86) node [anchor=north west][inner sep=0.75pt]  [font=\tiny] [align=left] {$\displaystyle t$};
% Text Node
\draw (201,69) node [anchor=north west][inner sep=0.75pt]  [font=\normalsize] [align=left] {$\displaystyle W_{\delta }$};
% Text Node
\draw (441,82) node [anchor=north west][inner sep=0.75pt]   [align=left] {$\displaystyle \Pi_{\delta }$};
% Text Node
\draw (104.67,122.47) node [anchor=north west][inner sep=0.75pt]  [font=\tiny] [align=left] {$\displaystyle \delta $};
% Text Node
\draw (163.2,109) node [anchor=north west][inner sep=0.75pt]  [font=\tiny] [align=left] {$\displaystyle \tau ( s)$};
% Text Node
\draw (151.25,130.36) node [anchor=north west][inner sep=0.75pt]  [font=\tiny] [align=left] {$\displaystyle \nu ( s)$};
% Text Node
\draw (293,42) node [anchor=north west][inner sep=0.75pt]   [align=left] {$\displaystyle \phi_{W}$};
\end{tikzpicture}
    \caption{Diffeomorphism between $\Pi_\delta$ and $W_\delta$}
    \label{fig: W delta}
\end{figure}

Note that for $G=(G_{ij})$, where $G_{ij} = \langle \partial_i \phi_W , \partial_j \phi_W \rangle$, one obtains
\[
G = \begin{pmatrix}
    (1- k(s))^2 & 0 \\ 0 & 1
\end{pmatrix}.
\]
In particular, $G$ is invertible provided that $\delta_0 \cdot k_{max} <  1$. By the implicit function theorem , for any $\delta_0 < 1/k_{max}$, the mapping $\phi_W: \Pi_\delta \rightarrow W_\delta $ is a diffeomorphism for all $\delta \in (0, \delta_0)$.

Let us now define the Laplacians with a $\delta$-interaction supported on $\Gamma$, whose spectral properties will be the focus of our analysis.
\begin{definition} For $x \in \{d,n\}$ and $\alpha> 0$, we define the sesquilinear form 
    \[
    x^\delta_{W, \alpha}(u) = \int_{W_\delta} |\nabla u|^2 \mathrm{d}x - \alpha\int_{\Gamma_\delta} |u|^2 \dS, \quad D(d^\delta_{W,\alpha}) = H^1_0(W_\delta), \quad D(n^\delta_{W,\alpha}) = H^1(W_\delta).
    \]
\end{definition}
By constructing a diffeomorphism between $W_\delta$ and $\Pi_\delta$, we can perform a change of variables and derive suitable estimates for the unitary equivalent operators that arise in this context.
\begin{lemma}\label{lem: unitary equivalence strip} Define the unitary operator
    \[
    \Phi:L^2(W_\delta) \rightarrow L^2(\Pi_\delta), \quad u(s,t) \mapsto (1-t k(s))^{\frac{1}{2}} u(\phi_W(s,t))  =: g(s,t).
    \]
    For given constants $a_D, a_N, \beta \in \rr$, consider the sesquilinear forms $b^D_{\delta,\alpha}$ and $ b^N_{\delta,\alpha}$ with $D(b^D_{\delta,\alpha}) = H^1_0(\Pi_\delta)$ and $D(b^N_{\delta,\alpha}) = H^1(\Pi_\delta)$, defined by
         \begin{align*}
             b^D_{\delta,\alpha}(g) :=& \int_{I_\delta} \int_{-\delta}^\delta \left( (1+ a_D \delta) | \partial_s g  |^2 + | \partial_t g  |^2 - \left(\frac{k^2}{4} - a_D \delta\right)|  g  |^2 \right)\mathrm{d}t \mathrm{d}s - \alpha \int_{I_\delta} |  g(s,0)|^2 \mathrm{d}s,\\
               b^N_{\delta,\alpha} (g) :=& \int_{I_\delta}\int_{-\delta}^\delta \left( (1-a_N \delta) |  \partial_sg |^2 + | \partial_t g |^2 -\left(\frac{k^2}{4} +a_N\delta \right)|  g  |^2\right) \mathrm{d}t \mathrm{d}s  \\
               &- \alpha \int_{I_\delta} |  g(s,0)|^2 \mathrm{d}s - \beta \int_{I_\delta} |  g(s,\delta)|^2 + | g(s,-\delta) |^2 \mathrm{d}s. 
         \end{align*}
Then, for sufficiently small  $\delta > 0$, there exist $a_D, a_N, \beta > 0$ such that  
    \begin{align*} 
    d^\delta_{W,\alpha}(u) \leq b^D_{\delta,\alpha} (g),\quad \forall u \in D(d^\delta_{W,\alpha}),\quad \text{and}\quad n^\delta_{W,\alpha} (u) \geq d^N_{\delta,\alpha} (g),  \quad \forall u\in D(n^\delta_{W,\alpha}).
    \end{align*}
\end{lemma}
\begin{proof} Since $\phi_W$ and $(1-tk(s))^{\frac{1}{2}}$ are smooth, it is clear that  $\Phi(H^1(W_\delta)) = H^1(\Pi_\delta)$.  We are going to construct unitary equivalent operators for $D^\delta_{W,\alpha}$ and $N^\delta_{W,\alpha}$ through $\Phi$. Set $v := u \circ \phi_W$, then performing the change of variables yields 
    \[
    \int_{W_\delta} |\nabla u|^2 \mathrm{d}x - \alpha \int_{\Gamma_\delta} |u|^2 \dS  = \int_{I_\delta}\int_{-\delta}^\delta \frac{1}{1-tk} |\partial_s v|^2 + (1-tk)|\partial_t v|^2 \mathrm{d}x - \alpha\int_{I_\delta}|v(s,0)|^2 \dS.
    \]
    Substituting $v = (1-tk)^{-\frac{1}{2}}g$, we compute
    \begin{align*}
        \int_{I_\delta} \int_{- \delta}^\delta \frac{1}{(1-tk)^2} &\left|\partial_s g + \frac{tk'}{2(1-tk)}g\right|^2 + \left|\partial_t g + \frac{k}{2(1-tk)}g\right|^2 \mathrm{d}t \text{ d}s - \alpha\int_{I_\delta}| g(s,0) |^2 \dS \\
         =& \int_{I_\delta} \int_{-\delta}^\delta \frac{1}{(1-tk)^2} |  \partial_s g |^2 + \frac{t k'}{(1-tk)^3} \Re (g \partial_s g) + \frac{(tk')^2}{4(1-tk)^4}|  g  |^2 \\
          &+ |  \partial_t g |^2 + \frac{k}{1-tk} \Re(g \partial_tg) + \frac{k^2}{4 (1-tk)^2} |  g  |^2 \mathrm{d}t \mathrm{d}s - \alpha \int_{I_\delta} | g(s,0) |^2 \dS .
    \end{align*}
    By integration by parts, we have
    \begin{align*}
    \int_{-\delta}^\delta \frac{k}{1-tk} \Re( g \partial_t g) \text{ d}t  &= \frac{1}{2} \int_{-\delta}^\delta \frac{k}{1-tk} \partial_t |  g  |^2  \mathrm{d}t\\
    &= \frac{k}{2(1-\delta k)} |  g(s,\delta)|^2 - \frac{k}{2(1+\delta k)} | g(s,-\delta) |^2 + \int_{-\delta}^\delta \frac{k^2}{2(1-tk)^2} |  g  |^2 \mathrm{d}t.
    \end{align*}
Therefore, the sesquilinar form associated to $ N^\delta_{W,\alpha} $ is unitarily equivalent to 
    \begin{align*}
        \tilde{n}^\delta_{W,\alpha}(g)  =& \int_{I_\delta} \int_{-\delta}^\delta \Bigg[ \frac{1}{(1-tk)^2} |  \partial_s g |^2 + \frac{tk'}{(1-tk)^3} \Re( g \partial_s g) + |  \partial_t g |^2 \\
        &+ \left(\frac{(tk')^2}{4(1-tk)^4} - \frac{k^2}{4(1-tk)^2}\right) |  g  |^2 \Bigg]\mathrm{d} t \mathrm{d}s - \alpha \int_{I_\delta} | g(s,0) |^2 \mathrm{d}s  \\
        &+ \int_{I_\delta}\left( \frac{k}{1-\delta k} |  g(s,\delta)|^2 + \frac{k}{1+\delta k} | g(s,-\delta) |^2\right) \mathrm{d}s, 
    \end{align*}
    with domain $D(\tilde{n}^\delta_{W,\alpha}) = H^1(\Pi_\delta)$. For sufficiently small $\delta<\delta_0$ and $t\in(0,\delta)$, we estimate $\tilde{n}^\delta_{W,\alpha}$ from below by applying the estimate
    \begin{align*}
        | g \partial_s g| \leq \frac{1}{2}( |  g  | ^2 + |  \partial_s g |^2), \quad | tk'(s)| \leq \delta \max_{s \in [0,l]} | k'(s) |,
    \end{align*}
     and the expansion estimates
     \begin{align*}
         \left| \frac{1}{(1-tk)^j} - 1\right| = \left| \frac{ \sum_{i = 1}^j \binom{i}{j}  (tk)^i   }{(1-tk)^j}\right| \leq \frac{\delta {\sum_{i = 1}^j \binom{i}{j}  \delta_0^{i-1}k_{max}^i }}{(1- \delta_0 k_{max})^j} \leq c \delta, \quad j \in \{1,2,3,4 \},
     \end{align*}
   for some $c>0$. Hence, for $\beta =k_{max}/(1 - \delta_0 k_{max}) $ and some $a_N> 0$, we obtain
    \begin{align*}
        \tilde{n}^\delta_{W,\alpha} (g) \geq& \int_{I_\delta}\int_{-\delta}^\delta \left[ (1-a_N \delta) |  \partial_sg |^2 + |  \partial_t g |^2 -\left(\frac{k^2}{4} +\delta a_N \right)|  g  |^2 \right]\mathrm{d}t \mathrm{d}s \\
       & - \alpha \int_{I_\delta} | g(s,0) |^2 \mathrm{d}s - \beta \int_{I_\delta} |  g(s,\delta)|^2 + | g(s,-\delta) |^2 \mathrm{d}s =:b^N_{\delta,\alpha} (g).
    \end{align*}
    Using similar arguments for  $d^\delta_{W,\alpha}$ and its unitarily equivalent form $  \tilde{d}^\delta_{W,\alpha}$ which coincides with $\tilde{n}^\delta_{W,\alpha}(g) $ but with domain $D( \tilde{d}^\delta_{W,\alpha}) = H^1_0(\Pi_\delta)$, we find for some $a_D>0$ that 
    \[
    \tilde{d}^\delta_{W,\alpha} (g) \leq   \int_{I_\delta} \int_{-\delta}^\delta \left[(1+ a_D \delta) | \partial_s g  |^2 + | \partial_t g  |^2 - \left(\frac{k^2}{4} - a_D \delta\right) |  g  |^2 \right]\mathrm{d}t \mathrm{d}s - \alpha \int_{I_\delta} | g(s,0) |^2 \mathrm{d}s=b^D_{\delta,\alpha}(g).
    \]
\end{proof}

\subsection{Spectral properties}In this subsection, we apply Lemma \ref{lem: unitary equivalence strip} to obtain estimates for the eigenvalues of $N^\delta_{W,\alpha}$ and $D^\delta_{W,\alpha}$. We denote by
    \begin{itemize}
        \item $D_\delta$ := Dirichlet Laplacian on $I_\delta$,
        \item  $D_l$ := Dirichlet Laplacian on $(0,l)$.
    \end{itemize}
We begin by proving a useful lemma comparing these two operators:
\begin{lemma}\label{lem: make it Laplace - k^2/4 }
Let $b \geq 0$. For any fixed $n \in \mathbb{N}$, it holds that
    \[
    E_n((1 + b \delta)D_\delta - k^2/4) = E_n(D_l - k^2/4) + \mathcal{O}(\delta)\quad \text{as } \delta\to 0.
    \]
\end{lemma}
\begin{proof}
    Fix $n \in \mathbb{N}$  and  let $J: L^2(I_\delta) \to L^2(0,l)$ be the extension-by-zero operator. It is straightforward to verify that
    \[
    \| Ju \|_{L^2(0,l)} = \| u \|_{L^2(I_\delta)} \quad \text{and} \quad ((1+ b \delta)D_\delta - \frac{k^2}{4} ) (u) = ((1+b \delta)D_l - \frac{k^2}{4}) (Ju).
    \]
    Hence, by the min-max principle, we get
    \[
    E_n\left((1 + b \delta)D_\delta - \frac{k^2}{4}\right)  \geq E_n\left((1 + b \delta) D_l - \frac{k^2}{4}\right)  \geq E_n\left(D_l - \frac{k^2}{4}\right).
    \]
    On the other hand, consider the bijective linear map $\phi:  [\lambda_0, l - \lambda_l]  \rightarrow [0, l]$ and its inverse given by
    \[\phi (y) = \frac{l}{l - \lambda_l(\delta) - \lambda_0(\delta) }(y- \lambda_0 (\delta)),
    \quad \phi^{-1}( x) =   \left(1- \frac{\lambda_l (\delta) + \lambda_0(\delta)}{l}\right)x  + \lambda_0 (\delta),
    \]
    and define the corresponding operator
    \[
    \Phi: L^2 (0,l) \rightarrow L^2 (I_ \delta), \quad f(x) \mapsto f (\phi(y)) \cdot (\phi'(y))^{\frac{1}{2}} = f(\phi(y))(1+ \frac{\lambda_l(\delta) + \lambda_0(\delta)}{l-\lambda_l(\delta) - \lambda_0(\delta)})^{\frac{1}{2}}.
    \]
    Since $\phi$ is smooth and invertible, $\Phi$ is also invertible. Moreover, a direct computation shows that
    \[
    \| \Phi (f)  \|^2_{L^2} = \int_{0}^{l} |f(\Phi(x))|^2 \cdot \phi'(x) \mathrm{d}x = \int_{\lambda_0(\delta)}^{l - \lambda_l(\delta)} | f|^2 \mathrm{d}y = \| f \|^2_{L^2(I_\delta)},
    \]
    which means that $\Phi$ is unitary. Combining this with a change of variables, we get
    \begin{align*}
    ((1 + b \delta)D_\delta- \frac{k^2}{4} ) (\Phi f) &= \int_{I_\delta} (1+ b \delta)|f'(\phi(x))|^2 \cdot (\phi'(x))^3 - \frac{k^2(x)}{4} |f(\phi(x))|^2 \phi'(x) \mathrm{d}x \\
    &= \int_0^{l} (1 + b \delta) | f'|^2 \cdot (\phi')^2 - \frac{k^2(y(1- \frac{\lambda_l(\delta) - \lambda_0 (\delta) }{l}) + \lambda_0 )  }{4} |f'| \mathrm{d}y =: q(f).
    \end{align*}
    Thus, it suffices to find a suitable upper bound for $q(f)$. For this,  note that $\phi' = 1+ \mathcal{O}(\delta)$, so for some $c_0 >0$,
    \begin{equation} \label{eq: agmon curved 3}\int_0^{l} (1 + b \delta) | f'|^2 \cdot (\phi')^2 \mathrm{d}y \leq (1+ c_0\delta) \int_0^{l} | f'|^2 \mathrm{d}y.
    \end{equation}
 As $k^2$ is Lipschitz continuous, denoting by $L$ its Lipschitz constant, we have
    \begin{align*} 
    \left| k^2 \left(y + \lambda_0(\delta) - y \frac{\lambda_l(\delta) + \lambda_0(\delta) }{l} \right)  - k^2(y)\right| & \leq L \left| \lambda_0(\delta) - y \frac{\lambda_l(\delta) + \lambda_0(\delta) }{l} \right|\\
   & \leq L (2\lambda_0(\delta) + \lambda_l(\delta) ).
    \end{align*}
    Hence, 
    \[
    \int_0^{l} \frac{k^2(y(1- \frac{\lambda_l(\delta) - \lambda_0 (\delta) }{l}) + \lambda_0 )  }{4} | f(y)| \mathrm{d}y 
    \geq (1+c_0 \delta) \int_0^{l} k^2 | f|^2  \mathrm{d}y  -  (c_1\delta +  c_0 \| k^2 \|_{\infty}\delta ) \| f \|^2_{L^2}
    \]
    for some $c_0,c_1>0$. Combining this estimate with \eqref{eq: agmon curved 3} gives 
    \[
    q(f) \leq (1+c_0 \delta) (D_l - \frac{k^2}{4} ) (f) + c_2 \delta \| f \|^2_{L^2(0,l)}
    \]
    for a suitable $c_2 >0$ and  sufficiently small $\delta$. Applying the min-max principle, we conclude that
    \[
    E_n((1 + b \delta)D_\delta - \frac{k^2}{4}) =E_n(Q) \leq  E_n(D_l - \frac{k^2}{4} ) + \mathcal{O}(\delta)
    \]
    as $\delta \to 0$, which completes the proof.
\end{proof}

With the help of Lemma \ref{lem: make it Laplace - k^2/4 }, we can derive an upper bound for $E_n(D^\delta_{W,\alpha})$.
\begin{lemma}\label{lem: ev strip dirichlet}
For any fixed $n \in \mathbb{N}$, there exists $c_D > 0$ such that, as $\delta \to 0^+$ and $\alpha \delta \to \infty$, one has
  \[
  E_n(D^\delta_{W,\alpha}) \leq -\frac{\alpha^2}{4} + E_n( D_l - \frac{k^2}{4}) + c_D( \delta + \alpha^2 e^{-\frac{1}{2} \delta \alpha} ).
  \]
\end{lemma}
\begin{proof} 
Let $n \in \mathbb{N}$. By Lemma \ref{lem: unitary equivalence strip}, we have $E_n(D^\delta_{W,\alpha}) \leq E_n(B^D_{\delta,\alpha})$ for some $a_D >0$, and moreover,
\[
B^D_{\delta,\alpha} \cong ((1+ a_D \delta)D_\delta- \frac{k^2}{4} )\otimes I + I \otimes T^D_{\delta,\alpha} +a_D \delta,
\]
with $T^D_{\delta,\alpha}$ as in Definition \ref{def t X}. Applying Lemma \ref{lem: make it Laplace - k^2/4 } yields
\[
E_n((1+ a_D \delta)D_\delta- \frac{k^2}{4}) = E_n (D_l - \frac{k^2}{4}) + \mathcal{O}(\delta) = \mathcal{O}(1) \quad \text{as }  \delta \to 0.
\]
Furthermore, assertions (i) and (iv) from Proposition \ref{lem: ev of t,N,L,alpha} ensure that there is a $c>0$ such that, for sufficiently large $\alpha \delta$,
\[
E_1(T^D_{\delta,\alpha}) < -\frac{\alpha^2}{4} + c \alpha^2 e^{- \frac{1}{2}\delta \alpha } , \quad\text{and}\quad E_2(T^D_{\delta,\alpha}) \geq 0.
\]
Combining these estimates, it follows that for sufficiently large $\alpha$ there exists $c_D >0 $ such that
\[
E_n(D^\delta_{W,\alpha}) \leq -\frac{\alpha^2}{4} + E_n( D_l - \frac{k^2}{4}) + c_D (\delta +  \alpha^2 e^{-\frac{1}{2} \delta \alpha}).
\]
\end{proof}
We conclude this part by deriving a lower bound for the Neumann operator $N^\delta_{W,\alpha}$.

\begin{lemma}
\label{lem: ev strip neuman} Let $T^{\beta}_{\delta, \alpha}$ be as in Definition \ref{def t X} and $\psi \in H^1(-\delta,\delta)$ be a normalized eigenfunction associated to its first eigenvalue. Define the projection 
 \[
 P:L^2(W_\delta) \rightarrow L^2(I_\delta), \quad (Pu)(s) = \int_{-\delta}^\delta \overline{\psi}(t) (\phi_W u ) (s,t) \mathrm{d}t,
 \]
 then there exist $a_N,\beta >0$ for which the inequality
\[
n^\delta_{W,\alpha} (u) \geq (1- a_N \delta) \| Pu'\|^2_{L^2(I_\delta)} +  
\int_{I_\delta} ( -\frac{\alpha^2}{4} - \frac{k^2}{4} ) |Pu|^2 \mathrm{d}s - a_N(  \alpha^2e^{-\frac{1}{2}\delta \alpha} + \delta ) \| Pu \|^2_{L^2(I_\delta)}
\]
holds and in particular, 
\[
N^\delta_{W,\alpha} \geq - \frac{\alpha^2}{4} - \frac{\|k\|_\infty^2}{4} + \mathcal{O}(\delta  + \alpha^2 e^{-\frac{1}{2} \alpha \delta}) 
\]
as $\alpha \to \infty$, $\delta \to0^+$ and $\alpha \delta \to \infty$.
\end{lemma}
\begin{proof}
By Lemma \ref{lem: unitary equivalence strip}, there exist $a_N,\beta >0$ such that
$ n^\delta_{W,\alpha} (u) \geq  b^N_{\delta, \alpha}(g)$. Set 
\[
f := Pu, \quad  z(s,t) := g(s,t) - f(s) \psi(t) \in L^2(\Pi_\delta).
\]
Observe that $z(\cdot,t)$ is orthogonal to  $\psi $ in $L^2(-\delta,\delta)$:
\[
\int_{-\delta}^\delta \overline{\psi}(t) z(\cdot,t) \mathrm{d}t = \int_{-\delta}^\delta \overline{\psi}(t) g(s,t) \mathrm{d}t  - \underbrace{ \int_{-\delta}^\delta | \psi(t) |^2 \mathrm{d} t}_{= 1} \cdot \int_{-\delta}^\delta \overline{\psi}(\tau) g(s,\tau) \mathrm{d} \tau \mathrm{d}t = 0.
\]
 Using the decomposition $g(s,t)= z(s,t) + f(s)\psi(t)$ and applying the min-max principle, we find
\begin{align*}
     \int_{I_\delta} \int_{-\delta}^\delta |  \partial_t g |^2 \mathrm{d}t - \alpha | g(s,0) |^2 - &\beta(|  g(s,\delta)|^2 + | g(s,-\delta) |^2) \mathrm{d}s \\
      &\geq \int_{I_\delta} \int_{-\delta}^\delta E_1(T^\beta_{\delta,\alpha}) |f(s) \psi(t)|^2 + E_2(T^\beta_{\delta,\alpha})|z(s,t)|^2 \mathrm{d}t \text{ d}s \\
&= E_1(T^\beta_{\delta,\alpha}) \| f\|^2_{L^2(I_\delta)} + E_2(T^\beta_{\delta,\alpha}) \|z\|^2_{L^2(\pi_\delta)}.
\end{align*}
Since $\psi$ and $\partial_s z(\cdot,t)$ are orthogonal, we have
\[
\int_{I_\delta} \int_{-\delta}^\delta |  \partial_s g |^2 \mathrm{d}t \text{ d}s = \| f' \|^2_{L^2(I_\delta)} + \|\partial_s z\|^2_{L^2(\Pi_\delta)} \geq \| f' \|^2_{L^2(I_\delta)}.
\]
Moreover, using the orthogonality of $z(\cdot,t)$ and $\psi(t)$ along with the independence of $k$ on $t$, it follows that
\[
\int_{I_\delta} \int_{-\delta}^\delta |  g  |^2 (-\frac{k^2}{4} - \delta a_N) \mathrm{d}t \text{ d}s = \int_{I_\delta} | f|^2 (-\frac{k^2}{4} - \delta a_N) \mathrm{d}s + \int_{I_\delta}\int_{-\delta}^\delta |z|^2(- \frac{k^2}{4} -\delta a_N).
\]
Combining these estimates, we obtain
\begin{align*}
b^N_{\delta, \alpha} (g) &\geq (1- a_N \delta) \| f' \|^2_{L^2(I_\delta)} + \int_{I_\delta} (E_1(T^\beta_{\delta,\alpha})- \frac{k^2}{4} - \delta a_N) | f|^2 \mathrm{d}s + \int_{\Pi_\delta} (E_2(T^\beta_{\delta,\alpha}) - \frac{k^2}{4} - \delta a_N) |z|^2 \mathrm{d}t \text{ d}s 
\\
&\geq (1- a_N \delta) \| f' \|^2_{L^2(I_\delta)} +  
\int_{I_\delta} ( -\frac{\alpha^2}{4} - \frac{k^2}{4}) | f|^2 \mathrm{d}s - ( c \alpha^2e^{-\frac{1}{2}\delta \alpha} + \delta a_N) \| f \|^2_{L^2(I_\delta)},
\end{align*}
where we used the fact, thanks to Proposition \ref{lem: ev of t,N,L,alpha}, there exists $c>0$ such that
\[
(E_2(T^\beta_{\delta,\alpha}) - \frac{k^2}{4} - \delta a_N) \geq \frac{c}{\delta^2}- \frac{k^2}{4} - \delta a_N \geq 0 \quad \text{for sufficiently small}\,\delta. 
\]
The second asserted inequality for $N^\delta_{W,\alpha}$ follows directly from the first assertion. 
\end{proof}

\section{Schr\"{o}dinger operator with a strong $\delta$-interaction supported on a curve with corners}\label{sec: Endzeit}
In this section, we apply the previous constructions and results to prove Theorems \eqref{theo: corner induced} and \eqref{theo: side induced}.
%As in Remark \ref{Transmission condition}, for a closed Lipschitz curve $\Gamma\subset\rr^2$, we denote by $\Omega_+$
\subsection{Decomposition of $\rr^2$ into neighborhoods of corners and edges}
We begin by defining the notion of a curve $\Gamma\subset\rr^2$ with corners as used in this context. In the following, we let $\Omega_+$ be the bounded part of $\rr^2$ enclosed by $\Gamma$ and set $\Omega_-=\rr^2\setminus\overline{\Omega_+}$. We further denote by $\nu$ the outward unit normal to $\Omega_+$.
\begin{definition} Let $\Gamma \subset \rr^2 $ be an injective,  continuous, closed curve. We say that $\Gamma$ is a \textit{curve with $M\geq1$ corners} if the following hold: 
    \begin{itemize}
        \item[(1)] There exist \textit{vertices} $A_1, \dots,A_M \in \rr^2$ and positive lengths $l_1,\dots , l_M > 0 $
        \item[(2)] There exist $ C^3$-smooth arc-length parameterizations $\gamma_j: [0,l_j] \rightarrow \rr^2$ with $|\gamma'_j|=1$, $j = 1, \dots,M$, such that 
        \begin{itemize}
            \item[(a)] The interiors $\gamma_j((0,l_j))$, $j=1,\dots, M$, are pairwise disjoint.
            \item[(b)] $ \gamma_j(0) = A_j$ and  $\gamma_j(l_j) = A_{j+1}$ for $j = 1, \ldots , M$, with the convention  $A_1 \equiv A_{M+1}$.
            \item[(c)] $\Gamma = \bigcup_{j=1}^M \Gamma_j$ where $\Gamma_j:= \gamma_j([0,l_j])$, $j = 1, \dots,M$.
        \end{itemize} 
    \end{itemize}
Furthermore, we assume that each $\gamma_j$ is orientated in such that, for the outward normal $\nu_j(s)$ to the bounded enclosed region by $\Gamma$ at a point $s\in(O,l_j)$, we have $\nu_j(s)\wedge \gamma'(s)=1$. Moreover, $k_j(s)$ denote the curvature of $\gamma_j(s)$ at $s\in(0,l_j)$, and let $\theta_j \in [0,\pi]$ be the half-angle between the tangent vectors of $\gamma_{j-1}$ and $\gamma_j$ at vertex $A_j$, defined by the relations
\[
\cos (2 \theta_j) = - \langle \nabla \gamma_j(0), \nabla \gamma_{j-1} (l_{j-1})  \rangle, \quad  \sin (2 \theta_j) = -\det( \nabla \gamma_j (0) \; \,\,\, \nabla \gamma_{j-1}(l_{j-1}) ).
\]
We further assume that $ \theta_j \notin \{  0, \frac{\pi}{2}, \pi \} $ for all $j = 1, \ldots, M$. A visualization is given in Figure~\ref{fig: curve with corners}.
\end{definition}
\begin{figure}
    \centering
    \scalebox{0.8}{\begin{tikzpicture}[x=0.9pt,y=0.9pt,yscale=-1,xscale=1]
%uncomment if require: \path (0,300); %set diagram left start at 0, and has height of 300

%Curve Lines [id:da6697252582185478] 
\draw [line width=1.0]    (90,160) .. controls (137.4,154.2) and (113.8,67.8) .. (190,80) ;
%Curve Lines [id:da9091187292621014] 
\draw [line width=1.0]    (90,160) .. controls (141.4,210.6) and (288.2,261) .. (360,230) ;
%Curve Lines [id:da40181230220956254] 
\draw [line width=1.0]     (190,80) .. controls (227.4,55.4) and (297,93.4) .. (290,130) ;
%Curve Lines [id:da6619535175576159] 
\draw [line width=1.0]     (290,130) .. controls (365.8,102.2) and (367.4,205.8) .. (360,230) ;
%Curve Lines [id:da12450672000519813] 
\draw    (106.67,153.17) .. controls (113.67,159.5) and (108.67,169.17) .. (104.67,172.17) ;
%Curve Lines [id:da07646621086357652] 
\draw    (344,235.17) .. controls (343.33,221.5) and (351.33,217.17) .. (362,219.17) ;
%Shape: Arc [id:dp2085609242308193] 
\draw  [draw opacity=0] (201.07,75.28) .. controls (201.4,76.16) and (201.65,77.08) .. (201.81,78.05) .. controls (203.01,85.29) and (198.68,92.03) .. (192.16,93.11) .. controls (185.64,94.18) and (179.38,89.19) .. (178.19,81.95) .. controls (178.03,80.97) and (177.97,80.01) .. (178,79.06) -- (190,80) -- cycle ; \draw   (201.07,75.28) .. controls (201.4,76.16) and (201.65,77.08) .. (201.81,78.05) .. controls (203.01,85.29) and (198.68,92.03) .. (192.16,93.11) .. controls (185.64,94.18) and (179.38,89.19) .. (178.19,81.95) .. controls (178.03,80.97) and (177.97,80.01) .. (178,79.06) ;  
%Shape: Arc [id:dp09557817661486967] 
\draw  [draw opacity=0] (301.32,127.26) .. controls (301.76,129.64) and (301.5,132.2) .. (300.42,134.62) .. controls (297.71,140.73) and (290.85,143.62) .. (285.09,141.06) .. controls (279.34,138.51) and (276.87,131.49) .. (279.58,125.38) .. controls (281.49,121.07) and (285.46,118.37) .. (289.65,118.04) -- (290,130) -- cycle ; \draw   (301.32,127.26) .. controls (301.76,129.64) and (301.5,132.2) .. (300.42,134.62) .. controls (297.71,140.73) and (290.85,143.62) .. (285.09,141.06) .. controls (279.34,138.51) and (276.87,131.49) .. (279.58,125.38) .. controls (281.49,121.07) and (285.46,118.37) .. (289.65,118.04) ;  

% Text Node
\draw (181.93,62.67) node [anchor=north west][inner sep=0.75pt]  [font=\scriptsize,color={black}  ,opacity=1 ] [align=left] {$\displaystyle A_{1}$};
% Text Node
\draw (66.67,157) node [anchor=north west][inner sep=0.75pt]  [font=\scriptsize,color={black}  ,opacity=1 ] [align=left] {$\displaystyle A_{2}$};
% Text Node
\draw (362.07,225.13) node [anchor=north west][inner sep=0.75pt]  [font=\scriptsize,color={black}  ,opacity=1 ] [align=left] {$\displaystyle A_{3}$};
% Text Node
\draw (293.2,110.27) node [anchor=north west][inner sep=0.75pt]  [font=\scriptsize,color={black}  ,opacity=1 ] [align=left] {$\displaystyle A_{4}$};
% Text Node
\draw (107.07,98.27) node [anchor=north west][inner sep=0.75pt]  [font=\scriptsize,color={black}  ,opacity=1 ] [align=left] {$\displaystyle \Gamma_{1}$};
% Text Node
\draw (183.33,81) node [anchor=north west][inner sep=0.75pt]  [font=\tiny] [align=left] {$\displaystyle 2\theta_{1}$};
% Text Node
\draw (94.67,158.33) node [anchor=north west][inner sep=0.75pt]  [font=\tiny] [align=left] {$\displaystyle 2\theta_{2}$};
% Text Node
\draw (346,222.33) node [anchor=north west][inner sep=0.75pt]  [font=\tiny] [align=left] {$\displaystyle 2\theta_{3}$};
% Text Node
\draw (281.67,130.67) node [anchor=north west][inner sep=0.75pt]  [font=\tiny] [align=left] {$\displaystyle 2\theta_{4}$};
% Text Node
\draw (367.07,157.27) node [anchor=north west][inner sep=0.75pt]  [font=\scriptsize,color={black}  ,opacity=1 ] [align=left] {$\displaystyle \Gamma_{3}$};
% Text Node
\draw (170.4,218.6) node [anchor=north west][inner sep=0.75pt]  [font=\scriptsize,color={black}  ,opacity=1 ] [align=left] {$\displaystyle \Gamma_{2}$};
% Text Node
\draw (265.73,64.93) node [anchor=north west][inner sep=0.75pt]  [font=\scriptsize,color={black}  ,opacity=1 ] [align=left] {$\displaystyle \Gamma_{4}$};
\end{tikzpicture}}
        \caption{A curve with corners}
        \label{fig: curve with corners}
\end{figure}
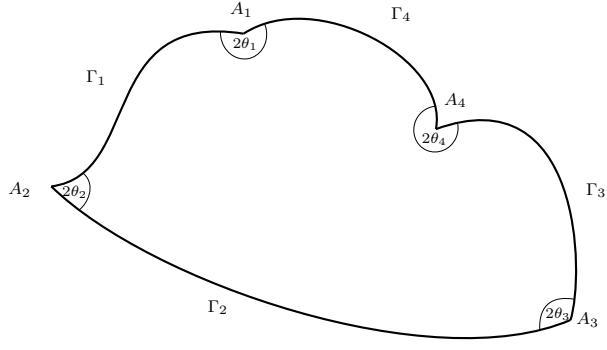

We now describe how to decompose $\rr^2$ into suitable neighborhoods around the corners and edges of a curve with corners. We first construct a decomposition of a neighborhood of a curve $\Gamma$ having $M$ corners as follows:
    \begin{itemize}
        \item[(i)] \textbf{Corners:} Consider a corner with interior angle $2\theta_j$. The regularity assumptions on the curve segments $ \gamma_{j-1} $ and $\gamma_j$ near the corner $A_j$ ensure the existence of  $C^3$-smooth extensions $\widetilde{\gamma_{j-1}},\, \widetilde{\gamma_j}$ beyond $A_j$.  Depending on the angle $\theta_j$, define the reparametrized curves $\gamma_{j}^\ast,\, \gamma_{j-1}^\ast$ by:
        \begin{enumerate}
            \item If  $\theta_j \in (0,\frac{\pi}{2})$, set
            $\gamma^*_j(s) := \widetilde{\gamma_j}(s) $ and $ \gamma_{j-1}^*(s) := \widetilde{\gamma_{j-1}}( l_{j-1}-s )$.
            \item If $\theta_j \in (\frac{\pi}{2}, \pi)$, set $\gamma^*_j(s) := \widetilde{\gamma_{j-1}}(l_{j-1}-s) $ and $ \gamma_{j-1}^*(s) := \widetilde{\gamma_{j}}( s ) $.
        \end{enumerate}
        After appropriate shifting and rotation of $\gamma^*_j$ and $\gamma^*_{j-1}$, we are exactly in the setting of of Section \ref{sec: V delta}. Thus, there exist neighborhoods $V_{j,\delta}$ of the corners, and functions
        \[
        \lambda_j^{0/l} (\delta) = \delta  + \mathcal{O}(\delta^2) \quad \text{as }\, \delta \to 0,
        \]
        such that for each $j \in 1, \dots M$,
        \[
        \partial V_{j,\delta}\, \cap \gamma_{j-1} = \{ \gamma_{j-1}(l_{j-1} - \lambda_{j-1}^l(\delta)) \}  ,\quad \partial V_{j,\delta} \, \cap \gamma_{j} = \{ \gamma_j(\lambda_j^0(\delta)) \}.
        \]
         Moreover, denote $ \partial_* V_{j,\delta} \subsetneq \partial V_{j,\delta} $ the straight boundary segments opposite to the angle $2 \theta_j$. These segments are orthogonal to the tangential vectors at $\gamma_{j-1}(l_{j-1}- \lambda^l_{j-1}(\delta))$ and $\gamma_j(\lambda_j^0(\delta))$, as described in Lemma \ref{lem: straigthen curved kite}. 
         
         Note that the eigenvalue asymptotics of the Laplacians on $V_{j,\delta}$ is independent of the choice of  extension $\widetilde{\gamma_{j-1}},\, \widetilde{\gamma_j}$.
        \item[(ii)] \textbf{Edges:} The neighborhoods around the edges $\gamma_j$ are tubular neighborhoods $\delta$ as constructed in Section \ref{sec: W delta}. Define
        \begin{align*} I_{j,\delta} :=  (\lambda_j^0(\delta) , l_j - \lambda_j^l(\delta)), \quad \Pi_{j,\delta} := I_{j,\delta} \times (-\delta, \delta),\\
        W_{j,\delta} := \phi_{W,j}(\Pi_{j,\delta}), \quad \phi_{W,j} (s,t) = \gamma_j(s) - t \nu_j(s).
        \end{align*}
    \end{itemize}
    We further define 
        \[\Omega^*_\delta = \rr^2 \setminus( \bigcup_{j=1}^M W_{j,\delta} \cup \bigcup_{j=1}^M V_{j,\delta} ),
        \]
         and set 
         \[
         \partial_* W_{j,\delta} := \phi_{W,j}(\{ \lambda_j^0(\delta), l_j -  \lambda_j^l(\delta) \}\times(-\delta,\delta)).
         \]
        By construction, these satisfy
        \[
        \bigcup_{j=1}^M \partial_* W_{j,\delta} = \bigcup_{j=1}^M \partial_* V_{j,\delta},
        \]
        and the sets $W_{j,\delta}$, $V_{j,\delta}$, and $\Omega_\delta^*$ form a pairwise disjoint decomposition of $\rr^2$. A visualization of this decomposition is provided in Figure \ref{fig: decomposition near Gamma}. 
        
        Using the notations from Section \ref{sec: angle}, define
        \begin{align*} 
        \mathcal{K} &:= \kappa(\theta_1) + \dots \kappa(\theta_M),\\
        \mathcal{E} &:= \text{ the disjoint union of }  \{ \mathcal{E}_n(\theta_j) \mid  n = 1, \dots, \kappa(\theta_j)  \}, \,  j = 1,\dots, M,\\
        \mathcal{E}_n &:= \text{the $n$th element of } \mathcal{E} \text{ in non decreasing order}.
        \end{align*}

\begin{figure}
    \centering
        \scalebox{0.7}{\begin{tikzpicture}[x=1pt,y=1pt,yscale=-1,xscale=1]
%uncomment if require: \path (0,300); %set diagram left start at 0, and has height of 300

%Straight Lines [id:da40315798314078577] 
\draw    (115.4,162.2) -- (98.6,184.2) ;
%Straight Lines [id:da9160268677012751] 
\draw    (97.4,141.4) -- (115.4,162.2) ;
%Straight Lines [id:da2899267911288105] 
\draw    (187.4,67.4) -- (184.2,92.2) ;
%Straight Lines [id:da6806343854287997] 
\draw    (187.4,67.4) -- (201.4,89) ;
%Straight Lines [id:da06903195202548407] 
\draw    (301,113) -- (277,121.8) ;
%Straight Lines [id:da45473120825151514] 
\draw    (301,113) -- (304.6,138.2) ;
%Straight Lines [id:da37761384054676417] 
\draw    (374.6,224.2) -- (347,221) ;
%Straight Lines [id:da8171366291707461] 
\draw    (347,221) -- (355.8,245.4) ;
%Curve Lines [id:da2592943547522182] 
\draw [fill={rgb, 255:red, 155; green, 155; blue, 155 }  ,fill opacity=0.2 ]   (97.4,141.4) .. controls (119.4,123) and (111,60.2) .. (187.4,67.4) ;
%Curve Lines [id:da8125072268941842] 
\draw [fill={rgb, 255:red, 155; green, 155; blue, 155 }  ,fill opacity=0.2 ]   (187.4,67.4) .. controls (229.8,42.2) and (292.6,77.4) .. (301,113) ;
%Curve Lines [id:da4372022267557947] 
\draw [fill={rgb, 255:red, 155; green, 155; blue, 155 }  ,fill opacity=0.2 ]   (301,113) .. controls (346.6,104.6) and (385.8,141) .. (374.6,224.2) ;
%Curve Lines [id:da5805117888478118] 
\draw [fill={rgb, 255:red, 155; green, 155; blue, 155 }  ,fill opacity=0.2 ]   (98.6,184.2) .. controls (149.8,227.4) and (281.4,273.4) .. (355.8,245.4) ;
%Curve Lines [id:da018654288202428315] 
\draw    (355.8,245.4) .. controls (360.67,242.5) and (364.67,240.17) .. (371,236.83) ;
%Curve Lines [id:da3412020579161663] 
\draw    (371,236.83) .. controls (371,232.83) and (373.4,229.4) .. (374.6,224.2) ;
%Shape: Polygon [id:ds9394924888377216] 
\draw  [draw opacity=0][fill={rgb, 255:red, 155; green, 155; blue, 155 }  ,fill opacity=0.2 ] (187.4,67.4) -- (184.2,92.2) -- (115.4,162.2) -- (97.4,141.4) -- cycle ;
%Curve Lines [id:da5357752105763686] 
\draw [fill={rgb, 255:red, 255; green, 255; blue, 255 }  ,fill opacity=1 ]   (115.4,162.2) .. controls (153.8,128.6) and (139,89.4) .. (184.2,92.2) ;
%Shape: Polygon [id:ds5382694666346873] 
\draw  [draw opacity=0][fill={rgb, 255:red, 155; green, 155; blue, 155 }  ,fill opacity=0.2 ] (187.4,67.4) -- (301,113) -- (277,121.8) -- (201.4,89) -- cycle ;
%Curve Lines [id:da6567462118566622] 
\draw [fill={rgb, 255:red, 255; green, 255; blue, 255 }  ,fill opacity=1 ]   (201.4,89) .. controls (221.8,75) and (268.2,97.4) .. (277,121.8) ;
%Shape: Polygon [id:ds14140127426600235] 
\draw  [draw opacity=0][fill={rgb, 255:red, 155; green, 155; blue, 155 }  ,fill opacity=0.2 ] (301,113) -- (374.6,224.2) -- (347,221) -- (304.6,138.2) -- cycle ;
%Curve Lines [id:da3772490405107618] 
\draw [fill={rgb, 255:red, 255; green, 255; blue, 255 }  ,fill opacity=1 ]   (347,221) .. controls (352.2,171.4) and (340.6,136.6) .. (304.6,138.2) ;
%Shape: Polygon [id:ds6675926663591841] 
\draw  [draw opacity=0][fill={rgb, 255:red, 155; green, 155; blue, 155 }  ,fill opacity=0.2 ] (115.4,162.2) -- (347,221) -- (355.8,245.4) -- (98.6,184.2) -- cycle ;
%Curve Lines [id:da676230880754012] 
\draw [fill={rgb, 255:red, 255; green, 255; blue, 255 }  ,fill opacity=1 ]   (115.4,162.2) .. controls (182.6,205) and (269.8,236.2) .. (347,221) ;
%Shape: Polygon [id:ds8878496190726676] 
\draw  [draw opacity=0][fill={rgb, 255:red, 155; green, 155; blue, 155 }  ,fill opacity=0.4 ] (97.4,141.4) -- (115.4,162.2) -- (98.6,184.2) -- (74.6,153.4) -- cycle ;
%Curve Lines [id:da3855903566074561] 
\draw [fill={rgb, 255:red, 255; green, 255; blue, 255 }  ,fill opacity=1 ]   (74.6,153.4) .. controls (75.88,153.2) and (77.11,152.95) .. (78.29,152.65) .. controls (88.87,150) and (95.24,143.92) .. (97.4,141.4) ;
%Curve Lines [id:da19614096006357273] 
\draw [fill={rgb, 255:red, 155; green, 155; blue, 155 }  ,fill opacity=0.4 ]   (74.6,153.4) .. controls (80.6,159.4) and (84.2,171) .. (98.6,184.2) ;
%Shape: Polygon [id:ds9022777635041745] 
\draw  [draw opacity=0][fill={rgb, 255:red, 155; green, 155; blue, 155 }  ,fill opacity=0.4 ] (201.4,89) -- (187.4,67.4) -- (184.2,92.2) -- (193,92.6) -- cycle ;
%Curve Lines [id:da9999112005684856] 
\draw [fill={rgb, 255:red, 255; green, 255; blue, 255 }  ,fill opacity=1 ]   (184.2,92.2) .. controls (189,91.4) and (185.4,91.8) .. (193,92.6) ;
%Curve Lines [id:da9681169633517563] 
\draw [fill={rgb, 255:red, 255; green, 255; blue, 255 }  ,fill opacity=1 ]   (193,92.6) .. controls (195.4,90.2) and (196.2,89.8) .. (201.4,89) ;
%Shape: Polygon [id:ds4580387797642874] 
\draw  [draw opacity=0][fill={rgb, 255:red, 155; green, 155; blue, 155 }  ,fill opacity=0.4 ] (301,113) -- (277,121.8) -- (282.6,141.8) -- (304.6,138.2) -- cycle ;
%Curve Lines [id:da20445210812461967] 
\draw [fill={rgb, 255:red, 255; green, 255; blue, 255 }  ,fill opacity=1 ]   (282.6,141.8) .. controls (287.4,141.4) and (293.33,137.83) .. (304.6,138.2) ;
%Curve Lines [id:da8265182411968249] 
\draw [fill={rgb, 255:red, 255; green, 255; blue, 255 }  ,fill opacity=1 ]   (282.6,141.8) .. controls (282.6,131) and (278.67,125.83) .. (277,121.8) ;
%Shape: Polygon [id:ds6061880815693634] 
\draw  [draw opacity=0][fill={rgb, 255:red, 155; green, 155; blue, 155 }  ,fill opacity=0.4 ] (371,236.83) -- (355.8,245.4) -- (347,221) -- (374.6,224.2) -- cycle ;
%Curve Lines [id:da4875368344855643] 
\draw [line width=1.5]    (190,80) .. controls (227.4,55.4) and (297,93.4) .. (290,130) ;
%Curve Lines [id:da46267311752116147] 
\draw [line width=1.5]    (290,130) .. controls (365.8,102.2) and (367.4,205.8) .. (360,230) ;
%Curve Lines [id:da7945653501167707] 
\draw [line width=1.5]    (90,160) .. controls (141.4,210.6) and (288.2,261) .. (360,230) ;
%Curve Lines [id:da10221800076705079] 
\draw [line width=1.5]    (90,160) .. controls (137.4,154.2) and (113.8,67.8) .. (190,80) ;

% Text Node
\draw (208.67,144) node [anchor=north west][inner sep=0.75pt]   [align=left] {$\displaystyle \Omega_{\delta }^{*}$};
% Text Node
\draw (308,70.6) node [anchor=north west][inner sep=0.75pt]   [align=left] {$\displaystyle \Omega_{\delta }^{*}$};
% Text Node
\draw (220.73,192.27) node [anchor=north west][inner sep=0.75pt]  [font=\normalsize,color={black}  ,opacity=1 ] [align=left] {$\displaystyle \Gamma $};
% Text Node
\draw (184.47,80.93) node [anchor=north west][inner sep=0.75pt]  [font=\tiny] [align=left] {$\displaystyle V_{1,\delta }$};
% Text Node
\draw (95.67,158.67) node [anchor=north west][inner sep=0.75pt]  [font=\tiny] [align=left] {$\displaystyle V_{2,\delta }$};
% Text Node
\draw (351,232) node [anchor=north west][inner sep=0.75pt]  [font=\tiny] [align=left] {$\displaystyle V_{3,\delta }$};
% Text Node
\draw (280.67,130) node [anchor=north west][inner sep=0.75pt]  [font=\tiny] [align=left] {$\displaystyle V_{4,\delta }$};
% Text Node
\draw (150,85) node [anchor=north west][inner sep=0.75pt]  [font=\tiny] [align=left] {$\displaystyle W_{1,\delta }$};
% Text Node
\draw (215.33,231) node [anchor=north west][inner sep=0.75pt]  [font=\tiny] [align=left] {$\displaystyle W_{2,\delta }$};
% Text Node
\draw (314.33,114.67) node [anchor=north west][inner sep=0.75pt]  [font=\tiny] [align=left] {$\displaystyle W_{3,\delta }$};
% Text Node
\draw (222,62.33) node [anchor=north west][inner sep=0.75pt]  [font=\tiny] [align=left] {$\displaystyle W_{4,\delta }$};

\end{tikzpicture}}
        \caption{Decomposition of a neighborhood of a curve with corners}
        \label{fig: decomposition near Gamma}
\end{figure}
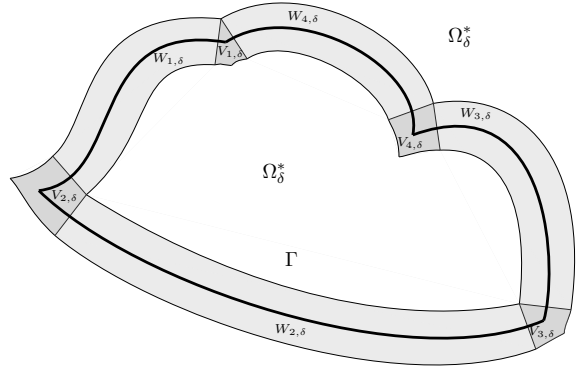

Finally,  recall that the  Schr\"{o}dinger operator with a strong $\delta$-interaction of strength $\alpha>0$ supported on a curve $\Gamma$ with $M\geq1$ corners, whose eigenvalue asymptotics we aim to derive in this final section, is the self-adjoint operator $H^\Gamma_\alpha$ defined by the sesquilinear form
    \[
    h^\Gamma_\alpha (u) = \int_{\rr^2} |\nabla u|^2 \mathrm{d}x - \alpha \int_\Gamma |u|^2 \dS , \quad D(h^\Gamma_\alpha) = H^1(\rr^2).
    \]
   Moreover, for $X \in \{ d,n \}$ and $U_{j,\delta} \in \{ V_{j,\delta},W_{j,\delta}\}$, $j=1,\ldots,M$, we define the forms
    $$ X^U_j(u) = \int_{U_{j,\delta}} | \nabla u|^2 \mathrm{d}x - \alpha \int_{U_{j,\delta} \cap \Gamma} |u|^2 \dS, \quad   D(n^U_j) = H^1(U_{j,\delta}), \, D(d^U_j) = H^1_0(U_{j,\delta})  $$
    i.e., the Dirichlet/Neumann Laplacians in $V_{j,\delta}$ / $W_{j,\delta}$ with $\delta$-interactions supported on their respective parts of $\Gamma$. 
    
    The following operators will also be needed:
    \begin{itemize}
        \item $N_0$ := Neumann Laplacian on $\Omega_\delta^*$,
        \item $D_{j,\delta}$ := Dirichlet Laplacian on $I_{j,\delta}$,
        \item $D_j$ := Dirichlet Laplacian on $(0,l_j)$,
        \item $R_j^V$:= the operator $N_j^V$ with an additional $\alpha$-Robin boundary condition on $\partial_* V_{j,\delta}$.
    \end{itemize}

\subsection{Asymptotics of corner-induced eigenvalues.} In this subsection, we prove the main result, Theorem \ref{theo: corner induced}. The following lemma summarizes key results from Section \ref{sec: V delta}.
\begin{lemma}
\label{lem: eigenvalues in corner neighborhoods}
As $\alpha \to \infty$, $\delta\to 0$, and $\alpha \delta \to \infty$, the following asymptotic relations holds:
    \begin{align*}
    E_n\big(\bigoplus_{j=1}^{M} N_j^V \big) &= \alpha^2 \mathcal{E}_n  + \mathcal{O}(\alpha^2 \delta+ \frac{1}{\delta^2}), \quad n = 1,\ldots,\mathcal{K},\\
    E_n\big(\bigoplus_{j=1}^{M} D_j^V\big) &= \alpha^2\mathcal{E}_n  + \mathcal{O}(\alpha^2 \delta + \alpha^2 e^{-c \alpha \delta}) , \quad n = 1,\ldots,\mathcal{K},\\
     E_{\mathcal{K}+1} \big(\bigoplus_{j=1}^{M} N_j^V\big) &\geq - \frac{\alpha^2}{4} + {o}(\alpha^2).
     \end{align*}
\end{lemma}
\begin{proof}
For $\theta_j \in (0,\frac{\pi}{2})$, Corollary \ref{lem: ev dirichlet neuman truncated curved} yields $E_{\mathcal{K}+1} (N_j^V) \geq - \alpha^2/4 + o(\alpha^2)$, and
$$ E_n(N_j^V) = \alpha^2 \mathcal{E}_n(\theta_j) + \mathcal{O}\big(\alpha^2 \delta + \frac{1}{\delta^2}\big) \quad \text{for } n = 1, \dots, \kappa(\theta_j). $$
For $\theta_j \in (\frac{\pi}{2} ,\pi)$ and $n = 1, \dots, \kappa(\pi-\theta_j) = \kappa(\theta_j)$,  Proposition \ref{lem: big prop}(ii) gives
$$ E_n(N_j^V) = \alpha^2 \mathcal{E}_n(\pi-\theta_j) + \mathcal{O}\big(\alpha^2 \delta + \frac{1}{\delta^2}\big) = \alpha^2 \mathcal{E}_n(\theta_j) + \mathcal{O}(\alpha^2 \delta + \frac{1}{\delta^2}).$$
The first and third identities in the lemma follow directly from these results. The second identity can be established analogously by applying Corollary \ref{lem: ev dirichlet neuman truncated curved}. 
\end{proof}
Using this lemma, we conclude the proof of Theorem \ref{theo: corner induced} by proving the following Proposition.
\begin{proposition}\label{lem: corner induced EV} As $\alpha \to \infty$, it holds that
    \begin{align}
    E_n(H^\Gamma_\alpha) &= \mathcal{E}_n \alpha^2 + \mathcal{O}(\alpha^{\frac{4}{3}}) \text{ for each } n \in \{1,\dots,\mathcal{K} \}, \label{est1}\\
    E_{n}(H^\Gamma_\alpha)&=-\dfrac{\alpha^2}{4}+o(\alpha^2) \text{ for each }n\ge \mathcal{K}+1 \label{est2}.
    \end{align}
\end{proposition}

\begin{proof}
%The first identity is a direct result of CITE[Brasche, Exner, Kuperin, Seba] Theorem 3.1 as $\Gamma$ is a closed curve.\\
    By applying Dirichlet-Neumann bracketing, for any $n \in \mathbb{N}$ we have
    \begin{align} \label{eq: proof corner ev 1} 
    E_n\big(N_0 \oplus \big(\bigoplus_{j=1}^M N_j^V\big) \oplus\big(\bigoplus_{j=1}^M N_j^W\big)\big)     \leq \Lambda_n\big(H^\Gamma_\alpha\big) \leq E_n\big(\bigoplus_{j=1}^M D_j^V\big).
    \end{align}
    By Lemma \ref{lem: ev strip neuman}, we have the lower bound $ N_j^W \geq -\alpha^2/4 + o(\alpha^2)$ for all $j = 1,\dots M$, and trivially $N_0 \geq 0 $. Combining this with Lemma \ref{lem: eigenvalues in corner neighborhoods}, it follows that
    \[
    E_n(N_0 \oplus (\bigoplus_{j=1}^M N_j^V) \oplus(\bigoplus_{j=1}^M N_j^W))  = E_n(\bigoplus_{j=1}^M N_j^V), \quad n = 1, \dots \mathcal{K}.
    \]
    Substituting this into \eqref{eq: proof corner ev 1}, together with Lemma \ref{lem: eigenvalues in corner neighborhoods}, yields
    \[
    \Lambda_n(H^\Gamma_\alpha) = \alpha^2 \mathcal{E}_n + \mathcal{O}(\frac{1}{\delta^2} + \alpha^2 \delta + \alpha^2 e^{-c\alpha\delta}).
    \]
    Since $\mathcal{E}_n < - \frac{1}{4}$, there exists $\alpha_0 >0$ such that for all $\alpha \geq \alpha_0$ one has  $\Lambda_n(H^\Gamma_\alpha) < \Sigma(H^\Gamma_\alpha) = 0 $, meaning $\Lambda_n(H^\Gamma_\alpha)$ is in fact the  $n$th eigenvalue of $H^\Gamma_\alpha$. Consequently, choosing $\delta := \alpha^{- \frac{2}{3}}$  gives \eqref{est1}.
    
    For the second estimate, from \eqref{eq: proof corner ev 1} and Lemma \ref{lem: eigenvalues in corner neighborhoods} we obtain
    \[
    \Lambda_{\mathcal{K}+1}(H^\Gamma_\alpha) \geq \min \big\{ \Lambda_1(N_0), E_{\mathcal{K}+1}\big(\bigoplus_{j=1}^M N_j^V \big), E_1\big(\bigoplus_{j=1}^M N_j^W\big)  \big\} \geq - \frac{\alpha^2}{4}+ o (\alpha^2).
    \]
    On the other hand, using the Dirichlet bracketing around $W_{1,\delta}$ for each $n\in\N$ we obtain
    $\Lambda_n(H^\Gamma_\alpha)\le E_n(D^W_1)$, while $E_n(D^W_1)=-\frac{\alpha^2}{4}+\mathcal{O}(1)$ by Lemma \ref{lem: ev strip dirichlet} (say, for $\delta:=1/\sqrt{\alpha}$).
    By combining these estimates, for each $n\ge \mathcal{K}+1$ we obtain
    \[
    -\dfrac{\alpha^2}{4}+o(\alpha^2)\le \Lambda_n(H^\Gamma_\alpha)\le-\dfrac{\alpha^2}{4}+\mathcal{O}(1).
    \]
    As the left-hand side is negative for large $\alpha$ (hence lies below the bottom of the essential spectrum), one also has $\Lambda_n(H^\Gamma_\alpha)=E_n(H^\Gamma_\alpha)$, which concludes the proof of \eqref{est2}.    
\end{proof}
%===================================================================
\subsection{Asymptotics of edge-induced eigenvalues}\label{ss73}
In this subsection, we prove the main result, Theorem \ref{theo: side induced}. This is accomplished by establishing suitable upper and lower bounds for $E_{\mathcal{K}+n}(H^\Gamma_\alpha)$ via the techniques developed in \cite{KOBP20}.

{}From now on, we assume that:
\[
\text{all angles $\theta_j$ are non-resonant.}
\] Moreover, we consider the following asymptotic regime for some fixed, sufficiently large $C>0$ to be specified later:
    \begin{equation} \label{eq: asymptotics} 
    \alpha \to \infty, \quad \delta =  \frac{C \log \alpha}{\alpha}\to  0,
    \end{equation}
    Under this regime, it follows that $\alpha \delta \to \infty$ and $\alpha^2 \delta^3 \to 0$.
We first derive an upper bound for $E_{\mathcal{K}+n}(H^\Gamma_\alpha)$:
\begin{proposition}
\label{lem: upper bound side induced ev}
    For any $n \in \mathbb{N}$, there exists $\alpha_0 >0$ such that for all $\alpha \geq \alpha_0$, the operator $H^\Gamma_\alpha$ has at least $\mathcal{K} +n$ discrete eigenvalues, with the following upper bound 
     \[
     E_{\mathcal{K}+n}(H^\Gamma_\alpha) \leq -\frac{\alpha^2}{4} + E_n\Big(\bigoplus_{j = 1}^M \big(D_j- \frac{k^2_j}{4} \big)\Big) + \mathcal{O} (\frac{\log \alpha}{\alpha}) .
     \]
\end{proposition}
\begin{proof}
    Let $n \in \mathbb{N}$. By the Dirichlet-bracketing, we have
    \begin{align}\label{eq: upper bound side induced ev1}
        \Lambda_{\mathcal{K}+n}(H^\Gamma_\alpha) \leq E_{\mathcal{K}+n}\big(\big(\bigoplus_{j =1}^M D_j^V\big)\oplus \big(\bigoplus_{j = 1}^M D_j^W\big)\big).
    \end{align} 
    Due to the non-resonance assumption, Corollary \ref{lem: idk man i just got here} guarantees the existence of a constant $c_0 >0$ such that
    $ E_{\kappa(\theta_j) +1} (D_j^V) \geq -\alpha^2/4 + \frac{c_0}{\delta^2}$ holds for each $j=1,\dots M$. Consequently,
    \[
    E_{\mathcal{K} +1} \big(\bigoplus_{j = 1}^M D_j^V\big) \geq - \frac{\alpha^2}{4} + \frac{c_0}{\delta^2}.
    \]
    Choosing $C\geq 6$ in the asymptotic regime \eqref{eq: asymptotics}, it follows from Lemma \ref{lem: ev strip dirichlet} that there exists $c_D >0$ such that
    \[
    E_n(D_j^W) \leq -\frac{\alpha^2}{4} + E_n( D_j - \frac{k_j^2}{4}) + c_D( \delta + \alpha^2 e^{-\frac{1}{2} \delta \alpha} ) \leq -\frac{\alpha^2}{4} + E_n( D_j - \frac{k_j^2}{4}) + c_D( \frac{1+C\log\alpha}{\alpha} ).
    \]
    Combining the above leads to    
    \begin{equation} E_n(\bigoplus_{j=1}^M D_j^W ) \leq - \frac{\alpha^2}{4} + E_n \left(\bigoplus_{j = 1}^M (D_j - \frac{k_j^2}{4}) \right) + \mathcal{O}(\frac{\log  \alpha}{\alpha})
    \leq -\frac{\alpha^2}{4} + \frac{c_0}{\delta^2} \leq E_{\mathcal{K}+1} (\bigoplus_{j =1}^M D_j^V ) \end{equation}
    Moreover, by Lemma \ref{lem: eigenvalues in corner neighborhoods}, for sufficiently large $\alpha$, 
    \begin{equation*}
         E_{\mathcal{K}} (\bigoplus_{j = 1}^M D_j^V) \leq -(\frac{1}{4} + \varepsilon) \alpha^2 = -\frac{\alpha^2}{4} - \alpha^2 \varepsilon  \leq E_n(\bigoplus_{j = 1}^M D_j^W)  \leq  E_{\mathcal{K}+1} (\bigoplus_{j = 1}^M D_j^V) 
    \end{equation*}
    for some $\varepsilon>0$. Together with the inequality \eqref{eq: upper bound side induced ev1}, it follows that
    $$ \Lambda_{\mathcal{K}+n}(H^\Gamma_\alpha) \leq E_{\mathcal{K} + n} ((\bigoplus_{j =1}^M D_j^V  ) \oplus ( \bigoplus_{j =1}^M  D_j^W )) = E_n(D_j^W) \leq - \frac{\alpha^2}{4} + E_n\left(\bigoplus_{j = 1}^M (D_j - \frac{k_j^2}{4})\right) + \mathcal{O}(\frac{\log  \alpha}{\alpha})$$
    We conclude the proof by noting that there exists $\alpha_0>0$ such that, for $\alpha \geq \alpha_0$, one has $\Lambda_{\mathcal{K} + n}(H^\Gamma_\alpha) <\Sigma(H^\Gamma_\alpha) =0$, and thus $E_{\mathcal{K}+n}(H^\Gamma_\alpha) = \Lambda_{\mathcal{K}+n}(H^\Gamma_\alpha)$.
\end{proof}

The remainder of this subsection is devoted to establishing a lower bound for the edge-induced eigenvalues. We begin by introducing some notation. From now on, we set
\begin{itemize}
    \item $L:=$ the subspace of $L^2(\rr^2)$ spanned by the first $\mathcal{K}$ eigenfunctions of $H^\Gamma_\alpha$,
    \item $L_j :=$ the subspace of $L^2(V_{j,\delta})$ spanned by the first $\kappa(\theta_j)$ eigenfunctions of $N_j^V$,
    \item $\sigma_j:L^2(\rr^2) \to L^2(V_{j,\delta})$ the restriction operator onto $V_{j,\delta}$.
\end{itemize}
Before proceeding to the final proof, we require two preliminary lemmas. The first concerns the distance between the subspaces $ \sigma^*_j L_j $ and $L$.
\begin{lemma}\label{lem: estimate for subspaces}
    For any $j \in \{1,\dots , M \}$ and under the asymptotics regime \eqref{eq: asymptotics}, there exists $c>0$ such that
    $d( \sigma_j^* L_j, L) = \mathcal{O}(e^{- c \alpha \delta})$.
\end{lemma}
\begin{proof}
Let $j \in \{1,\dots,M \}$ be fixed. Set $ \Upsilon_j : = \sigma^*_j L_j $ and $v^* :=  \sigma^*_j v$ for $v \in L^2(V_{j, \delta})$. Our goal is to estimate $d(\Upsilon, L)$ using the triangular inequality \eqref{lem: triangular inequality distance subspaces}:
\begin{align} \label{eq: ladidadida}
d(\Upsilon_j, L ) \leq d(\Upsilon_j, \Upsilon_j^\chi) + d(\Upsilon_j^\chi , L),
\end{align}
where $\Upsilon_j^\chi$ is an intermediate subspace defined via a suitable cutoff.

Recall from Lemma \ref{lem: cuttoff function for V delta} that there exist constants $0<a<A<1$ and a cutoff function $\chi_\delta \in C^2(\rr^2) $ such that for all $j \in \{1,\dots, M \}$:
\begin{itemize}
    \item $0 \leq \chi_\delta \leq 1$, with $\chi_\delta \equiv 1$ in $V_{j,a \delta}$ and $\chi_\delta \equiv 0 $ in $\rr^2 \setminus \overline{V_{j, A \delta}}$,
    \item for all $\beta \in \mathbb{N}^2 $ with $1 \leq | \beta | \leq 2$, there exists $C_0 > 0 $ such that $ \|\partial^\beta \chi_\delta\| \leq  C_0 \delta^{- | \beta |}$, and the normal derivative of $\chi_\delta$ vanishes on $\Gamma$.
\end{itemize}
Moreover, there exists $a_0 > 0$ such that $ |x-A_j| > a_0 \delta $ for all $x \in V_{j,\delta} \setminus \overline{V_{j, a \delta}} $. We then define the subspace $\Upsilon_j^\chi$ by
\[ 
\Upsilon_j^\chi : = \{ \chi_\delta v^* :\, v^* \in \Upsilon_j \} \subseteq L^2(\rr^2).
\]
We first estimate $d(\Upsilon_j, \Upsilon_j^\chi)$. By definition, we have
\[
d(\Upsilon_j, \Upsilon_j^\chi) = \sup_{0 \neq v^* \in \Upsilon_j} \frac{\|v^* - P_j v^*\|}{\|v^*\|},
\]
where $P_j:  L^2(\rr^2) \rightarrow \Upsilon_j^\chi$ denotes orthogonal projector. Due to the construction of $\chi$, for any $v^* \in \Upsilon_j$, 
\[
\|v^* - \chi_\delta v^*\|_{L^2(\mathbb{R}^2)} = \|(1 - \chi_\delta) v^*\|_{L^2(\mathbb{R}^2)} \leq  \|v^*\|_{L^2(V_{j,\delta} \setminus \overline{V_{j, a\delta}})}.
\]
Now, we use the Agmon estimate from Lemma \ref{lem: agmon truncated curved} for the first $\kappa(\theta_j)$ eigenfunctions of $N_j^V$ to obtain constants $b_,B > 0$ such that for all $v \in L_j$,
\[
\int_{V_{j,\delta}} e^{b \alpha |x-A_j |} (\frac{1}{\alpha^2} | \nabla v|^2 + | v |^2) \mathrm{d}x \leq B \| v  \|^2_{L^2(V_{j, \delta})}.
\]
Therefore,
\begin{align*}
    \int_{V_{j, \delta} \setminus \overline{V_{j, a \delta}}} (\frac{1}{\alpha^2} |\nabla v|^2 + | v |^2 ) \mathrm{d} x &=  \int_{V_{j, \delta} \setminus \overline{V_{j, a \delta}}} e^{-b \alpha |x-A_j |} \cdot e^{b \alpha |x-A_j |} (\frac{1}{\alpha^2} |\nabla v|^2 + | v^\ast |^2 ) \mathrm{d} x \\
   &  \leq e^{- b \alpha a_0 \delta} \int_{V_{j, \delta} \setminus \overline{V_{j, a \delta}}} e^{b \alpha |x-A_j | } (\frac{1}{\alpha^2} |\nabla v|^2 + | v^\ast |^2 ) \mathrm{d} x \\ 
     &\leq e^{- b \alpha a_0 \delta} \int_{V_{j, \delta} } e^{b \alpha |x-A_j | } (\frac{1}{\alpha^2} |\nabla v|^2 + | v^\ast |^2 ) \mathrm{d} x 
\leq Be^{- b \alpha a_0 \delta} \| v  \|^2_{L^2(V_{j,\delta})}.
\end{align*}
Setting $c:= (b a_0)/2$, this implies
\[
\int_{V_{j, \delta} \setminus \overline{V_{j, a \delta}}} (\frac{1}{\alpha^2} |\nabla v|^2 + | v |^2 ) \mathrm{d} x \leq Be^{- 2 c \alpha \delta} \| v  \|^2_{L^2(V_{j,\delta})}.
\]
In particular, 
\begin{align*}
 \| v^* - \chi_\delta v^* \|^2_{L^2(\rr^2)} \leq  \|v^*\|_{L^2(V_{j,\delta} \setminus \overline{V_{j, a\delta}})}&=\|v\|_{L^2(V_{j,\delta} \setminus \overline{V_{j, a\delta}})}\leq B e^{-2 c \alpha \delta} \| v  \|^2_{L^2(V_{j,\delta})} 
= B e^{-2 c \alpha \delta} \| v^*  \|^2_{L^2(\rr^2)}.
\end{align*}
Hence,
\[
\frac{\| v^* - P_j v^*\|}{\| v^*\|} =  \frac{\inf_{u \in \Upsilon^{\chi}_{j}} \| v^* - u\|}{\| v*\|} \leq \frac{\| v^* - \chi_\delta v^*\|}{\| v^*\|} = \frac{\|(1- \chi_\delta) v^*\|}{\| v^*\|} \leq \sqrt{B} e^{- c \alpha \delta},
\]
and thus
\begin{align}\label{IN1}
    d(\Upsilon_j, \Upsilon_j^\chi) = \sup_{0 \neq v^* \in \Lambda_j} \frac{\| v^* - P_j v^*\|}{\| v^*  \|} \leq \sqrt{B} e^{-c \alpha \delta}.
\end{align}

We now turn to estimate $d(\Upsilon_j^\chi, L)$. Let $\psi_1, \ldots \psi_{\kappa(\theta_j)}$ denote the first $\kappa(\theta_j)$ eigenfunctions of $N_j^V$ associated with $E_1, \ldots, E_{\kappa(\theta_j)}$. Define $\tilde{\psi}_n := \chi_\delta \psi_n$ for $n=1,\ldots,\kappa(\theta_j)$. Since $\psi_n$ is an eigenfunction of $N_j^V$, we have
\[
- \Delta \psi_n = E_n  \psi_n \,\,\text{ in } V_{j,\delta}, \quad \alpha(\partial_\nu \psi_n^+ - \partial_\nu\psi_n^- ) = \frac{1}{2}\big( \psi_n^+ + \psi_n^- \big)\,\, \text{ on } \Gamma \cap V_{j,\delta},
\]
where $\partial_\nu$ is the normal derivative and we recall that $\psi_n^{\pm}$ denotes the restriction of $\psi_n$ onto $\Omega_\pm$. As $\chi_\delta $ is smooth with a support contained in $V_{j,\delta}$, we have $\tilde{\psi_n} \in H^1(\rr^2)$. Moreover, we have 
\begin{align*}
\Delta \tilde\psi_n &= \Delta \chi_\delta \cdot  \psi_n + 2  \nabla \chi_\delta \cdot \nabla \psi_n  + E_n \chi_\delta \psi_n \in L^2(\rr^2),
\end{align*}
and since $\partial_\nu\chi_\delta$ vanishes on $\Gamma$,  
\begin{align*}
\alpha\big( \partial_\nu(\tilde{\psi_n})^+ - \partial_\nu(\tilde{\psi_n})^- \big)& = 
\alpha\left( \partial_\nu \chi_\delta \psi_n^+   + \partial_\nu \psi_n^+ \chi_\delta  - \chi_\delta\partial_\nu   \psi_n^-   - \chi_\delta\partial_\nu\psi_n^-  \right)  =
\frac{1}{2}( \tilde{\psi_n^+} + \tilde{\psi_n^-} ) 
\end{align*}
holds on  $\Gamma \cap V_{j,\delta}$. Hence,  $\tilde \psi_n\in D(H^\Gamma_\alpha)$, and one easily shows that $\tilde{\psi}_1,\dots, \tilde{\psi}_{\kappa(\theta_j)}$ are linearly independent. Thus, by Proposition \ref{lem: distance between subspaces}, we have
\begin{align*}
d(\Upsilon_j^\chi, L) \leq \frac{\varepsilon}{\eta} \sqrt{\frac{\kappa(\theta_j)}{\lambda}}
\text{ with }
\varepsilon= \max_{n} \| (H^\Gamma_\alpha - E_n ) \tilde\psi_n\|, \ \eta = \frac{1}{2} \text{dist}(I, \spec(H^\Gamma_\alpha)\setminus I), 
\end{align*}
with $I$ an interval containing $E_1, \ldots , E_{\kappa(\theta_j)}$,  and
\[
\lambda = \text{the smallest eigenvalue of the Gram matrix } ( \langle \tilde\psi_k, \tilde\psi_l \rangle)_{k,l}. 
\]
We are now going to construct a suitable $I$ and estimate $\varepsilon$, $\lambda$, and $\eta$. We first estimate $\varepsilon$.  For this, we compute the norm of $ (H^\Gamma_\alpha - E_n)\tilde{\psi}_n = -(\Delta \chi_\delta) \psi_n - 2 \langle \nabla \chi_\delta , \nabla \psi_n\rangle $. Since the supports of $\nabla \chi_\delta$ and $\Delta \chi_\delta$ lie in $ V_{j, A\delta} \setminus \overline{V_{j, a\delta}} $, by the Agmon estimate, we obtain
\begin{align*}
\int_{\rr^2} | (\Delta \chi_\delta) \psi_n|^2 \mathrm{d}x \leq \frac{C_0^2}{\delta^4} \int_{V_{j, b \delta} \setminus \overline{V_{j,a \delta}}} | \psi_n |^2 \mathrm{d}x  \leq \frac{B C_0^2}{\delta^4} e^{-2c \alpha \delta } \| \psi_n \|^2_{L^2(V_{j,\delta})} =\frac{B C_0^2}{\delta^4} e^{-2c \alpha \delta };\\
 \int_{\rr^2} | \nabla \chi_\delta \cdot \nabla \psi_n |^2 \mathrm{d}x \leq \int_{\rr^2} | \nabla \chi_\delta |^2 |\nabla \psi_n|^2  \mathrm{d}x \leq \frac{C_0^2}{\delta^2} \int_{V_{j,A\delta} \setminus \overline{V_{j,a\delta}}}  |\nabla\psi_n|^2 \mathrm{d}x \leq \frac{B C_0^2 \alpha^2}{\delta^2} e^{-2 c\alpha \delta}.
 \end{align*}
Since $1/\delta^2 = o (\alpha/\delta)$ as $\alpha\to\infty$, it follows that $ \varepsilon = \| (H^\Gamma_\alpha - E_n)\tilde\psi_n\|_{L^2} = \mathcal{O}(\frac{\alpha}{\delta} e^{-c\alpha\delta})$.

Let us now estimate the smallest eigenvalue of the Gram matrix:
\begin{align*}
    | \langle \tilde{ \psi}_k , \tilde{\psi}_n \rangle_{L^2(\rr^2)} -  \langle \psi_k, \psi_n \rangle_{L^2(\rr^2)} | &= \left|\,\int_{\rr^2} (\chi_\delta^2 -1 ) \psi_k \psi_n \mathrm{d}x\right| \leq \int_{V_{j, A\delta} \setminus \overline{V_{j, a\delta}}} |\psi_k \psi_n| \mathrm{d}x\\
    &\leq \frac{1}{2} \bigg(\int_{V_{j, A\delta} \setminus \overline{V_{j, a\delta}}} |\psi_k|^2 \mathrm{d}x + \int_{V_{j, A\delta} \setminus \overline{V_{j, a\delta}}} | \psi_n |^2 \mathrm{d}x \bigg)\leq B e^{-2c\alpha \delta}.
\end{align*}
This implies that $\langle \tilde{\psi}_k, \tilde{\psi}_n \rangle_{L^2(\rr^2)} = \delta_{k,n} + \mathcal{O}(e^{-2c\alpha \delta})$. By standard perturbation arguments,  the lowest eigenvalue satisfies $\lambda = 1 + \mathcal{O}(e^{-2c\alpha \delta})$, ensuring that $\tilde{\psi}_1,\dots,\tilde{\psi}_{\kappa(\theta)}$ are linearly independent. 

Finally, choose the interval $I:= ( \alpha^2 (\mathcal{E}_1 - h), \alpha^2(\mathcal{E}_\mathcal{K} +h))$ with $h :=(- 1/4 - \mathcal{E}_{\mathcal{K}} )/2$ and $\mathcal{E}_1, \ldots, \mathcal{E}_{\mathcal{K}}$ as in Lemma \ref{lem: corner induced EV}. Then, $\{ E_1, \dots , E_{\kappa(\theta)}  \} \subset I $, and Lemma \ref{lem: corner induced EV} guarantees that $ \eta \geq (h \alpha^2)/8$. Combining these estimates, we get
\[
d(\Upsilon_j^\chi, L) \leq \frac{\varepsilon}{\eta} \sqrt{\frac{\kappa(\theta_j)}{\lambda}} \leq \frac{8\mathcal{O}(\frac{\alpha}{\delta} e^{-2c\alpha\delta} )}{\alpha^2h} \cdot \sqrt{\frac{\kappa(\theta_j)}{1+ \mathcal{O}(e^{-2c\alpha\delta})}} = \mathcal{O}(\frac{e^{-2c \alpha \delta}}{\alpha \delta}).
\]
This, together with \eqref{IN1} and \eqref{eq: ladidadida}, implies the desired result and completes the proof.
\end{proof}
With the help of Lemma \ref{lem: estimate for subspaces}, we can now derive norm and trace estimates similar to those stated in Corollary \ref{lem: upper bound LT in V delta }.
\begin{lemma}
\label{lem: upper bound orthogonal complement in R2}
Let $u \in H^1(\rr^2)$ satisfy $u\perp L $. Then there exist $b,c >0$ such that, under the asymptotic regime \eqref{eq: asymptotics} and for any $ j \in \{1,\dots,M \}$, the following estimates hold:
    \begin{align*}
        \| \sigma_j u\|^2_{L^2(V_{j,\delta})} &\leq b \delta^2(n_j^V(\sigma_j u) + \frac{\alpha^2}{4} \|\sigma_j u\|^2_{L^2(V_{j,\delta})} ) + b \alpha^2 \delta^2 e^{-c \alpha \delta} \| u \|^2_{L^2};\\
        \int_{\partial_* V_{j,\delta} } |\sigma_j u|^2 \dS &\leq b \alpha \delta^2(n_j^V(\sigma_j u) + \frac{\alpha^2}{4} \|\sigma_j u\|^2_{L^2(V_{j,\delta})} ) + b \alpha^3 \delta^2 e^{-c \alpha \delta} \| u \|^2_{L^2}.
        \end{align*}
\end{lemma}
\begin{proof}
    Consider the orthogonal projectors
    $P: L^2(\rr^2) \rightarrow L$, $P_j:L^2(V_{j,\delta}) \rightarrow L_j$
    and set, for $u \perp L$,
    \[
    v = \sigma_ju \in L^2(V_{j,\delta}),\quad v_P = P_jv \in L^2(V_{j,\delta}), \quad v_0 =v  - v_P = (1-P_j) v \in L^2(V_{j,\delta}).
    \]
    Our goal is to estimate the norm of $v$ by splitting $ \| v  \|^2 = \| v_P\|^2 + \| v_0\|^2 $. 
    An upper bound for $ \| v_0 \|$ can be derived using Corollary \ref{lem: upper bound LT in V delta } as $v_0$ is orthogonal to $L_j$. To estimate for $\| v_P  \|$, we use that $u = (1-P)u $, hence
    \begin{align} \label{eq: upper bound orthogonal complement in R2 1} 
    \| v_P  \|_{L^2(V_{j,\delta})} = \| \sigma_j^* P_j v \|_{L^2(R^2)} = \|\sigma^*_j P_j \sigma_j (1-P)u\|_{L^2(V_{0j,\delta})} \leq \| \sigma_j^* P_j \sigma_j(1-P)\| \| u \|_{L^2(\rr^2)} 
    \end{align}
    Note that $\sigma_j^* P_j \sigma_j: L^2(\rr^2) \to \sigma_j^* L_j $ is the orthogonal projector onto $ \sigma^*_j L_j$. Using Lemma \ref{lem: estimate for subspaces}, there exists $c>0$ such that
    \[
    \| \sigma_j^* P_j \sigma_j(1-P)\| = \| \sigma_j^* P_j \sigma_j-\sigma_j^* P_j \sigma_jP)\| =d(\sigma_j^* L_j, L_j)  = \mathcal{O}(e^{-\frac{c}{2}\alpha \delta}).
    \]
    Combining this with \eqref{eq: upper bound orthogonal complement in R2 1}, for some $b_0 >0$ we have
    $\| v_P\|_{L^2(V_{j,\delta})}^2 \leq b_0 e^{-c\alpha\delta} \| u \|^2_{L^2(\rr^2)}$.
    Hence, by Corollary \ref{lem: upper bound LT in V delta }, there also exists $b_1 >0$ such that
    \begin{equation} \label{eq: upper bound orthogonal complement in R2 2} 
    \| v  \|^2 = \|v_0\|^2 + \| v_P  \|^2 \leq b_1 \delta^2 (n_j^V(v_0) + \frac{\alpha^2}{4} \|v_0\| ) + b_0^2 e^{-2c \alpha \delta} \| u  \|^2.
    \end{equation}
    To estimate $n_j^V(v_0)$, recall that since $P_j$ is a spectral projector for $N_j^V$,  we have $ n_j^V(v) = n_j^V(v_0) + n_j^V(v_p) $. Furthermore, as $v_P \in L_j$, we have
    \[
    E_1(N_j^V)\| v_P  \|^2_{L^2()V_{j,\delta}} \leq n_j^V(v_P) \leq E_{\kappa(\theta_j)}(N_j^V) \| v_P  \|^2_{L^2(V_{j,\delta})}.
    \]
    By Lemma \ref{lem: eigenvalues in corner neighborhoods}, eigenvalues satisfy $E_n(N_j^V) = \mathcal{O}(\alpha^2) $. Hence, for some $b_2 > 0$, 
    \[
    |n_j^V(v_P)| \leq b_2 \alpha^2 e^{-2 c \alpha\delta} \| u \|^2_{L^2(\rr^2)},
    \]
    which implies
    \begin{align} \label{eq: upper bound orthogonal complement in R2 3} 
      n_j^V(v_0) \leq n_j^V(v) + b_2 \alpha^2 e^{-2c \alpha \delta} \| u \|^2_{L^2(\rr^2)}.
    \end{align}
    Substituting \eqref{eq: upper bound orthogonal complement in R2 3} into \eqref{eq: upper bound orthogonal complement in R2 2} gives: 
    \begin{align*}
         \| v  \|^2_{L^2(V_{j,\delta})} &\leq b_1 \delta^2 (n_j^V (v) + \frac{\alpha^2}{4} \|v_0\|^2_{L^2(V_{j,\delta})}) + (b_2 b_1 \alpha^2 \delta^2 e^{-2c \alpha \delta} + b_0^2 e^{-2c\alpha \delta}) \| u \|^2_{L^2(\rr^2)}\\
          &\leq b_1 \delta^2( n_j^V (v) + \frac{\alpha^2}{4} \| v  \|^2_{L^2(V_{j,\delta})}) + b_0 \alpha^2 \delta^2 e^{-2c \alpha \delta} \| u \|^2_{L^2(\rr^2)} 
    \end{align*}
    which yields the first claimed estimate
    
    For the second estimate, observe that
    \[
    \int_{\partial_* V_{j,\delta}} | v |^2 \dS = \int_{\partial_* V_{j,\delta}} | v_p + v_0|^2 \leq 2 \int_{\partial_* V_{j,\delta}} |v_P |^2 \dS + 2 \int_{\partial_* V_{j,\delta}} |v_0 |^2 \dS .
    \]
    Using  Corollary \ref{lem: upper bound LT in V delta } and estimate \eqref{eq: upper bound orthogonal complement in R2 3}, we obtain
    \begin{align*}
         \int_{\partial_* V_{j,\delta}} |v_0 |^2 \dS &\leq b_1 \alpha \delta^2 (n_j^V(v_0) + \frac{\alpha^2}{4} \|v_0\|^2_{L^2(V_{j,\delta})}) \\
        & \leq b_1 \alpha \delta^2 (n_j^V(v) + \frac{\alpha^2}{4} \| v  \|^2_{L^2(V_{j,\delta})}) + b_1b_2 \alpha^3 \delta^2 e^{-2c\alpha\delta} \| u \|^2_{L^2(\rr^2)} 
    \end{align*}
    For the integral involving $v_P$, by Lemma \ref{lem: lower bound robin bc V delta}, there exists $b_3 >0$ such that
    \[
    n_j^V(v_P) - \alpha \int_{\partial_* V_{j,\delta}} |v_P |^2 \dS = \int_{V_{j,\delta}} |\nabla v_P|^2 \mathrm{d}x- \alpha \int_{\Gamma \cap V_{j,\delta} } |v_P |^2 \dS = r_j^V(v_P) \geq -b_3 \alpha^2 \| v_P  \|^2_{L^2(V_{j,\delta})}.
    \]
    Applying Corollary \ref{lem: upper bound LT in V delta } and previous estimates for $\| v_P \|^2$ and $ n_j^V(v_P) $ yields, for some $ b_4 >0$,
    \[ \int_{\partial_* V_{j,\delta}} |v_P |^2 \dS \leq \frac{1}{\alpha} (n_j^V(v_P) + b_3 \alpha^2 \| v_P  \|^2_{L^2(V_{j,\delta})}) \leq  b_4 \alpha e^{-2c \alpha\delta} \| u \|^2_{L^2(\rr^2)}.
    \]
    Collecting the above estimates concludes the proof of the lemma.
\end{proof}
Thanks to Proposition \ref{lem: upper bound side induced ev}, we conclude the proof of Theorem \ref{theo: side induced} by proving the following proposition. 
\begin{proposition}\label{lem: lower bound side induced}
 For any fixed $n \in \mathbb{N}$,  under the asymptotic regime \ref{eq: asymptotics} it holds that
    \[
    \Lambda_{\mathcal{K} + n} (H^\Gamma_\alpha) \geq - \frac{\alpha^2}{4} + E_n\big(\bigoplus_{j = 1}^M (D_j- \frac{k_j^2}{4}) \big) + \mathcal{O}\left(\frac{\log \alpha}{\sqrt{\alpha}}\right).
    \]
\end{proposition}
\begin{proof}
Let $n \in \mathbb{N}$. The main idea is to apply Proposition \ref{lem: comparing operators but fancy} with the following choices:
\begin{align*}
    \mathcal{H} &:= L^\perp \text{ in } L^2(\rr^2), \quad \mathcal{H}' := \bigoplus_{j =1}^M L^2(I_{j,\delta}),\\
    B&:= H^\Gamma_\alpha + \frac{\alpha^2}{4} + A_0\frac{\log \alpha}{\alpha}\quad\text{with} \quad D(b) = H^1(\rr^2) \cap \mathcal{H}, \\
    B' &= \bigoplus_{j =1}^M (D_{j,\delta} - \frac{k^2_j}{4} )\quad\text{with} \quad D(b') = \bigoplus_{j=1}^M H^1_0(I_{j,\delta}),
\end{align*} 
where $A_0 >0$ is a constant to be chosen later. We also construct a suitable linear map $J: D(B)\mapsto D(B')$ and $\varepsilon_1, \varepsilon_2 >0$ such that the following hold: 
\begin{enumerate}
    \item[(A)] $B \geq - \frac{k^2_{max}}{4}$ and $ \Lambda_n(B)=\mathcal{O}(1)$ as $\alpha\to \infty$, where $ k_{max} = \max ( \|k_1\|_\infty, \dots, \| k_M \|_\infty ) $;
    \item[(B)] $ \varepsilon_1 < 1/(1+ \Lambda_n(B) + \frac{k^2_{max}}{4}) $;
    \item[(C)] $\| u \|^2 - \| Ju \|^2 \leq \varepsilon_1 (b(u) + \| u \|^2 (1+ \frac{k^2_{max}}{4}))$ for all $u\in D(B)$;
    \item[(D)] $ b'(Ju)- b(u) \leq \varepsilon_2 (b(u) + \| u \|^2(1+ \frac{k^2_{max}}{4}))$ for all $u\in D(B)$.
\end{enumerate}
Then it follows that 
\[
E_n(B') \leq \Lambda_n(B) + \frac{(\varepsilon_1 \Lambda_n(B) + \varepsilon_2)(\Lambda_n(B) +1 + \frac{k^2_{max}}{4})}{1-\varepsilon_1(\Lambda_n(B)+1+\frac{k^2_{max}}{4})}.
\]
Ensuring $\varepsilon_1 , \varepsilon_2 = \mathcal{O}\big(\frac{\log \alpha}{\sqrt{\alpha}}\big)$ together with Lemma \ref{lem: make it Laplace - k^2/4 } yields
\[
\Lambda_{\mathcal{K}+n}(H^\Gamma_\alpha) = \Lambda_n(B)- \frac{\alpha^2}{4} + \mathcal{O}\left(\frac{\log \alpha}{\alpha}\right) \geq - \frac{\alpha^2}{4} +  E_n\big( \bigoplus_{j=1}^M (D_j - \frac{k_j^2}{4}) \big) + \mathcal{O}\left(\frac{\log \alpha}{\sqrt \alpha}\right).
\]
To construct $J$ and find suitable $\varepsilon_1$ and $\varepsilon_2$, for any $u \in \mathcal{H}$, denote the restriction:
\begin{align*} 
v_j &:= \text{restriction of } u \text{ onto } V_{j,\delta}, \quad \| v_j  \|:= \| v_j  \|_{L^2(V_{j,\delta})},\\
 w_j &:= \text{restriction of } u \text{ onto } W_{j,\delta}, \quad \| w_j  \|:= \| w_j  \|_{L^2(W_{j,\delta})} \\
 u_* &:= \text{restriction of } u \text{ onto } \Omega_\delta^*, \quad \| u_* \|:= \| u_* \|_{L^2(\Omega^*_\delta)}.
\end{align*}
We begin by verifying (A). Using the quadratic form of $B$,
\begin{align*}
b(u) = \int_{\rr^2} |\nabla u|^2 \mathrm{d}x - \alpha \int_{\Gamma} |u|^2 \dS + \frac{\alpha^2}{4} \int_{\rr^2} |u|^2 \mathrm{d}x + A_0\frac{\log \alpha}{\alpha} \| u \|^2_{L^2(\rr^2 )} ,
\end{align*}
we split this as 
\begin{align}\label{eq: endzeit b geq}
\begin{split}
b(u) \geq & \sum_{j=1}^M (n_j^V(v_j) + \frac{\alpha^2}{4}\| v_j  \|^2)  + \sum_{j=1}^M (n_j^W(w_j) + \frac{\alpha^2}{4} \| w_j  \|^2) + \frac{\alpha^2}{4} \| u_* \|^2 + A_0\frac{\log \alpha}{\alpha} \| u \|^2_{L^2(\rr^2)},
\end{split}
\end{align}
where $n_j^V$ and $n_j^W$ are the Neumann forms on $V_{j,\delta}$ and $W_{j,\delta}$ respectively. We first estimate the terms containing $v_j$.
Using Lemma \ref{lem: upper bound orthogonal complement in R2}, there exist $a_0, c >0$ such that suitable norm and trace estimates for $v_j$ hold. Under the asymptotic regime \eqref{eq: asymptotics} with a constant $C \geq 3/c$, these estimates simplify to
\begin{align*}
\| v_j  \|^2 &\leq a_0 \frac{\log^2 \alpha}{\alpha^2} (n_j^V(v_j) + \frac{\alpha^2}{4}\| v_j  \|^2 ) + a_0 \frac{\log^3 \alpha}{\alpha^3}\| u \|^2_{L^2(\rr^2)},\\
\int_{\partial_* V_{j,\delta}} | v_j|^2 \dS &\leq a_0 \frac{\log^2 \alpha}{\alpha} (n_j^V(v_j) + \frac{\alpha^2}{4}\| v_j  \|^2 ) + a_0 \frac{\log^3}{\alpha^2} \| u \|^2_{L^2(\rr^2)}.
\end{align*}
From these, it follows that for each $j \in \{1,\dots M \}$,
\begin{align}\label{eq: fin n_j^W + alpha^2/4 lower bound}
n_j^V(v_j) + \frac{\alpha^2}{4} \| v_j  \|^2 \geq \frac{1}{2a_0} \bigg(\frac{\alpha^2}{\log^2 \alpha} \| v_j  \|^2 + \frac{\alpha}{\log^2\alpha}\int_{\partial_* V_{j,\delta}} | v_j|^2 \dS  \bigg) - \frac{ \log \alpha}{\alpha} \| u \|^2_{L^2(\rr^2)}.
\end{align}
Next, the $w_j$ terms are estimated using Lemma \ref{lem: ev strip neuman}.There exist $a_N,\beta >0$ such that
\[
n_j^W(w_j) \geq (1-a_N \frac{\log \alpha}{\alpha}) \| P_ju'\|^2 + \int_{I_{j,\delta}} (-\frac{\alpha^2}{4} - \frac{k_j^2}{4}) | P_j u |^2 \dS - a_N\frac{1+C\log \alpha}{\alpha} \| P_j u \|^2,
\]
where $ P_j: \mathcal{H} \to L^2(I_{j,\delta})$ is defined by
\[
 (P_ju)(s) = \int_{-\delta}^\delta \overline{\psi(t)}w_j (\phi_j (s,t)) \sqrt{1- tk_j(s) }  \mathrm{ d}t,
 \]
 with $\psi \in H^1(-\delta,\delta) $ a normalized eigenfunction associated with the first eigenvalue of  $T^\beta_{\delta,\alpha}$, where $T^\beta_{\delta,\alpha}$ defined in Definition \ref{def t X}. Applying the Cauchy-Schwarz inequality and normalization of $\psi$ gives
\begin{align*}
    \| P_j u \|^2 &= \int_{I_{j,\delta}} \left| \int_{-\delta}^\delta \overline{\psi(t)}  
w_j(\phi_j(s,t)) \sqrt{1-tk_j(s)} \text{ d}t \right|^2 \mathrm{d}s \leq  \int_{W_{j,\delta}}| w_j |^2 \mathrm{d}x = \| w_j \|^2.
\end{align*}
Hence,
\begin{align} \label{eq: fin n_j^V + alpha^2/4 lower bound}
\begin{split}
n_j^W(w_j) + \frac{\alpha^2}{4} \geq & (1-a_N \frac{\log \alpha}{\alpha}) \| P_j u' \|^2 + \frac{\alpha^2}{4} (\| w_j  \|^2 - \| P_j  \|^2) \\
&- \int_{I_{j,\delta}} \frac{k_j^2}{4} | P_j u |^2 \dS - a_N \frac{(C + 1)\log \alpha}{\alpha} \| w_j  \|^2.
\end{split}
\end{align}
Plugging  \eqref{eq: fin n_j^V + alpha^2/4 lower bound} and \eqref{eq: fin n_j^W + alpha^2/4 lower bound} into \eqref{eq: endzeit b geq}, and choosing $A_0:= M + (C + 1)a_N $, we obtain
\begin{align}\label{eq: endzeit  b geq 2}
\begin{split}
b(u) \geq \frac{\alpha^2}{2 a_0\log^2 \alpha} \sum_{j = 1}^M  \| v_j  \|^2 + \frac{\alpha}{2 a_0 \log^2 \alpha} \sum_{j=1}^M \int_{\partial_* V_{j,\delta}} | v_j|^2 \dS + (1- a_N\frac{\log \alpha}{\alpha}) \sum_{j = 1}^M \| P_j u' \|^2 \\
+ \frac{\alpha^2}{4} \sum_{j = 1}^M(\| w_j  \|^2 - \| P_j u \|^2) - \sum_{j =1}^M \int_{I_{j,\delta}} \frac{k_j^2}{4} | P_j u |^2 \dS + \frac{\alpha^2}{4} \| u_* \|^2.
\end{split}
\end{align}
Since
\begin{align} \label{eq: endzeit lower bound k^2max /4} 
- \sum_{j =1}^M \int_{I_{j,\delta}} \frac{k_j^2}{4} | P_j u |^2 \dS \geq - \sum_{j=1}^M \frac{\| k_j \|^2_\infty}{4} \| P_j  \|^2 \geq -\frac{k^2_{max}}{4} \sum_{j=1}^M \| w_j  \|^2 \geq - \frac{k^2_{max}}{4} \| u \|^2_{L^2(\rr^2)},
\end{align}
we conclude that $B \geq -\frac{k^2_{max}}{4}$, establishing the first par of (A). Together with Proposition \ref{lem: upper bound side induced ev}, it follows that $\Lambda_n(B)=\mathcal{O}(1)$ as $\alpha \to \infty$. That is, we can choose a $\lambda_n \in \rr$ independent of $\alpha$ such that
\[
-\frac{k^2_{max}}{4} \leq \Lambda_n(B) \leq \lambda_n, \quad \frac{1}{1+\Lambda_n(B) + \frac{k^2_{max}}{4}} \geq \frac{1}{1+ \lambda_n + \frac{k_{max}^2}{4}},
\]
which completes the proof of (A). Note that this also implies (B) once we show that $\varepsilon_1 = \mathcal{O}(\frac{\log \alpha}{\sqrt{\alpha}})$.

Let us now construct the linear map $J$. Let $\chi_j^{0},\chi_j^{l} \in C^1([0,l_j])$ satisfy
\[
\chi_j^l(s) = \begin{cases}
    1, & s \text{ near  }0, \\0, & s \text{ near }l_j.
\end{cases} \quad , \quad \chi_j^0 = 1- \chi_j^l,
\]
and fix $\chi_0 >0$ such that
 \[
 \|\chi^{l/0}_j \|_\infty +\| (\chi^{l/0}_j)' \|_\infty \leq \chi_0, \quad \text{ for all } \, j = 1, \ldots,M.
\]
Recall $ I_{j,\delta} = ( \lambda_j^0(\delta) , l_j - \lambda^l_j(\delta) ) := (\tau_j^0, \tau_j^l)$ and $ \tau_j^0 = \mathcal{O}(\delta) $, $ \tau_j^l = l_j - \mathcal{O}(\delta)$, and evaluate at the endpoints accordingly:
 \[
 \chi_j^l(\tau_j^0) =1, \quad \chi_j^l(\tau_j^0)=0,\quad \chi_j^0(\tau_j^1) = 0 , \quad \chi^0_j(\tau_j^l)=1,
 \]
for large enough $\alpha$. Define the linear map $J: D(B) \rightarrow D(B')$, $Ju = (J_ju)$, where 
\[
(J_ju)(s) := (P_j u) (s) -(P_j u) (\tau_j^0)\chi_j^0(s)-(P_j u) (\tau_j^l)\chi_j^l(s).
\]
Before proving (B)-(D), let us first show some useful estimates. Using the Cauchy-Schwarz inequality and properties of $\psi_j$, we obtain
\begin{align}\label{eq: endzeit Pj leq 1}
\begin{split}
 | P_ju(\tau_j^0)|^2 + |P_j u( \tau_j^l)|^2 =& \left(\int_{-\delta}^\delta \overline{\psi_j} \sqrt{1- t k_j(s)} w_j(\phi_j(\tau_j^0 ,t)) \mathrm{d}t \right)^2\\
 &\qquad + \left(\int_{-\delta}^\delta \overline{\psi_j} \sqrt{1- t k_j(s)} w_j(\phi_j(\tau_j^l ,t)) \mathrm{d}t \right)^2
\\ 
\leq & \int_{-\delta}^\delta | w_j(\phi_j(\tau_j^0,t)) |^2 (1- tk_j(s)) \mathrm{d}t +\int_{-\delta}^\delta | w_j(\phi_j(\tau_j^l,t)) |^2 (1- tk_j(s)) \mathrm{d}t\\
 \leq & 2 \left(\int_{-\delta}^\delta |w_j(\phi_j(\tau_j^0,t))|^2 +\int_{-\delta}^\delta | w_j(\phi(\tau_j^l,t))|^2  \right) =2 \int_{\partial_* W_{j,\delta}}| w_j |^2 \dS.
 \end{split}
\end{align}
Using the inequality $ (x+y)^2 \geq (1-\varepsilon)x^2 +y^2/\varepsilon $ for some $\varepsilon >0$, it follows that
\begin{align*}
    \| J_ju\|^2 &= \int_{I_{j,\delta}} |P_ju (s) - (P_ju)(\tau_j^0) \chi_j^0(s)  -(P_ju)(\tau_j^l) \chi_j^l(s) |^2 \mathrm{d}s \\
    & \geq (1- \varepsilon) \int_{I_{j,\delta}} |P_ju|^2 \mathrm{d}s - \frac{1}{\varepsilon} \int_{I_{j,\delta}} |(P_ju)(\tau_j^0) \chi_j^0(s)  + (P_ju)(\tau_j^l) \chi_j^l(s) |^2 \mathrm{d}s.
\end{align*}
 Using  \ref{eq: endzeit Pj leq 1}  we obtain
\begin{align*}
    \int_{I_{j,\delta}} |(P_j u) (\tau_j^0)\chi_j^0(s)+(P_j u) (\tau_j^l)\chi_j^l(s)|^2 \mathrm{d}s &\leq 2 \int_{I_{j,\delta}} |(P_j u) (\tau_j^0)\chi_j^0(s)|^2 +|(P_j u) (\tau_j^l)\chi_j^l(s)|^2 \\
    & \leq 2 l_j \chi_0^2 (|P_ju(\tau_j^0)|^2 + |P_j u( \tau_k^l)|^2 ) \leq  4l_j \chi_0^2 \int_{\partial_* W_{j,\delta}}| w_j |^2 \mathrm{d}s .
\end{align*}
Thus, taking $ l = \max_j l_j$ yields
\[
\| J_ju \|^2 \geq (1-\varepsilon) \| P_j u \|^2 - \frac{4 l \chi_0^2}{\varepsilon} \int_{\partial_* W_{j,\delta}}| w_j |^2 \mathrm{d}s.
\]
We can now verify (C) and complete the proof of (B). Note that
\begin{align*}
    \| u \|^2 - \| Ju \|^2 &= \sum_{j =1}^M \| v_j  \|^2 + \sum_{j =1}^M \| w_j  \|^2 +\| u_* \|^2 - \sum_{j =1}^M \| J_ju \|^2  \\
    &\leq  \sum_{j=1}^M \| v_j  \|^2+ \sum_{j=1}^M \| w_j  \|^2 + \| u_* \|^2 - (1- \varepsilon) \sum_{j=1}^M \| P_j u \|^2 + \frac{4 l \chi_0}{\varepsilon} \sum_{j=1}^M \int_{\partial_* W_{j,\delta}}| w_j |^2 \mathrm{d}s \\
    &\leq \sum_{j = 1}^M \| v_j  \|^2 + \sum_{j=1}^M (\| w_j  \|^2 - \| P_j u \|^2) +  \| u_* \|^2 + \frac{4 l \chi_0}{\varepsilon} \sum_{j=1}^M \int_{\partial_* V_{j,\delta}} | v_j|^2 \mathrm{d}s + \varepsilon \| u \|^2.
\end{align*}
Combining this with \eqref{eq: endzeit  b geq 2} and \eqref{eq: endzeit lower bound k^2max /4}, we get
\[
\| u \|^2 - \| Ju \|^2 \leq \left(\frac{2a_0 \log^2 \alpha}{\alpha^2} + \frac{4}{\alpha^2} + \frac{8 l \chi_0^2 a_0 \log^2 \alpha}{\alpha \varepsilon} \right) \left(b(u)+ \frac{k^2_{max}}{4}\right) + \varepsilon\| u \|^2.
\]
Choosing $\varepsilon = \frac{ \log \alpha}{\sqrt \alpha}$ yields a constant $c_1 >0$ such that
\[
\| u \|^2 - \| Ju \|^2 \leq \frac{c_1 \log \alpha}{\sqrt \alpha} \left(b(u) + (1+\frac{k^2_{max}}{4}) \| u \|^2, \right) 
\]
which proves (B) and (C) with $\varepsilon_1:=c_1\varepsilon= \mathcal{O}(\frac{\log \alpha}{\sqrt{\alpha}})$.

We are now going to verify (D). Let us estimate $b'(Ju)$. Using the inequality $(x+y)^2 \leq (1+ \varepsilon)x^2 + 2\frac{y^2}{\varepsilon} $ for any $x,y \in \rr$ and $\varepsilon \in (0,1)$:
\begin{align*}
     \|(J_ju)'\|^2 &= \int_{I_{j,\delta}} | (P_ju)'(s) - P_ju(\tau_j^0) (\chi_j^0)'(s) - P_j u( \tau_k^l) (\chi_j^l)'(s) |^2 \mathrm{d}s \\
     &\leq (1+\varepsilon) \int_{I_{j,\delta}} |(P_j u)'(s)|^2 \mathrm{d}s + \frac{2}{\varepsilon }  \int_{I_{j,\delta}} | P_ju(\tau_j^0) (\chi_j^0)'(s) +P_j u( \tau_k^l) (\chi_j^l)'(s)|^2 \mathrm{d}s
\end{align*}
By the same arguments as before, we have
\begin{align*}
    \int_{I_{j,\delta}} |P_ju(\tau_j^0) (\chi_j^0)'(s) + P_j u( \tau_k^l) (\chi_j^l)'(s) |^2 \mathrm{d}s \leq 2 \int_{I_{j,\delta}} |P_ju(\tau_j^0) (\chi_j^0)'(s)|^2 + | P_j u( \tau_k^l) (\chi_j^l)'(s)|^2 \mathrm{d}s \\
     \leq 2l_j \chi_0 2 \left(|P_j u(\tau_j^0) |^2 +| P_j u(\tau_j^l) |^2\right) \leq 4l \chi_0^2 \int_{\partial_* W_{j,\delta}}| w_j |^2 \mathrm{d}s.
\end{align*}
To estimate the remaining part of $b'(Ju)$, we use the identity $(x+y)^2 \geq (1-\varepsilon)x^2-\frac{1}{\varepsilon} y^2 $ for $x,y \in \rr$ and $ \varepsilon > 0$,
\begin{align*}
     \int_{I_{j,\delta}} \frac{k^2_j(s)}{4} | J_j u |^2 \mathrm{d}s &\geq (1-\varepsilon)\int_{I_{j,\delta}} \frac{k^2_j(s)}{4} | P_j u |^2 \mathrm{d}s - \frac{1}{\varepsilon} \int_{I_{j,\delta}} \frac{k^2_j(s)}{4} | P_ju(\tau_j^0) (\chi_j^0)'(s) + P_j u( \tau_k^l) (\chi_j^l)'(s)|^2 \mathrm{d}s \\
   & \geq (1-\varepsilon)\int_{I_{j,\delta}} \frac{k^2_j(s)}{4} | P_j u |^2 \mathrm{d}s - \frac{\| k_j\|^2_\infty}{4\varepsilon} \int_{ I_{j,\delta}} |P_ju(\tau_j^0) (\chi_j^0)'(s) + P_j u( \tau_k^l) (\chi_j^l)'(s) |^2 \mathrm{d}s \\
   &  \geq (1-\varepsilon)\int_{I_{j,\delta}} \frac{k^2_j(s)}{4} | P_j u |^2 \mathrm{d}s - \frac{ l \chi_0^2 k^2_{max}}{\varepsilon} \int_{ \partial_* W_{j,\delta}}| w_j |^2 \dS.
\end{align*}
In summary, we have the following inequalities for $B$ and $B'$:
\begin{align*}
&b'(Ju) \leq (1+\varepsilon) \sum_{j=1}^M \| (P_ju)'\|^2 + \frac{(8 + k^2_{max})l \chi_0^2 }{\varepsilon} \sum_{j=1}^M \int_{\partial_* W_{j,\delta}}| w_j |^2 \dS - (1-\varepsilon) \sum_{j=1}^M \int_{I_{j,\delta}} \frac{k_j(s)}{4} | P_j u |^2 \mathrm{d}s;\\
 &b(u) + \frac{k_{max}^2}{4} \| u \|^2_{L^2(\rr^2)} \geq \big(1- a_N\frac{\log \alpha}{\alpha}\big) \sum_{j =1 }^M \| P_ju'\|^2 +  \frac{\alpha}{2a_0\log^2 \alpha} \sum_{j=1}^M \int_{W_{j,\delta}}| w_j |^2\dS;\\
 &b(u) \geq \big(1- a_N\frac{\log \alpha}{\alpha}\big) \sum_{j =1 }^M \| P_ju'\|^2 + \frac{\alpha}{2a_0\log^2 \alpha} \sum_{j=1}^M \int_{W_{j,\delta}}| w_j |^2\dS
- \sum_{j=1}^M \int_{I_{j,\delta}} \frac{k_j^2(s)}{4} | P_j u |^2\mathrm{d}s.
\end{align*}
Therefore,
\begin{align*}
    b'(Ju) &- b(u) \leq  (1+\varepsilon) \sum_{j=1}^M \|(P_ju)'\|^2 + \frac{(8 + k^2_{max})l \chi_0^2 }{\varepsilon} \sum_{j=1}^M \int_{\partial_* W_{j,\delta}} |w_j|^2 \dS   + \sum_{j=1}^M \int_{I_{j,\delta}} \frac{k_j^2(s)}{4} |P_j u|^2\mathrm{d}s \\
     &\qquad-(1-\varepsilon) \sum_{j=1}^M \int_{I_{j,\delta}} \frac{k^2_j(s)}{4} |P_j u|^2 \mathrm{d}s -\left(1- a_N\frac{\log \alpha}{\alpha}\right) \sum_{j =1 }^M \|P_ju'\|^2 \\
    \leq &\left(\varepsilon + a_N \frac{\log \alpha}{\alpha}\right) \sum_{j=1}^M \|P_ju'\|^2 + \frac{(8 + k^2_{max})l \chi_0^2 }{\varepsilon} \sum_{j=1}^M \int_{\partial_* W_{j,\delta}} |w_j|^2 \dS + \varepsilon \frac{k^2_{max}}{4} \| u \|^2_{L^2(\rr^2)} \\
    \leq& \left( \frac{\varepsilon + a_N \frac{\log \alpha}{\alpha}}{1- a_N \frac{\log \alpha}{\alpha}} + \frac{(8+k_{max}^2)l \chi_0^2}{\varepsilon} \cdot \frac{2 a_0 \log^2 \alpha}{\alpha} \right) \left(b(u) + \frac{k^2_{max}}{4} \| u \|^2_{L^2(\rr^2)} \right)+ \varepsilon \frac{k_{max}^2}{4} \| u \|^2_{L^2(\rr^2)}. 
\end{align*}
Choosing $\varepsilon = \frac{\log \alpha}{\sqrt \alpha}$, there exists $c_2 > 0$ such that
 \[
 b'(Ju) - b(u) \leq \frac{c_2 \log \alpha}{\sqrt{\alpha}} \left(b(u) + (1+ \frac{k_{max}^2}{4}) \| u \|^2_{L 2(\rr^2)}\right) ,
\]
which proves (D) with $\varepsilon_2 := c_2 \frac{\log \alpha}{\sqrt \alpha}=\mathcal{O}\big( \frac{\log \alpha}{\sqrt \alpha}\big)$.
\end{proof}

\section*{Funding}
The authors disclosed receipt of the following financial support for the research, authorship, and/or publication of this article: All authors were partially supported by the Deutsche For\-schungsgemein\-schaft (German Research Foundation, DFG), project 491606144.
\section*{Conflicting Interests}
The authors declared no potential conflicts of interest with respect to the research, authorship, and/or publication of this article.

%=================================================================================================

\end{document}